%% file: origami.tex
\documentclass[11pt]{amsart}

\usepackage{amssymb,amsfonts,psfrag,amscd,enumerate,epsfig,latexsym,amsmath}
\usepackage{pinlabel}

\input{defs_origami}

\usepackage{hyperref}

\author{Andr\'e de Carvalho} \address{Departamento de Matem\'atica
  Aplicada, IME - USP \\ Rua do Mat\~ao 1010, Cidade Universit\'aria
  \\ 05508-090 S\~ao Paulo, SP, Brazil}\email {andre@ime.usp.br}
\author{Toby Hall} \address{Department of Mathematical
  Sciences\\University of Liverpool\\Liverpool L69 7ZL, UK} \email
       {T.Hall@liv.ac.uk}

\title{Paper folding, Riemann surfaces, and convergence of
  pseudo-Anosov sequences}

\begin{document}

\begin{abstract}
A method is presented for constructing closed surfaces out of
Euclidean polygons with infinitely many segment identifications along
the boundary. The metric on the quotient is identified.  A sufficient
condition is presented which guarantees that the Euclidean structure on
the polygons induces a unique conformal structure on the quotient
surface, making it into a closed Riemann surface. In this case, a
modulus of continuity for uniformizing coordinates is found which
depends only on the geometry of the polygons and on the
identifications. An application is presented in which a uniform
modulus of continuity is obtained for a family of pseudo-Anosov
homeomorphisms, making it possible to prove that they converge to a
Teichm\"uller mapping on the Riemann sphere.
\end{abstract}

\maketitle

\bibliographystyle{amsplain}

\section{Introduction}

This article addresses the classical problem of constructing surfaces
out of subsets of the plane by making identifications along the
boundary: in contrast to the usual discussion arising in the
classification of surfaces, however, infinitely many identifications
are allowed. The topological structure of the identification space $S$
is studied and conditions are given which guarantee that $S$ is a
closed surface. The quotient metric on $S$ induced by the Euclidean
metric on the plane is identified, and the question of whether or not
this metric induces a unique complex structure on~$S$ is discussed. A
sufficient condition for uniqueness of the complex structure is given
and, when it holds, a modulus of continuity for uniformizing
coordinates is obtained. The interplay between the metric and
conformal structures is central to the paper, promoting the
topological stucture to a Riemann surface structure and providing
quantitative control over the quotient map. This analytic control is
then used to prove convergence of a certain sequence of pseudo-Anosov
homeomorphisms to a {\em generalized} pseudo-Anosov homeomorphism.

\medskip
Let $P$ be a finite collection of disjoint polygons in the (complex)
plane. A {\em paper-folding scheme} is an equivalence relation which
glues together segments --- possibly infinitely many --- along the
boundary of $P$. The image of $\partial P$ in the quotient space $S$
is called the {\em scar}: it contains {\em cone points}, where the
total angle is not equal to~$2\pi$, and {\em singular points}, such as
accumulations of cone points. The following statements summarize the
main theorems of this article.

\bigskip
\noindent\textbf{Topological Structure Theorems
  (Theorems~\ref{thm:plain-top-structure}
  and~\ref{thm:top-structure})} \,\, \emph{Necessary and sufficient
  conditions are given for the quotient space $S$ of a paper-folding
  scheme to be a closed surface. In particular, if the paper-folding
  scheme is plain ($P$ is a single polygon and the identifications
  along its boundary are unlinked) then $S$ is a topological sphere
  and the scar is a dendrite.}
\bigskip

\noindent\textbf{Metric Structure Theorem
  (Theorem~\ref{thm:metric-structure})} \,\, \emph{The quotient metric
  on~$S$ is intrinsic, and~$S$ is a conic-flat surface (that is, it is
  locally isometric to cones on circles) away from singular points.}

\bigskip

\noindent\textbf{Conformal Structure Theorem (Theorem~\ref{thm:main})}
\,\, \emph{The natural conformal structure on the conic-flat part of
  $S$ extends uniquely across an isolated singularity provided that a
  certain integral diverges. In particular, if there are only finitely
  many singular points at each of which the relevant integral
  diverges, then the conformal structure extends uniquely across the
  singular set, making $S$ into a closed Riemann surface.}

\bigskip

It is possible to extend this theorem to the case of arbitrary
singular sets, and this will be the subject of a forthcoming paper.

\smallskip
If the conditions of both Theorem~\ref{thm:top-structure} and
Theorem~\ref{thm:main} are satisfied, then~$S$ has a natural closed
Riemann surface structure. If $S$ is a topological sphere, as is the
case for plain paper foldings, then $S$ is isomorphic to the Riemann
sphere. A modulus of continuity for a
suitably normalized uniformizing map from the polygon to the Riemann
sphere is found and the following theorem is proved:

\bigskip

\noindent\textbf{Modulus of continuity of uniformizing map
  (Theorem~\ref{thm:modcont})} \,\, \emph{The uniformizing map has a
  modulus of continuity which depends only on the geometry of the
  polygon and on the metric on the scar.}

\bigskip

The theorems above belong
to the fields of surface topology, Riemann surface theory, and
geometric function theory. The question that motivated this work,
however, comes from dynamical systems theory and was described to the
first author many years ago by Dennis Sullivan. Consider a
self-homeomorphism of a surface, such as the much-studied H\'enon
diffeomorphism of the plane
\[ f(x,y):=(a-x^2-by,x),\] 
where $a,b$ are real parameters. There is a large set of parameters
for which $f$ is {\em chaotic} which means, amongst other things, that
it has infinitely many periodic orbits. If $Q$ is a finite
$f$-invariant set (a union of periodic orbits), consider the isotopy
class of~$f$ in the punctured plane $\R^2\setminus Q$.  Thurston's
classification of surface homeomomorphisms provides a canonical
representative $\vphi_Q$ in this isotopy class, which is, typically, a
{\em pseudo-Anosov} homeomorphism. If $Q_n$ is an increasing sequence
which exhausts the collection of finite $f$-invariant sets, one can
imagine passing to the limit in the sequence $\vphi_{Q_n}$ of
associated pseudo-Anosov maps. The resulting limit would be a {\em
  tight} representative of $f$ which contains `all' of its dynamics.
The paper ends with an application of the theorems above to establish
a case of this scenario:

\bigskip

\noindent\textbf{Convergence to the tight horseshoe
  (Theorem~\ref{thm:convergence})}
\,\, \emph{There exists a cofinal sequence in the 
  set of horseshoe braid types partially ordered by forcing whose
  associated pseudo-Anosov homeomorphisms converge to the tight
  horseshoe.}

\bigskip 

To understand the existence of limits of sequences of pseudo-Anosov
maps on a fixed surface $S$ relative to varying finite subsets $Q$ has
several consequences. From a purely intrinsic point of
view, it is interesting to be able to describe the closure of the set of
relative pseudo-Anosov maps on $S$ --- an important class of
homeomorphisms --- in the space of all self-homeomorphisms of $S$.
Also, Thurston's classification of surface homeomorphims up to isotopy
produces geometrically and conformally rigid structures out of the
topology. Understanding limits of pseudo-Anosovs would do the same in
a more general context and, in this case, the dynamics and not the
topology would be the driving force.  It is hoped that this will
eventually lead to the correct 2-dimensional analogue of the
Milnor-Thurston theorem stating that every multimodal interval map has
a piecewise-linear quotient with same entropy. There are also
connections to be explored with infinite dimensional Teichm\"uller
theory, since the tight limits are Teichm\"uller mappings with respect
to quadratic differentials in the $L^1$-closure of the set of
meromorphic quadratic differentials.

\medskip
Section~\ref{sec:background} provides a brief summary of results from
surface topology, metric geometry, and geometric function theory which
will be used in the remainder of the paper. In
Section~\ref{sec:surfaceorigami} the main objects of study,
paper-folding schemes, are defined. These are identification schemes
around the boundaries of finite disjoint unions of polygons, whose
(metric) quotients are called {\em paper spaces}.

Sections~\ref{sec:topol-struct-plain} and~\ref{sec:cxstructure}
contain discussions of the topological, metric, and conformal
structures of paper spaces. To simplify the exposition, some of the
arguments are first developed in the simplest and most common case of
plain paper foldings, before being extended to the general case.

In Section~\ref{sec:modcont}, moduli of continuity are discussed,
first locally for uniformizing coordinates charts on the neighborhood
of points on $S$ at which the conformal structure extends uniquely and
then globally in the case where $S$ is a complex sphere and is thus
isomorphic to the Riemann sphere.

In Section~\ref{sec:dynam-appl}, the dynamical application is given. A
sequence of pseudo-Anosov maps of the punctured sphere is considered:
each can be regarded as a quasi-conformal homeomorphism of an
appropriate (finite) paper surface. Using the results of
Section~\ref{sec:modcont}, the natural limit of this sequence is
constructed. In fact this is equivalent to taking the limit of
pseudo-Anosov maps over an exhaustion of the collection of finite
invariant sets of Smale's horseshoe map.

\smallskip
A brief account of the results described in this article, their
application, and an outline of the methods of proof can be found
in~\cite{PNAS}.

\bigskip\noindent {\bf Acknowledgements:} The authors would like to
thank Lasse Rempe for his help with the proof of
Lemma~\ref{lem:convergence-phi-n}. They gratefully acknowledge the
support of FAPESP
% (Funda\c c\~ao de Amparo \`a Pesquisa do Estado de
%S\~ao Paulo) 
grant 2006/03829-2. The first author would also like to
acknowledge the support of CNPq 
%(Conselho Nacional de Desenvolvimento
%Cient\'ifico e Tecnol\'ogico) 
grant 151449/2008-2 and the hospitality
of IMPA,
%(Instituto Nacional de Matem\'atica Pura e Aplicada) 
where part of this work was developed.

\section{Background}
\label{sec:background}

This section contains a summary of background theory which will be
used in the article, together with references to more detailed
accounts. Section~\ref{subsec:top-background} contains topological
results which will be used in Section~\ref{sec:topol-struct-plain} to
give conditions under which the quotient space associated with a
paper-folding scheme is a topological sphere or, more generally, a
surface. The main tool here is Moore's theorem~\cite{MoorePaper,Moo},
which provides conditions under which the quotient of a sphere by an
equivalence relation is again a
sphere. Section~\ref{subsec:metric-background} contains a brief
summary of the theory of quotient metric spaces and their
relationship to topological quotients. Finally,
Section~\ref{subsec:gft-background} describes some results from
geometric function theory which will be used in
Section~\ref{sec:cxstructure} to give conditions under which the
conformal structure on the quotient space extends uniquely across an
isolated singularity; and in Section~\ref{sec:modcont} to determine a
modulus of continuity for a uniformizing map from the polygon to the
Riemann sphere in the case of a plain paper folding.

\subsection{Moore's theorem and the dendrite quotient theorem}
\label{subsec:top-background}

\begin{defns}[Separation, continuum]
\label{defns:continuum}
A {\em separation} of a topological space~$X$ is a decomposition
of~$X$ as a disjoint union~$X=A\cup B$ of non-empty closed subsets
of~$X$. The space~$X$ is {\em connected} if no separation of~$X$
exists.

A set $C\sbs X$ {\em separates} two points $x,y\in X$
(respectively subsets $D,E\sbs X$) if there is a separation
$X\setminus C=A\cup B$ with $x\in A$, $y\in B$ (respectively
$D\sbs A$, $E\sbs B$). If $X$ is connected and $X\setminus
C$ is not, $C$ {\em separates} $X$.

A {\em continuum} is a compact connected (subset of a) Hausdorff
topological space.
\end{defns}

\begin{defns}[Monotone upper semi-continuous decomposition]
\label{defn:musc}
An equivalence relation $\sim$ on a topological space $X$ is {\em
  closed} if it is closed as a subset of $X\times X$. (On compact
Hausdorff spaces, being closed is the same as saying that $x_n\sim
y_n$, $x_n\to x$ and $y_n\to y$ imply $x\sim y$; and an equivalence
relation is closed if and only if the quotient space is Hausdorff.)

A {\em decomposition} of $X$ is synonymous with a partition of $X$,
i.e., a collection of disjoint subsets whose union is~$X$.  A
decomposition $\cG$ of a topological space into compact subsets is
{\em upper semi-continuous~ (usc)} if the associated equivalence
relation is closed. A decomposition is {\em monotone} if its elements
are connected.

The expression ``monotone upper semi-continuous'' will be abbreviated
{\em musc}.
\end{defns}

\begin{defn}[Realizing an equivalence relation]
\label{defn:realize}
Let $P$ be a subset of a topological space~$X$, and $\sim$ be an
equivalence relation on $P$. A decomposition $\cG$ of $X$ {\em
  realizes}~$\sim$ if all elements of $\cG$ intersect $P$, and each
intersection of an element of $\cG$ with $P$ is a $\sim$-equivalence
class. (This means that $P/{\sim}$ is naturally identified with
$X/\cG$.)
\end{defn}

\begin{defn}[Saturation]
\label{defn:saturation}
Let~$\rmR$ be a reflexive and symmetric relation on a set~$X$. A subset~$U$
of~$X$ is {\em $\rmR$-saturated} if it contains $\{y\in X\,:\,y\rmR x\}$ for
all~$x\in U$.

Any collection~$\cP$ of subsets of~$X$ (not necessarily a partition)
naturally defines a symmetric and reflexive relation~$\rmR$ (two
distinct points are related if they belong to the same element
of~$\cP$), which is transitive if the elements of~$\cP$ are
disjoint. A subset~$U$ of~$X$ is {\em $\cP$-saturated} if it is
$\rmR$-saturated.
\end{defn}

\medskip
The following theorem, due to Moore~\cite{MoorePaper},
is the main tool used for deciding when the quotient of an equivalence
relation on the 2-sphere is again the sphere. A generalization for
surfaces was given by Roberts and Steenrod~\cite{RoSt}. 

\begin{thm}[Moore on the 2-sphere] 
\label{thm:Moore}
A topological quotient of the 2-sphere by a musc decomposition whose
elements do not separate it is again homeomorphic to the 2-sphere.
\end{thm}

The topological structure of the scar of a paper-folding scheme will
also be important, and the main tool which will be used in this
regard is Theorem~\ref{thm:dendritequot} below.

\begin{defns}[Dendrite, dendritic collection]\label{defn:dendrite}
A {\em dendrite} is a locally connected continuum which does not
contain any simple closed curve. A {\em local dendrite} is a continuum
which is locally a dendrite, i.e., for which every point has a
(closed) neighborhood which is a dendrite.

A collection $\cG$ of disjoint subsets of a topological space $X$ is
{\em non-separated} if no element of $\cG$ separates two points of any
other single element of $\cG$. A non-separated collection $\cG$ is
{\em dendritic} if given $g\in\cG$ and $y\not\in g$, there exists
$g'\in\cG$ which separates $y$ and $g$.
\end{defns}

The following theorem can be found in~\cite{Whyburn}. There the word
{\em saturated} is used to signify what was defined as {\em dendritic}
above (the former word has already been used here).

\begin{thm}[Dendrite quotient]\label{thm:dendritequot}
Let $\cG$ be a dendritic decomposition of a continuum $X$. Then $\cG$
is usc and the quotient space $X/\cG$ is a dendrite.
\end{thm}

Here is a list of properties of dendrites which will be used
later. The main references for them
are~\cite{KuraII,Whyburn}.
\begin{thm}[Properties of dendrites] \mbox{}
\label{thm:pprtsdendrites}
\begin{enumerate}[a)] 
\item Any two distinct points in a dendrite $G$ are separated by a
  third point in $G$. Conversely, any continuum with this property is a
  dendrite. 
\item If $G$ is a dendrite and $x,y\in G$, there exists a unique
 arc $\gamma\sbs G$ whose endpoints are $x,y$. The notation
  $[x,y]_G$ will be used to indicate this arc and expressions such as
  $(x,y]_G$ will have the usual meaning (i.e.,
    $(x,y]_G=[x,y]_G\setminus\{x\}$).
\item A point $p$ in a continuum $G$ is an {\em endpoint} if it has
  arbitrarily small neighborhoods (in~$G$) whose boundary is a single
  point; and it is a {\em cut point} if $G\setminus\{p\}$ is
  disconnected. A continuum is a dendrite if and only if all of its
  points are either endpoints or cut points.
\item Dendrites are contractible~\cite{Fort}. 
\item Every subcontinuum of a dendrite is also a dendrite. (A {\em
  subcontinuum} of a continuum $G$ is a subset which is also a
  continuum.) 
\item Every connected subset of a dendrite is arcwise connected, and
  the intersection of any two connected subsets is connected.
\end{enumerate}
\end{thm}

\subsection{Quotient metric spaces and intrinsic metrics}
\label{subsec:metric-background}

The quotient spaces associated with paper-folding schemes are metric
rather than topological quotients, and this section summarizes some
relevant definitions and results. \cite{BBI} is an excellent reference
for readers seeking further details.

\begin{defns}[Metric and semi-metric]
A {\em metric} on a set $X$ is a function $d_X\co X\times X\to
\R\cup\{+\infty\}$ satisfying the usual conditions of positivity,
symmetry and the triangle inequality (it is convenient to allow two points
to be infinitely distant from one another). When no ambiguity arises,
subscripts may be dropped so that, for example, the usual Euclidean
distance on $\R^2$ or $\C$ may be denoted $d_{\R^2}(x,y)$,
$d_{\C}(x,y)$, $d(x,y)$, or even $|xy|$.

A {\em semi-metric} on a set $X$ satisfies the same axioms as a
metric, except that distinct points are allowed to be at semi-distance
zero from one another.
\end{defns}

\begin{notation}
\label{notation:ballsandcircles}
The following notation is used for metric and semi-metric
spaces~$(X,d_X)$. If $x\in X$ and $r\ge 0$, then
\begin{eqnarray}
B_{X}(x;r)&:=&\{y\in X:\,\, d_X(y,x)<r\},\nonumber \\ 
 \bB_{X}(x;r)&:=&\{y\in X:\,\, d_X(y,x)\leq r\},
 \qquad\text{and} \label{eq:ballsandcircles}\\
 C_{X}(x;r)&:=&\{y\in X:\,\, d_X(y,x)=r\}. \nonumber
\end{eqnarray}
If ambiguity lurks, subscripts may be enhanced so that $B_{X}(x;r)$,
for example, may be denoted $B_{(X,d_X)}(x;r)$ or $
B_{d_X}(x;r)$. At the other extreme, if the ambient space is clear,
$B(x;r)$ may be used.

If $x\in X$ and $A\sbs X$, the distance from $x$ to $A$ is 
\[d_X(x,A):=\inf\{d_X(x,y):\;y\in A\}.\]
and if $A,B\sbs X$, the distance between them is 
\[d_X(A,B):=\inf\{d_X(x,y):\;x\in A\,,y\in B\}.\]
Given a set $A\sbs X$, the sets $B_X(A;r)$, $\bB_{X}(A;r)$ and
$C_{X}(A;r)$ are defined by substituting $d(y,A)$ for
$d(y,x)$ in~(\ref{eq:ballsandcircles}).
\end{notation}

\begin{defns}[Quotient metric]\label{defn:quotmetric}
Let $(X,d)$ be a metric space and $\rmR$ be a reflexive and symmetric
relation on~$X$ (which in this article will usually arise from a
collection~$\cP$ of subsets of~$X$, as in
Definition~\ref{defn:saturation}). An $\rmR$-{\em chain} from $x$ to
$y$ is a sequence $((p_i,q_i))_{i=0}^k$ in~$X^2$ such that $x \rmR
p_0$, $q_i\rmR p_{i+1}$ for $i=0,\ldots,k-1$, and $q_k\rmR y$. Its
{\em length} is
\[\rmL^\rmR\left((p_i,q_i)\right):=\sum_{i=0}^k d(p_i,q_i)\] 
(thus one pays to move between unrelated elements of~$X$, but moving
between related elements is free).

Define $d^\rmR\co X\times X\to\R_{\geq 0}\cup\{\infty\}$ by
\[ d^\rmR(x,y):=\inf\left\lbrace \rmL^\rmR((p_i,q_i)):\, ((p_i,q_i)) \text{
  is an $\rmR$-chain from $x$ to $y$}\right\rbrace.\] Then $d^\rmR$ is
a semi-metric on~$X$. The equivalence relation which identifies points
at $d^\rmR$-semi-distance 0 is denoted~$\sim_\rmR$ and the quotient
space under this equivalence relation is the {\em quotient metric
  space} of~$(X,d)$ under the relation~$\rmR$, denoted either
$(X/{d^\rmR},d^\rmR)$ or $(X/{\sim_\rmR},d^\rmR)$. (So the same symbol
$d^\rmR$ is used to denote both the semi-metric and the metric on the
quotient space.)
\end{defns}

Many mathematicians, including the authors, are more familiar with
topological quotients. Assuming that~$\rmR$ is an equivalence
relation, there are essentially two reasons why the quotient metric
space $(X/{d^\rmR},d^\rmR)$ and the topological quotient $X/\rmR$ may
differ. First, they may differ as sets: for example, the metric
quotient of~$\R$ by~$\Q$ is a point (as is the metric quotient of any
metric space by a relation~$\rmR$ which identifies a dense set of
points), since there are arbitrarily short $\rmR$-chains joining any
two points. Second, even if the two quotients agree as sets, they may
have different topologies. An instructive example (Example~3.1.17
of~\cite{BBI}) is the following. Let~$X$ be a disjoint union of
countably many intervals $I_i$, of lengths~$\ell_i$, and consider the
quotients by the relation which identifies the left hand endpoints of
all of the intervals. The topology of the metric quotient depends on
the lengths~$\ell_i$: in particular, it is compact if $\ell_i\to 0$,
and non-compact otherwise. Clearly, though, the topological quotient
does not depend on these lengths. 

The case where $\ell_i\to 0$ will be of interest later:
\begin{defn}[$\infty$-od]
\label{defn:infty-od}
An {\em $\infty$-od} is the metric quotient of countably many
intervals $I_i$ of lengths~$\ell_i$, with $\ell_i\to 0$ as
$i\to\infty$, under the relation which identifies all of their left
hand endpoints.
\end{defn}

The problem of having quotients which agree as sets but not as
topological spaces does not arise if~$X$ is compact (see~\cite{BBI}):
\begin{thm}[Metric and topological quotients]
\label{thm:compact-metr-top-agree}
Let~$X$ be a compact metric space, and $\rmR$ be a reflexive and
symmetric relation on~$X$. Then the topological quotient
$X/{\sim_\rmR}$ and the metric quotient $(X/{\sim_\rmR},d^\rmR)$ are
homeomorphic.
\end{thm}
(Notice that in this statement the two quotients are equal as sets by
definition, since the same equivalence relation $\sim_\rmR$ is used in
both cases.) 

The quotient associated to a paper-folding scheme on a
polygon~$P$ is constructed by taking the metric quotient of $P$ by a
relation~$\rmR$ determined by the folding scheme. Since~$P$ is
compact, this metric quotient is homeomorphic to the topological
quotient $P/{\sim_\rmR}$, which will be studied using Moore's theorem.

\medskip\medskip

In this article subsets of metric spaces will normally be endowed not
with the subspace metric, but with the intrinsic metric inherited from
the parent space. The remainder of this section provides a brief
summary of these ideas.

\begin{defns}[Length and intrinsic metric]
\label{defns:lengths-intrinsic-metrics}
Let $(X,d)$ be a metric space. A {\em path} in X is a continuous map
$\gamma\co [a,b]\to X$. The {\em length} of the path $\gamma$ is
\[ |\gamma|_X= \sup\left\lbrace \sum 
d\left(\gamma(t_i),\gamma(t_{i+1})\right) \right\rbrace \in \R_{\ge0}
\cup\{\infty\},\] where the supremum is taken over all finite
partitions $a=t_0<t_1<\cdots<t_k=b$ of the interval $[a,b]$. When the
metric space~$X$ is clear from the context, the length will be
denoted~$|\gamma|$ without the subscript. If $\gamma$ is injective,
then the length of the image of~$\gamma$ is defined to be
$|\im(\gamma)|_X := |\gamma|_X$.

A path is {\em rectifiable} if its length is finite.

A metric is {\em intrinsic} if the distance between any two points is
arbitrarily well approximated by lengths of curves joining the two
points, that is,
\[ d(x,y)=\inf\left\lbrace |\gamma|:\ 
\gamma\co [a,b]\to X \text{ is a path with } \gamma(a)=x,\,\,
\gamma(b)=y\right\rbrace.\] The metric is {\em strictly intrinsic} if
the infimum is attained, that is, if for every $x,y\in X$, there
exists a continuous path from $x$ to $y$ whose length equals
$d(x,y)$. (In particular, $d(x,y)=+\infty$ if $x$ and $y$ lie in
different path components of~$X$.) It can be shown (see~\cite{BBI})
that a complete locally compact intrinsic metric is strictly
intrinsic. In particular,
\begin{thm}
\label{thm:compact-intrinsic-strict}
A compact intrinsic metric is strictly intrinsic.
\end{thm}

If~$d$ is not an intrinsic metric, then there is an {\em induced
  intrinsic metric} $\widehat{d}$ on~$X$, defined by
\[ \widehat{d}(x,y)=\inf\left\lbrace |\gamma|:\ 
\gamma\co [a,b]\to X \text{ is a path with } \gamma(a)=x,
\gamma(b)=y\right\rbrace\] (that~$\widehat{d}$ is an intrinsic metric
follows from the straightforward observation that the
length~$|\gamma|$ doesn't depend on whether the metric $d$ or
$\widehat{d}$ is used).  For example, if~$P\sbs\R^2$ is a
polygon, then there is an intrinsic metric $d_P$ on~$P$ induced by the
metric on~$\R^2$, which does not agree with the subspace metric on~$P$
unless~$P$ is convex. 
\end{defns}

A proof of the following result can be found on pp. 62\,--\,63 of~\cite{BBI}.
\begin{lem}
\label{lem:intrinsic-quotient-intrinsic}
Let~$(X,d)$ be an intrinsic metric space, and~$\rmR$ be a reflexive
and symmetric relation on~$X$. Then the quotient metric space
$(X/d^{\rmR}, d^{\rmR})$ is also intrinsic.
\end{lem}

\begin{rmk}
\label{rmk:distance-decreasing}
The projection~$\pi$ from a metric space~$X$ to a metric
quotient~$X/d^{\rmR}$ does not increase the distances between points
or the lengths of paths. That is, if $x,y\in X$ then
$d^\rmR(\pi(x),\pi(y))\le d(x,y)$; and if $\gamma$ is a path in~$X$, then
$|\pi\circ\gamma|_{X/d^\rmR} \le |\gamma|_X$.
\end{rmk}

The following three results about Hausdorff 1-dimensional measure
$\mu^1_X$ on a metric space~$X$ are Lemma~2.6.1 of~\cite{BBI},
Theorem~2.6.2 of~\cite{BBI}, and a corollary of Theorem~29
of~\cite{Rogers} respectively.

\begin{lem}
\label{lem:hausdorff-connected}
If~$X$ is a connected metric space, then $\mu^1_X(X) \ge \diam X$.
\end{lem}

\begin{lem}
\label{lem:hausdorff-arc}
If $\gamma\co[a,b]\to X$ is a rectifiable simple path, then $|\gamma|_X =
\mu^1_X(\gamma([a,b]))$. 
\end{lem}

\begin{lem}
\label{lem:hausdorff-nonincreasing}
Let~$(X,\rho)$ and~$(Y,\sigma)$ be metric spaces, and $f\co X\to Y$ be
a function with $\sigma(f(x),f(y)) \le \rho(x,y)$ for all~$x,y\in
X$. Then $\mu^1_Y(f(A)) \le \mu^1_X(A)$ for every $\mu^1_X$-measurable
subset~$A$ of~$X$.
\end{lem}

\subsection{Background geometric function theory}
\label{subsec:gft-background}

In this section some results from geometric function theory which will
be used in Sections~\ref{sec:cxstructure} and~\ref{sec:modcont} are
summarized. Standard references are Ahlfors-Sario~\cite{AhSa},
Ahlfors~\cite{Ah} and Lehto and Virtanen~\cite{LeVi}.

\begin{defns}[Module of an annular region, concentric, nested]
\label{defns:annulus-stuff}
An {\em annular region} in a surface is a subset homeomorphic to an
open round annulus
\[A(r_1,r_2) = \{z\in\C\,:\, r_1<|z|<r_2\},\]
where $0\le r_1<r_2\le\infty$. If~$R\sbs\C$ is an annular region,
then there is a conformal map taking~$R$ onto some round
annulus~$A(r_1,r_2)$, and this map is unique up to postcomposition
with a homothety. It follows that the ratio $r_2/r_1$ is a conformal
invariant of~$R$, and the {\em module} $\mod R$ of~$R$ is defined by
\[
\mod R = 
\left\{
\begin{array}{ll}
\infty & \text{ if }r_1=0 \text{ or }r_2=\infty,\\
\frac{1}{2\pi} \ln\frac{r_2}{r_1} \quad & \text{ otherwise.}
\end{array}
\right.
\]

Suppose that~$S$ is a closed topological disk and that~$R\sbs S$
is an annular region. The {\em bounded component} of~$S\setminus R$ is
the one which is disjoint from~$\partial S$.  Annular regions
$R_0,R_1\sbs S$ are {\em concentric} if the bounded
complementary component of one of them is contained in the bounded
complementary component of the other. They are {\em nested} if they
are concentric and disjoint, that is, if one is entirely contained in
the bounded complementary component of the other. 
\end{defns}

If $R_0,R_1$ are concentric annular regions and $R_1\sbs R_0$, then
$\mod R_1\leq \mod R_0$. The way in which this observation will be
used here is in the form of the following:

\begin{lem}[Conformal puncture]\label{lem:conformalpnct}
Consider a 1-parameter family $R(t)$ ($t\ge0$) of concentric annular
regions contained in a disk in the complex plane. Assume each $R(t)$
has finite module and that $R(t')\sbs R(t)$ if $t<t'$. If $\mod
R(t)\to\infty$ as $t\downarrow 0$ then $R(0)=\bigcup_{t>0}R(t)$ is an
open annular region whose bounded complementary component is a point.
\end{lem}

To verify the hypotheses of this lemma, the following fundamental
inequality will be used. 

\begin{lem}\label{lem:summod}
If $\{R_n\}$
is a finite or countable family of nested annular regions
all contained in and concentric with the annular region $R$, then
\begin{equation}\label{eq:summod}
\mod R\geq \sum\mod R_n.
\end{equation}
\end{lem}

From this it follows that if $\sum\mod R_n$ diverges then at least
one of the complementary components of $R$ is a point.

A {\em conformal metric} on $\C$ is a metric obtained defining the
length of arcs by $|\gamma|_\nu:=\int_\gamma\nu(z)|\!\rmd z|$, where
$\nu$ is a nonnegative Borel measurable function defined on $\C$. If
$R$ is an annular region in~$\C$ whose boundary components are
$C_1,C_2$ then
\[\mod R= \sup\left\lbrace\dfrac{d_\nu(C_1,C_2)^2}{\Area_\nu(R)}\,:\, 
\nu(z)|\!\rmd z| \text{ is a conformal metric}\right\rbrace,\] where
$d_\nu(C_1,C_2)$ is the $\nu$-distance between the boundary components
of $R$ (i.e., the minimum $\nu$-length of an arc with endpoints
in~$C_1$ and~$C_2$) and $\Area_\nu(R)=\iint_R\nu(z)^2\rmd x\rmd y$ is
the $\nu$-area of $R$.  It follows that
\begin{equation}\label{eq:modest}
\mod R\geq \dfrac{d_\nu(C_1,C_2)^2}{\Area_\nu(R)}
\end{equation}
for any conformal metric $\nu(z)|\!\rmd z|$.

\begin{defn}[Gr\"otzsch annular region]
Let~$t\in[0,1)$. The {\em Gr\"otzsch annular region} $\Gr(t)$ is the
  region obtained from the open unit disk by removing the closed interval
  $[0,t]$ in the real axis.
\end{defn}

The following inequality can be found in~\cite{LeVi} and~\cite{Ah}:
\begin{equation}\label{eq:Grotzschmodest}
 \mod \Gr(t)\leq \frac{1}{2\pi}\ln\frac{4}{t}.
\end{equation}

The following theorem (see~\cite{Ah}, p.72) will be of fundamental
importance in Section~\ref{sec:modcont}.

\begin{thm}[Gr\"otzsch Annulus Theorem]
\label{thm:GAT}
Let~$R$ be an annular region contained in the unit disk in the complex
plane. If both $0$ and a point~$z$ are contained in the bounded
complementary component of~$R$, then
\[\mod R \le \mod {\Gr(|z|)} \le \frac{1}{2\pi}\ln\frac{4}{|z|}.\]
\end{thm}

The following distortion theorem will be needed. For completeness, a
proof is included. The term {\em univalent} is used in this article
as a synonym for {\em injective holomorphic}.

\begin{thm}\label{thm:Koebedistn}
Let $P\sbs\C$ be a closed topological disk, $\tp\in\Int(P)$ and
$\Phi\co \Int(P)\to\C$ be a univalent function with $\Phi(\tp)=0$ and
$\Phi'(\tp)=1$. Let $\tQ(h)$ denote the closed interior collar
neighborhood of $\partial P$ of $d_\C$-width~$h>0$ and set
$P_h:=P\setminus\tQ(h)$. Assume that $h$ is small enough that $P_h$ is
path connected and $\tp\in P_h$, and set
\[\kappa:=\exp\left(\frac{8\diam_{P_h}P_h}{h}\right),\]
where $\diam_{P_h}P_h$ denotes the diameter of $P_h$ in the intrinsic
metric $d_{P_h}$ on $P_h$ induced by~$d_\C$.
Then for any $\tq\in P_h$,
\[\frac{1}{\kappa}\leq |\Phi'(\tq)| \leq 
\kappa.
\]
In particular, $\Phi$ is $\kappa$-biLipschitz in $P_h$.
\end{thm}
\begin{proof}
The usual distortion theorem (see~\cite{Ah}) states that if
$f\co\D\to\C$ is univalent then, for any $z_1,z_2\in\D$, 
\begin{equation}\label{eq:usualdistn}
\frac{1}{m(z_2,z_1)}\leq\frac{|f'(z_1)|}{|f'(z_2)|}\leq m(z_1,z_2),
\end{equation}
where
\begin{equation}\label{eq:m}
m(z_1,z_2):= \frac{1+|z_1|}{1-|z_2|}\cdot
\left(\frac{1+|z_2|}{1-|z_1|}\right)^3.
\end{equation}

Let $\zeta\co\D\to \Int(P)$ be a Riemann mapping with $\zeta(0)=\tp$
and set $f:=\Phi\circ\zeta$. Then both $f$ and $\zeta$ are univalent
so that, writing $\zeta_i:=\zeta(z_i)$
for $i=1,2$, it follows from~(\ref{eq:usualdistn}) that
\[\frac{1}{m(z_1,z_2)\cdot m(z_2,z_1)}
\leq\frac{|\Phi'(\zeta_1)|}{|\Phi'(\zeta_2)|}=
\frac{|f'(z_1)|}{|f'(z_2)|}\cdot\frac{|\zeta'(z_2)|}{|\zeta'(z_1)|}
\leq m(z_1,z_2)\cdot m(z_2,z_1), \]
and~(\ref{eq:m}) gives
\begin{eqnarray*}
m(z_1,z_2)\cdot m(z_2,z_1)&=&\left(\frac{1+|z_1|}{1-|z_1|}\cdot
\frac{1+|z_2|}{1-|z_2|}\right)^4 \\
 &=& \left[\exp\left( d_\D^H(0,z_1)\right)\cdot
\exp\left( d_\D^H(0,z_2)\right)\right]^4,
\end{eqnarray*}
where $d_\D^H$ denotes the Poincar\'e distance in the unit disk
$\D$. Setting $z_2=0$ it follows that, for any $z\in\D$,
\begin{equation} \label{eq:ineq}
\left[\exp\left( d_\D^H(0,z)\right)\right]^{-4}\leq|\Phi'(\zeta(z))|
\leq\left[\exp\left( d_\D^H(0,z)\right)\right]^{4}.
\end{equation}
Now let $\lambda_P(\zeta)|\rmd\zeta|$ denote the Poincar\'e metric on
$\Int(P)$. It is well known (see for example~\cite{Ah}) that
\[\lambda_P(\zeta)\leq \frac{2}{d_{\C}(\zeta,\partial P)}.\]
Let $\tq=\zeta(z)\in P_h$. Then, if $d_P^H$ denotes the
Poincar\'e distance in $\Int(P)$,
\begin{eqnarray*}
d_\D^H(0,z)&=& d_P^H(\tp,\tq)\\
    &=& \inf\left\lbrace
\int_{\gamma}\lambda_P(\zeta)|\rmd\zeta|;\;\gamma\text{ is a path in
$\Int(P)$ from } \tp \text{ to } \tq\right\rbrace \\
    &\leq& \inf\left\lbrace
\int_{\gamma}\lambda_P(\zeta)|\rmd\zeta|;\;\gamma\text{ is a path in
$P_h$ from } \tp \text{ to } \tq\right\rbrace \\
    &\leq& \frac{2d_{P_h}(\tp,\tq)}{h}\\
    &\leq& \frac{2\diam_{P_h}P_h}{h},
\end{eqnarray*}
which, together with~(\ref{eq:ineq}), establishes the result.
\end{proof}

\medskip

Finally, the following {\em holomorphic removability} criterion will
be needed: it is a consequence of Theorem~V.3.2 on p.\ 202
of~\cite{LeVi}. 

\begin{thm}\label{thm:removability2}
Let $\Omega\sbs\C$ be a domain and $E\sbs\Omega$ be a compact subset
with finite \mbox{1-dimensional} Hausdorff measure. Suppose
$g\co\Omega\to\C$ is an orientation-preserving homeomorphism onto its
image which is conformal on $\Omega\setminus E$. Then, in fact, $g$ is
conformal on all of $\Omega$.
\end{thm}

\section{Paper-folding schemes}
\label{sec:surfaceorigami}

This section provides the definition of paper-folding schemes, which
are the main objects of study in this article. Some properties of
the metric and measure on the scar of a paper-folding scheme are
established (Lemma~\ref{lem:metricstructureofG}), and the notion of a
{\em plain} paper-folding scheme is introduced.

The polygons used in the definition of paper-folding schemes are
subsets of the Euclidean plane $\R^2$, which will be identified with
the complex plane $\C$ in Sections~\ref{sec:cxstructure}
and~\ref{sec:modcont} when the complex structure of paper surfaces is
discussed. The notation $\csph$ will be used for the Riemann sphere.

\begin{defns}[Arc, segment, polygon, multipolygon]
An {\em arc} in a metric space~$X$ is a homeomorphic image
$\gamma\sbs X$ of the interval $[0,1]$.  Its {\em endpoints} are
the images of $0$ and $1$ and its {\em interior} is the image of
$(0,1)$, denoted $\interior{\gamma}$. An {\em open arc} is the
interior of an arc. A {\em segment} is an arc in $\R^2$ which is a
subset of a straight line. The length of a segment $\alpha$ is denoted
$|\alpha|$.  A {\em simple closed curve} in $X$ is a homeomorphic
image of the unit circle.

An arc or simple closed curve in~$\R^2$ is called {\em polygonal} if
it is the concatenation of finitely many segments. Its {\em vertices}
are the intersections of consecutive maximal segments, and the maximal
segments themselves are its {\em edges}.

 A {\em polygon} is a closed topological disk in $\R^2$ whose boundary
 is a polygonal simple closed curve. Its {\em vertices} are the same
 as its boundary's vertices and its {\em sides} are the edges forming
 its boundary.

A {\em multipolygon} is a disjoint union of finitely many polygons.  A
{\em (polygonal) multicurve} is a disjoint union of finitely many
(polygonal) simple closed curves. 
\end{defns}

\begin{defns}[Segment pairing, interior pair, full collection, fold]
Let $C\sbs\R^2$ be an oriented polygonal multicurve and
$\alpha,\alpha'\sbs C$ be segments of the same length with disjoint
interiors. The {\em segment pairing}
$\left\langle\alpha,\alpha'\right\rangle$ is the relation which
identifies pairs of points of $\alpha$ and $\alpha'$ in a
length-preserving and orientation-reversing way. The segments
$\alpha,\alpha'$ and any two points which are identified under the
pairing are said to be {\em paired}. Two paired points which lie in
the interior of a segment pairing form an {\em interior pair}. Observe
that the notation for a pairing is not ordered, so that
$\left\langle\alpha,\alpha'\right\rangle$ and
$\left\langle\alpha',\alpha\right\rangle$ represent the same pairing.

A collection $\{\left\langle\alpha_i,\alpha_i'\right\rangle\}$ of
segment pairings is {\em interior disjoint} if the interiors of all of
the segments $\alpha_i$ and $\alpha_i'$ are pairwise disjoint.

The {\em length} of a segment pairing
$\llangle\alpha,\alpha'\rrangle$, denoted
$|\llangle\alpha,\alpha'\rrangle|$, is the length of one of the arcs
in the pairing, i.e.,
$|\llangle\alpha,\alpha'\rrangle|:=|\alpha|=|\alpha'|$. If
$\cP=\{\llangle\alpha_i,\alpha'_i\rrangle\}$ is a (countable) interior
disjoint collection of segment pairings on~$C$, its {\em length},
denoted $|\cP|$, is the sum of the lengths of the pairings in $\cP$,
i.e., $|\cP|=\sum_i|\llangle\alpha_i,\alpha'_i\rrangle|$.

An interior disjoint collection $\cP$ of segment pairings on~$C$ is
{\em full} if $|\cP|$ equals half the length of~$C$. This means that
the pairings in $\cP$ cover $C$ up to a set of Lebesgue 1-dimensional
measure zero.

A pairing of two segments which have an endpoint in common is called a
{\em fold} and the common endpoint is called its {\em folding
  point}. The folding point in a fold is therefore not paired with
any other point.

An interior disjoint collection~$\cP$ of segment pairings induces a
reflexive and symmetric {\em pairing relation}, also denoted~$\cP$:
\[\cP = \{(x,x'):\text{ $x,x'$ are paired or $ x=x'$}\}.\]
\end{defns}

\begin{defns}[Paper-folding scheme]
\label{defn:origami}
A {\em paper-folding scheme} is a pair $(P,\cP)$ where $P\sbs\R^2$ is
a multipolygon with the intrinsic metric $d_P$ induced from $\R^2$,
and $\cP$ is a full interior disjoint collection of segment pairings
on $\partial P$ (positively oriented). The metric quotient
$S:=P/d_P^\cP$ of $P$ under the semi-metric $d_P^\cP$ induced by the
pairing relation $\cP$ is the associated {\em paper space}. When $S$
is a closed (compact without boundary) topological surface, then
$(P,\cP)$ is called a {\em surface paper-folding scheme} and $S$ is
the associated {\em paper surface}.

The projection map is denoted $\pi\co P\to S$ and the quotient
$G=\pi(\partial P)\sbs S$ of the boundary is the {\em scar}. Notice
that the restriction $\pi:\Int(P)\to S\setminus G$ is a homeomorphism.

The (quotient) metric on~$S$ is denoted~$d_S$. The metric~$d_G$ on~$G$
is defined to be the intrinsic metric as a subset of~$S$: it will be
shown in Lemma~\ref{lem:metricstructureofG} below that this is equal
to the quotient metric on~$G$ as a quotient of~$\partial P$, where
$\partial P$ is endowed with its intrinsic metric as a subset of~$P$.

The measure $\rmm_G$ on~$G$ is defined to be the push-forward of
Lebesgue $1$-dimensional measure $\rmm_{\partial P}$ on $\partial
P$. Hausdorff $1$-dimensional measure on~$G$ is denoted~$\mu^1_G$ ---
it will be shown in Lemma~\ref{lem:metricstructureofG} that $\mu^1_G=
\frac{1}{2}\rmm_G$.
\end{defns}

\medskip

Example~\ref{ex:generalexample} below provides simple concrete
examples of paper-folding schemes. The next definitions distinguish
different types of points in a paper space.

\begin{defns}[Vertex, edge, singular point, planar point]
\label{defn:origamipts}
For $k\in\N\cup\{\infty\}$, a point $x\in G$ is a {\em vertex of
  valence~$k$}, or a {\em $k$-vertex}, if either (i) $\#\pi\I(x)=k\neq
2$; or (ii) $\#\pi\I(x)=k=2$ and $\pi\I(x)$ contains a vertex of
$P$. Let~$\cV$ denote the set of all vertices of~$G$.

The points of the paper space~$S$ are divided into three types:
\begin{description}
\item[Singular points] vertices of valence~$\infty$ and accumulations
  of vertices. Let~$\cV^s$ denote the set of singular points.
\item[Regular vertices] vertices which are not singular.
\item[Planar points] all other points of~$S$: that is, the points of
  $S\setminus \bcV$.
\end{description}

The closures of the connected components of $G\setminus\bcV$ are
called {\em edges} of the scar~$G$.
\end{defns}

\begin{rmks}\mbox{}
\begin{enumerate}[i)]
\item When discussing paper foldings, symbols with tildes will usually
  refer to objects in the source~$P$ and symbols without tildes to
  objects in the quotient~$S$. Thus, for example, $\tx$ may denote a
  point in $P$ and $x$ its projection in $S$; or
  $\lang\talpha,\talpha'\rang$ may be a pairing on $\partial P$ and
  $\alpha:=\pi(\talpha)=\pi(\talpha')\sbs G$. 
For notational simplicity, however, tildes
  will be omitted when all objects being discussed lie in~$P$.
\item The metric on a multipolygon $P$ is always the intrinsic
  metric induced by the metric on~$\R^2$, denoted $d_P$. The symbol
  $d_P^\cP$ will be used to denote both the semi-metric on $P$ induced
  by the pairing relation and the distance beteween $\sim_\cP$
  equivalence classes: if $\tx,\ty\in P$ and $[\tx],[\ty]$ denote
  their $\sim_\cP$-classes, then
  $d_P^\cP([\tx],[\ty]):=d_P^\cP(\tx,\ty)$. Thus, if
  $x=\pi(\tx),y=\pi(\ty)\in S$, then $d_S(x,y):=d_P^\cP([\tx],[\ty])$.
\item The endpoints of paired segments need not necessarily project to
  vertices or singular points. For example, a segment pairing
  $\left\langle\alpha,\alpha'\right\rangle$ can be split into two
  segment pairings $\ssegpair{\beta}$, $\ssegpair{\gamma}$ by
  subdividing the segments~$\alpha$ and~$\alpha'$, and the common
  endpoint of $\beta$ and $\gamma$ then projects to a planar point
  (see also Definition~\ref{defn:subpairing}).
\end{enumerate}
\end{rmks}

The next lemma summarises the properties of the metric and measure on
the scar~$G$ which will be important later. One of the issues which it
addresses is the equivalence of the two natural ways to define a
metric on~$G$: as the intrinsic metric induced by the inclusion of~$G$
as a subset of $(S,d_S)$; and as the quotient metric coming from the
intrinsic metric on~$\partial P$.

\begin{lem}
\label{lem:metricstructureofG}
Let~$G$ be the scar of the paper-folding scheme~$(P,\cP)$. Then
\begin{enumerate}[a)]
\item The set of planar points is open and dense in the scar~$G$, while the
set~$\bcV$ of vertices and singular points is a closed nowhere dense subset
of~$G$ with zero $\rmm_G$-measure.
\item The intrinsic metric on~$G$ induced by the inclusion $G\subset
  S$ agrees with the quotient metric on $G=\partial P / d_{\partial
    P}^{\cP}$ induced by the intrinsic metric~$d_{\partial P}$
  on~$\partial P$.
\item $G$ has Hausdorff dimension 1, and Hausdorff $1$-dimensional
  measure $\mu_G^1$ on~$G$ is equal to~$\frac{1}{2}\rmm_G$.
\item Every arc~$\gamma$ in~$G$ is rectifiable, and $|\gamma|_G =
  \frac{1}{2}\rmm_G(\gamma)$. 

\end{enumerate}
\end{lem}

\begin{proof} \mbox{}
\begin{enumerate}[a)]
\item The set of planar points is open in~$G$ by definition (it is the set
of points which are neither vertices nor accumulations of vertices),
and is dense in~$G$ since it contains the image under the continuous
surjection~$\pi$ of the (dense) set of points of~$\partial P$ which
belong to interior pairs. The set of vertices and singular points is
the complement of the set of planar points, and is therefore closed
and nowhere dense.

$\pi^{-1}(\bcV)$ is contained in the complement of the set of points
belonging to interior pairs, which has zero $\rmm_{\partial
  P}$-measure, so that $\rmm_G(\bcV)=0$.
\item Observe first that the intrinsic metrics $d_P$ on~$P$ and
  $d_{\partial P}$ on~$\partial P$ satisfy $d_{\partial P}(\tx,\ty)
  \ge d_P(\tx,\ty)$ for all $\tx, \ty\in\partial P$ (since any path
  in~$\partial P$ is also a path in~$P$), and that $d_{\partial
    P}(\tx,\ty) =d_P(\tx,\ty)=d_{\R^2}(\tx,\ty)$ if $\tx$ and $\ty$ lie
  in the same side of~$P$.

Let $d^i_G$ denote the intrinsic metric on~$G$ as a subset
of~$(S,d_S)$, and $d^q_G$ denote the metric on~$G$ considered as the
quotient metric space of $(\partial P, d_{\partial P})$ under the
relation $\cP$. It will be shown that $d^i_G \le d^q_G$ and $d^q_G \le
d^i_G$, which will establish the result.

\textbf{To show $d^i_G \le d^q_G$:} Let $x,y\in G$, and
write $D=d^q_G(x,y)$. Let~$\veps>0$ be any positive number. A
path~$\gamma$ in~$G$ from~$x$ to~$y$ with length $|\gamma|_S <
D+\veps$ will be constructed, which will establish the result since
$d^i_G(x,y) \le |\gamma|_S$.

Let $\tx,\ty\in\partial P$ with $\pi(\tx)=x$ and $\pi(\ty)=y$. By
definition of~$d^q_G$, there is a $\cP$-chain
$((\tp_i,\tq_i))_{i=0}^k$ from $\tx$ to $\ty$ in $\partial P$ with
length
\[\rmL^{\cP}((\tp_i,\tq_i)) = \sum_{i=0}^k d_{\partial
  P}(\tp_i,\tq_i) < D+\veps.\] For each~$i$, let $\tgamma_i$ be a path
from $\tp_i$ to $\tq_i$ in~$\partial P$ of length $d_{\partial
  P}(\tp_i, \tq_i)$ (the metric~$d_{\partial P}$ is
strictly intrinsic by Theorem~\ref{thm:compact-intrinsic-strict}), and
let $\gamma_i = \pi\circ\tgamma_i$. Since $\tq_i\cP\tp_{i+1}$ for
each~$i$, the concatenation of the paths $\gamma_i$ is a path $\gamma$
in~$G$ from~$x$ to~$y$; and $|\gamma_i|_{S} \le |\tgamma_i|_P \le
|\tgamma_i|_{\partial P}=d_{\partial P}(\tp_i,\tq_i)$ (the first
inequality is from Remark~\ref{rmk:distance-decreasing}, the second
follows from $d_P \le d_{\partial P}$, and the equality is by choice
of the $\tgamma_i$). Hence $|\gamma|_S < D+\veps$ as required.

\textbf{To show $d^q_G \le d^i_G$:} Let~$x,y\in G$, and write
$D=d^i_G(x,y)$. Let~$\veps>0$ be any positive number. A $\cP$-chain
in $\partial P$ which connects points $\tx$ above~$x$ and $\ty$
above~$y$ will be constructed with length less than $D+\veps$,
which will establish the result since $d^q_G(x,y)$ is less than or
equal to the length of such a chain.

By definition of~$d^i_G$ there is a path~$\gamma\co[0,1]\to G$
from~$x$ to~$y$ with $|\gamma|_S < D+\veps$. That is, 
\[\sum_{i=0}^{k-1} d_S(\gamma(t_i), \gamma(t_{i+1})) < D+\veps\]
for any partition $0=t_0<t_1<\cdots<t_k=1$.

Construct such a partition as follows.  Let~$\te_0$ be a side of~$P$
which contains a point~$\tx_0$ with $\pi(\tx_0)=\gamma(0)=x$, and
let $t_1\in[0,1]$ be the greatest parameter for which there is a point
$\tx_1\in\te_0$ with $\pi(\tx_1)=\gamma(t_1)$.

If $t_1\not=1$, then there must be a side $\te_1\not=\te_0$ of~$P$
containing either $\tx_1$ or a point identified with
it. Let~$t_2\in[t_1,1]$ be the greatest parameter for which there is a
point $\tx_2\in\te_1$ with $\pi(\tx_2)=\gamma(t_2)$. If $t_2\not=1$,
then there is a side $\te_2$ distinct from $\te_1$ and $\te_0$
containing $\tx_2$ or a point identified with it. Let~$t_3\in[t_2,1]$ be
the greatest parameter for which there is a point $\tx_3\in\te_2$ with
$\pi(\tx_3)=\gamma(t_3)$. 

Since~$P$ has only finitely many sides, this process terminates after
a finite number of steps, yielding (after removing repeated parameters
if necessary) a partition $0=t_0<t_1<\cdots<t_k=1$ with the property
that, for each~$i$, there are points $\tp_i$ and $\tq_i$ above
$\gamma(t_i)$ and $\gamma(t_{i+1})$ which lie on the same side
of~$P$. In particular, $d_{\partial P}(\tp_i, \tq_i) =
d_S(\gamma(t_i), \gamma(t_{i+1}))$.

The sequence $((\tp_i,\tq_i))_{i=0}^{k-1}$ is therefore a $\cP$-chain
in $\partial P$ with length
\[\sum_{i=0}^{k-1}d_{\partial P}(\tp_i, \tq_i) = \sum_{i=0}^{k-1}
d_S(\gamma(t_i), \gamma(t_{i+1})) < D+\veps\]
as required.
\item  
Because $\pi\co\partial P\to G$ is distance non-increasing, it follows
from Lemma~\ref{lem:hausdorff-nonincreasing} that
\[\mu^1_G(\ol\cV) \le \mu^1_G(\pi(\ol E)) \le \mu^1_{\partial P}(\ol
E) = \rmm_{\partial P}(\ol E) = 0,\] where~$E\subset\partial P$ is the
set of endpoints of segments in segment pairings. Since
$\rmm_G(\ol\cV)=0$ as well, it is enough to show that $\mu^1_G =
\frac{1}{2}\rmm_G$ on $G\setminus \ol\cV$. However this set has
countably many connected components, each an open arc which is the
homeomorphic and locally isometric image under~$\pi$ of exactly two
disjoint open arcs in~$\partial P$, and the result follows.

Since~$\mu^1_G(G)>0$ is finite, $G$ has Hausdorff dimension~1 as
required.

\item Let~$\gamma$ be an arc in~$G$, consider a partition of~$\gamma$
  with points $x_0,\ldots,x_n\in\gamma$, and let~$\gamma_i$ be the
  subarc of~$\gamma$ with endpoints $x_i$ and $x_{i+1}$. Then, by
  Lemma~\ref{lem:hausdorff-connected},
\[\sum_{i=0}^{n-1} d_S(x_i, x_{i+1}) \le
\sum_{i=0}^{n-1}\diam_G(\gamma_i) \le \sum_{i=0}^{n-1} \mu^1_G(\gamma_i)
\le \mu^1_G(\gamma),\]
so that~$\gamma$ is rectifiable. That $|\gamma|_G =
\frac{1}{2}\rmm_G(\gamma)$ is immediate from~c) and
Lemma~\ref{lem:hausdorff-arc}. 
\end{enumerate}
\end{proof}

The main focus of this article is on paper-folding schemes whose
quotient is a surface and, in particular, a sphere.  For clarity of
exposition, attention will be concentrated on {\em plain} folding
schemes: these are both the most common and the simplest type of paper
foldings and for them the paper space is always a sphere and the scar
is always a dendrite (Theorem~\ref{thm:plain-top-structure}).

\begin{defns}[Unlinked pairing, plain paper-folding scheme]
\label{defns:linked-plain}
Let~$\gamma$ be a polygonal arc or polygonal simple closed curve.

Two pairs of (not necessarily distinct) points $\{x,x'\}$ and
$\{y,y'\}$ of~$\gamma$ are {\em unlinked} if one pair is contained in
the closure of a connected component of the complement of the
other. Otherwise they are {\em linked}.

A symmetric and reflexive relation~$\rmR$ on $\gamma$ is {\em
  unlinked} if any two unrelated pairs of related points are unlinked:
that is, if $x\,\rmR\,x'$, $y\,\rmR\,y'$, and neither $x$ nor~$x'$ is
related to either~$y$ or~$y'$, then $\{x,x'\}$ and $\{y,y'\}$ are
unlinked. 

An interior disjoint collection~$\cP$ of segment pairings on~$\gamma$
is {\em unlinked} if the corresponding relation~$\cP$ is unlinked.

A paper-folding scheme~$(P,\cP)$ is {\em plain} if~$P$ is a single
polygon and~$\cP$ is unlinked.
\end{defns}

\begin{example}
\label{ex:generalexample}
Consider the unit square $P=\{(a,b)\in\R^2:\,0\leq a,b\leq 1\}$ and pick
a decreasing sequence of real numbers $(a_i)_{i\in\N}$ such that
$\sum a_i= 1/2$.  Define the following segment pairings on $\partial P$
(Figures~\ref{fig:hsscheme} and~\ref{fig:cantorscheme}). The two vertical
sides are paired and the top side is folded in half. On the bottom
side define countably many folds of length $a_i$ (i.e., pairings of
two segments with one endpoint in common, each of length $a_i$),
with pairwise disjoint interiors.

In Figure~\ref{fig:hsscheme} the folds are placed in order of
decreasing length from right to left with coincident endpoints. The
complement of the folds is the bottom left vertex of~$P$: all of the
fold endpoints are identified with this vertex in the quotient
space~$S$.

In Figure~\ref{fig:cantorscheme}, the folds are arranged in such a way
that they are all disjoint. The closure of the complement of the folds
is a Cantor set of Lebesgue measure~0, uncountably many of whose
points are unpaired. This Cantor set is identified to a point in the
quotient space~$S$.

In both examples (and for any other way of arranging the folds) the
paper-folding scheme is unlinked, and hence the quotient space is a
topological sphere by Theorem~\ref{thm:plain-top-structure}
below. Notice also that the scars corresponding to the two schemes are
isometric $\infty$-ods (Definition~\ref{defn:infty-od}).

These simple examples will be revisited in
Example~\ref{ex:HorseshoeCantor} below.

\begin{figure}[htbp]
\lab{a0}{a_0}{}\lab{a1}{a_1}{}\lab{a2}{a_2}{}\lab{a3}{a_3}{}
\lab{P}{P}{}\lab{G}{G}{}
\begin{center}
\pichere{0.6}{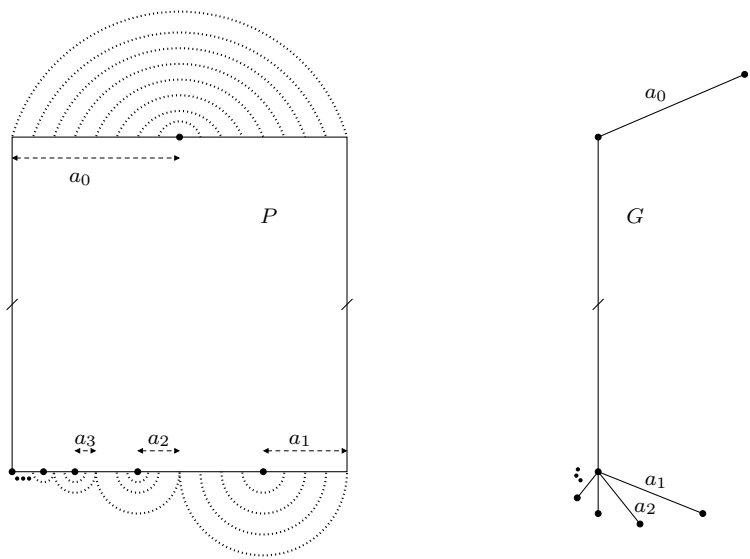}
\end{center}
\caption{A paper-folding scheme and the associated scar.}
\label{fig:hsscheme}
\end{figure}

\begin{figure}[htbp]
\lab{a0}{a_0}{}\lab{a1}{a_1}{}\lab{a2}{a_2}{}\lab{a3}{a_3}{}
\lab{P}{P}{}\lab{G}{G}{}
\begin{center}
\pichere{0.6}{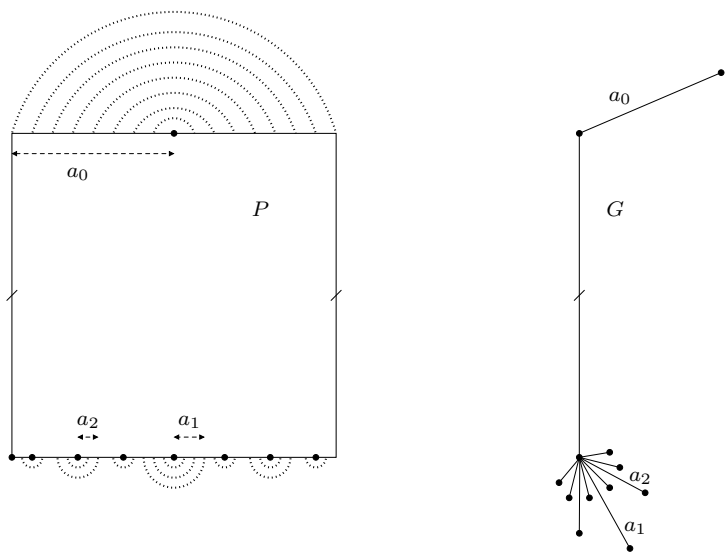}
\end{center}
\caption{Another paper-folding scheme and the associated scar.}
\label{fig:cantorscheme}
\end{figure}
\end{example}

\section{The topological and metric structure of paper spaces}
\label{sec:topol-struct-plain}

The first main theorem of this section,
Theorem~\ref{thm:plain-top-structure}, states that the paper space
associated to a plain paper folding is a sphere, and its scar is a
dendrite. The tools used to prove this result are the dendrite
  quotient theorem (Theorem~\ref{thm:dendritequot}), which gives
sufficient conditions for a quotient of a continuum to be a dendrite;
and Moore's theorem (Theorem~\ref{thm:Moore}), which gives
sufficient conditions for a quotient of a sphere to be a sphere.

 In Section~\ref{subsec:general-top-struct} the topological structure
 of general paper spaces is considered: it is shown that they are
 paper surfaces provided that the boundary identifications are only
 finitely linked, and the structure of the scar in this case is
 described (Theorem~\ref{thm:top-structure}). Finally,
 Section~\ref{subsec:metric-structure} contains a discussion of the
 metric structure of paper spaces.

\subsection{The plain paper folding structure theorem}
\label{subsec:plain-top-structure}

Many of the technical lemmas needed in the proof of the plain paper
folding structure theorem are common to the general case and are not
simplified by the plainness assumption. Thus in the beginning of this
section $(P,\cP)$ is a general paper folding scheme: that is, $\cP$ is
a full interior disjoint collection of segment pairings on the
boundary $\partial P$ of a multipolygon~$P$; the equivalence relation
induced by the semi-metric $d_P^\cP$ is denoted $\sim_\cP$; the paper
space is denoted $S$, the scar $G$, and the projection map $\pi\co
P\to S$; and if $x\in P$, its $\sim_\cP$-equivalence class is denoted
$[x]$.

\begin{defns}[Plain arc, maximal plain arc]
\label{defns:unlinked-plain}
Let $(P,\cP)$ be a paper-folding scheme. An arc $\gamma\sbs
\partial P$ is {\em plain} if:
\begin{enumerate}[a)]
\item Every pairing in $\cP$ which intersects
the interior of $\gamma$ is contained in $\gamma$ (that is, if
$\ssegpair{\alpha}$ is a segment pairing and either~$\alpha$
or~$\alpha'$ intersects the interior of~$\gamma$, then both $\alpha$
and $\alpha'$ are contained in~$\gamma$); and 
\item The restriction of $\cP$ to $\gamma$ is unlinked
  (Definitions~\ref{defns:linked-plain}).
\end{enumerate}
A component $\gamma$ of $\partial P$ is {\em plain} if it is
$\cP$-saturated and the restriction of $\cP$ to $\gamma$ is unlinked.

A plain arc is {\em maximal} if it is not strictly contained in any
plain arc or component. 
\end{defns}

\begin{rmks}\mbox{}
\label{rmks:unlinkedplain}
\begin{enumerate}[a)]
\item Recall (Definitions~\ref{defns:linked-plain}) that a paper
  folding scheme~$(P,\cP)$ is itself called plain if $P$ is a single
  polygon and $\cP$ is unlinked. With the definitions just introduced
  this is equivalent to saying that $P$ consists of a single polygon
  whose boundary $\partial P$ is plain as just defined.
\item 
  Every plain arc contains at least one pairing: pairings are dense in
  $\partial P$ and so intersect the interior of --- and thus must be
  contained in --- any plain arc or component.
\item If every pairing which intersects the interior of $\gamma$ is
  contained in $\gamma$, then the same holds for the complement of
  $\gamma$ in $\partial P$: every pairing which intersects $\partial
  P\setminus\gamma$ is contained in $\ol{\partial
  P\setminus\gamma}$. 
\item The union of two plain arcs with non-empty intersection is
  plain. Hence any two distinct maximal plain arcs are disjoint.
\end{enumerate}
\end{rmks}

\begin{lem}\label{lem:plainendpts}
If $\gamma$ is a plain arc, then its endpoints are
$\sim_\cP$-equivalent.  
\end{lem}
\begin{proof}
Let~$a$ and $b$ be the endpoints of~$\gamma$: for convenience,
they will be referred to as the left and right endpoints respectively,
and the points of~$\gamma$ will be ordered from $a$ to $b$.

Let~$\veps>0$ be any positive number. It will be shown that
$d^\cP_P(a,b)<\veps$, which will establish the result. 

By fullness, there is a finite collection of segment pairings
$\{\llangle\alpha_r,\alpha_r'\rrangle\}_{r=1}^N$ contained in $\gamma$ which
cover $\gamma$ to within length~$\veps$. Label the segments so that
$\alpha_r$ is to the left of~$\alpha_r'$ for each~$r$.

Let~$p_0=a$, and construct a $\cP$-chain from $a$ to $b$
inductively as follows: for each~$i\ge 0$, let~$q_i$ be the leftmost
point of $[p_i,b]$ which is a left hand endpoint of
some~$\alpha_{r(i)}$, and let~$p_{i+1}$ be the right hand endpoint of
$\alpha'_{r(i)}$, so that $q_i\, \cP\, p_{i+1}$. Since the
subcollection of segment pairings is finite, there is some~$k$ such
that $[p_k,b]$ is disjoint from all of the~$\alpha_r$:
set~$q_k=b$. By construction, $p_0\le q_0 < p_1 \le q_1 <
\cdots <p_k \le q_k$.

The intervals $(p_i,q_i)$ are disjoint from all of the segments
$\alpha_r$ by choice of $q_i$. They are also disjoint from the
$\alpha_r'$. For if $(p_i, q_i)$ were to contain
some~$\alpha_r'$, then~$\alpha_r$ couldn't lie in any $(p_j, q_j)$
with $j<i$ as above; and if $\alpha_r$ were to lie in
$(q_j, p_{j+1})$ then the segments
$\alpha_{r(j)}<\alpha_r<\alpha'_{r(j)} < \alpha'_r$ would be linked,
contradicting plainness.

Hence \[d^\cP_P(a,b) \le \rmL^\cP((p_i,q_i)) = \sum_{i=0}^k d(p_i, q_i) <
\veps\]
as required.

\end{proof}

\begin{defn}[Subpairing]
\label{defn:subpairing}
Let~$\llangle\alpha,\alpha'\rrangle$ be a pairing in~$\cP$,
$\beta$~be a subarc of~$\alpha$, and $\beta'$ be the corresponding
subarc of~$\alpha'$. Then $\llangle\beta,\beta'\rrangle$ is said to
be a {\em subpairing} of $\llangle\alpha,\alpha'\rrangle$.
\end{defn}

\begin{lem}
\label{lem:separatingpairing}
Let~$\gamma$ be a plain arc or component of~$\partial P$, and
$\llangle\alpha,\alpha'\rrangle$ be a pairing in~$\cP$, or a
subpairing of a pairing in~$\cP$, which is contained
in~$\gamma$. Let~$z\in \gamma\setminus(\interior{\alpha}\cup
\interior{\alpha'})$, and if $\gamma$ is an arc then assume that
$\interior{\alpha}\cup \interior{\alpha'}$ separates~$z$ from the
endpoints of~$\gamma$.

Then the~$\cP$-semi-distance along~$\partial P$ from $z$ to any
point~$w$ in a different component of $\partial P \setminus
(\interior{\alpha} \cup \interior{\alpha'})$ is bounded below by the
length~$L$ of $\llangle\alpha,\alpha'\rrangle$:
\[d_{\partial P}^\cP(z,w)\ge 
 L := \left|\llangle\alpha,\alpha'\rrangle\right|.\]
\end{lem}

\begin{proof}
Let~$\gamma_1$ be the component of
$\partial P \setminus (\interior{\alpha} \cup \interior{\alpha'})$
which contains~$z$, and let~$\gamma_2$ denote $\partial P \setminus
(\gamma_1 \cup \interior{\alpha} \cup \interior{\alpha'})$. Since
all of the pairings in~$\cP$ are unlinked with
$\llangle\alpha,\alpha'\rrangle$, both~$\gamma_1$ and $\gamma_2$
are~$\cP$-saturated. 

Let~$((p_i,q_i))_{i=0}^k$ be a $\cP$-chain in $\partial P$ from $z$ to
$w$. By saturation of $\gamma_1$, if $q_r\in \gamma_1$ then
$p_{r+1}\in\gamma_1$: hence the last point of the chain which lies in
$\gamma_1$ must be some~$p_r$; and by saturation of~$\gamma_2$, the
first point of the chain after~$p_r$ which lies in~$\gamma_2$ must be
some $q_s$ with $s\ge r$.

If $s=r$ then $d_{\partial P}(p_r, q_r) \ge L$. If $s>r$, then the
points $q_r,p_{r+1},q_{r+1},\ldots, p_s$ all lie in
$\interior{\alpha} \cup \interior{\alpha'}$. Let
$\xi,\xi':[0,L] \to
\partial P$ be parameterisations of $\alpha$ and $\alpha'$ by arc
length, and let~$t_i$ be parameters such that $q_i$ and $p_{i+1}$
are equal to $\xi(t_i)$ or $\xi'(t_i)$ for $r\le i < s$. Then
\[d_{\partial P}^\cP(z,w)\ge \sum_{i=r}^s d_{\partial P}(p_r,
  q_r) \ge t_r + \sum_{i=r+1}^{s-1}|t_{i}-t_{i-1}| + |L-t_{s-1}| \ge L\]
as required.
\end{proof}

Recall that an arc~$\gamma$ is plain if the pairing relation~$\cP$
has the properties that: every pairing in~$\cP$ which intersects
$\interior{\gamma}$ is contained in~$\gamma$; and the restriction
of~$\cP$ to $\gamma$ is unlinked. The following lemma states that
analogous properties also hold for the equivalence relation $\sim_\cP$.

\begin{lem}\label{lem:plainequivsaturated}
Let $\gamma$ be a plain arc and let $a,b$ denote its
endpoints. Then $\sim_\cP$ is unlinked on
$\gamma$ and $\gamma\setminus[a] = \gamma\setminus[b]$ is
$\sim_\cP$-saturated. Similarly, if $\gamma$ is a plain
component of $\partial P$, then $\sim_\cP$ is unlinked on $\gamma$ and
$\gamma$ is $\sim_\cP$-saturated.   
\end{lem}
\begin{proof}
\textbf{Unlinkedness:} Assume that~$\gamma$ is
a plain arc: the case where it is a plain component of $\partial P$ is
analogous. Let~$x$, $x'$, $y$, and $y'$ be points
of~$\gamma$ with $x\sim_\cP x'$ and $y\sim_\cP y'$, such that
the pairs $\{x,x'\}$ and $\{y,y'\}$ are linked (and in
particular all four points are distinct). It is necessary to show that
all four points are $\sim_\cP$ equivalent. 

Without loss of generality, suppose that $a \le x < y < x' <
y' \le b$ with respect to an order on~$\gamma$. The three arcs
$[x,y]$, $[y,x']$, and $[x',y']$ are all plain, by
Lemma~\ref{lem:separatingpairing}: for example, if a point of
$(x,y)$ were paired with a point of $[a,x)$ or a point of
  $(x',b]$, the segment pairing realizing this would separate
$x$ from $x'$, contradicting $x\sim_\cP x'$; while if it were
paired with a point of $(y,x']$ the segment pairing would separate
  $y$ from $y'$.  Hence $x\sim_\cP y\sim_\cP x'\sim_\cP y'$
  by Lemma~\ref{lem:plainendpts}.

\textbf{Saturation:} Let~$\gamma$ be a plain arc, and
$x\in\interior{\gamma}$. It is required to show that either
$x\sim_\cP a$, or $[x]\sbs\gamma$.

If there is a pairing or subpairing in~$\gamma$ which separates $x$
from $a$ and $b$, then it separates $x$ from $\partial P
\setminus \gamma$, and $[x]\sbs\gamma$ by
Lemma~\ref{lem:separatingpairing}.

If there is no such pairing, then the arcs $[a,x]$ and $[x,b]$
are both plain, and $[x] = [a] = [b]$ by Lemma~\ref{lem:plainendpts}.

If $\gamma$ is a plain component of~$\partial P$ then it is clear
that no point of $\gamma$ can be $\sim_\cP$-equivalent to a point in
another boundary component of~$P$.
\end{proof}

\begin{thm}[Topological structure of a plain paper folding]
\label{thm:plain-top-structure}
The quotient $S$ of a plain paper-folding scheme is a topological
sphere, and its scar~$G$ is a dendrite.
\end{thm}

\begin{proof}
Since both $\partial P$ and $P$ are compact, the scar~$G$ and the
paper space~$S$ are homeomorphic to the topological quotients
$\partial P/{\sim_\cP}$ and $P/{\sim_\cP}$ respectively by
Theorem~\ref{thm:compact-metr-top-agree}.

To show that~$G$ is a dendrite, it is enough by
Theorem~\ref{thm:dendritequot} to show that the decomposition~$\cG$ of
$\partial P$ into equivalence classes of~$\sim_\cP$ is
dendritic. That~$\cG$ is non-separated is immediate by
Lemma~\ref{lem:plainequivsaturated}: saying that $\sim_\cP$ is
unlinked is the same as saying that $\cG$ is non-separated. Thus it
remains to show that if $[x]\in\cG$ and $y\not\in[x]$, then there is
some $[z]\in\cG$ which separates $y$ and~$[x]$.

Let~$A\sbs\partial P$ be the maximal arc containing~$y$ which is
disjoint from~$[x]$. Thus~$A$ is an open arc with endpoints $x_1, x_2
\in[x]$ (possibly $x_1=x_2$ if $[x]$ is a single point). Since
$\sim_\cP$ is unlinked, $A$ is $\sim_\cP$-saturated.

If there are paired points~$z_1$ and~$z_2$ in different components
of~$A\setminus\{y\}$, pick nearby interior paired points $z_1'$ and
$z_2'$, so that $[z_1']=[z_2']$ does not contain~$y$. Then
$[z_1']$ separates $y$ and $[x]$ as required.

If there are no such points then $[x_1,y]$ and $[y,x_2]$ are
plain, so that $y\in[x]$ by Lemma~\ref{lem:plainendpts}, a
contradiction. 

\medskip\medskip

To show that the paper space~$S$ is a sphere,
it is convenient to change coordinates by a suitable homeomorphism so
that~$\partial P$ is the unit circle~$S^1$ in $\widehat\C$, and~$P$ is
the exterior of the unit disk $\D\sbs\csph$: the unit disk $\D$
will be regarded as the hyperbolic plane. This change of coordinates
is purely for convenience: the constructions don't use hyperbolic
geometry in any essential way.

Connect the points $x,y$ of each interior pair in $\partial P$ with
the geodesic having them as endpoints.  Because~$\cP$ is unlinked,
these geodesics are pairwise disjoint. Let $\cL$ be the geodesic
lamination obtained by taking the closure of the union of such
geodesics. Because the equivalence relation~$\sim_\cP$ is closed,
every geodesic in~$\cL$ joins a pair of equivalent points.

If~$U$ is a component of $\D\setminus\cL$, let~$\bU$ denote its
closure in~$S$. Then~$\bU$ is a closed disk, whose boundary is the
union of some geodesics (including endpoints) in~$\cL$ joining
equivalent points which are not in interior pairs, and other points
of~$\partial P$.

Let~$x,y\in\partial P$ be distinct points which are not in interior
pairs, and let~$[x,y]\sbs\partial P$ be an interval with these
endpoints. Then the following conditions are equivalent:
\begin{enumerate}[i)]
\item $x\sim_\cP y$.
\item $[x,y]$ is plain.
\item $x$ and $y$ lie in the closure~$\bU$ of the same component~$U$
  of $\D\setminus\cL$ or are the endpoints of a geodesic in~$\cL$.
\end{enumerate}

For \textbf{i)$\iff$ii)}, suppose that~$[x,y]$ is not
plain. Since~$\cP$ is unlinked, this means that there is a pairing
$\llangle \alpha,\alpha'\rrangle$ which intersects~$(x,y)$ but is not
contained in~$[x,y]$. The assumption that~$x$ and~$y$ are not in
interior pairs means that one of the segments, say~$\alpha$, is
contained in~$[x,y]$ while the other, $\alpha'$, is disjoint from
$(x,y)$. Thus~$x$ and~$y$ lie in different components of $\partial P
\setminus (\interior\alpha\cup\interior{\alpha'})$, so that
$x\not\sim_\cP y$ by Lemma~\ref{lem:separatingpairing}. The converse
is immediate from Lemma~\ref{lem:plainendpts}.

\medskip

For \textbf{ii)$\iff$iii)}: Suppose that $[x,y]$ is plain and that the
geodesic with endpoints~$x$ and~$y$ is not in~$\cL$. Then there are no
geodesics in~$\cL$ with one endpoint in~$(x,y)$ and the other disjoint
from~$[x,y]$, and hence the geodesic with endpoints~$x$ and~$y$ lies
in a single component~$U$ of $\D\setminus\cL$, so that
$x,y\in\bU$. If~$[x,y]$ is not plain then as above there is a pairing
$\llangle\alpha,\alpha'\rrangle$ with $\interior\alpha \sbs (x,y)$ and
$\interior{\alpha'} \cap [x,y]=\emptyset$, and any geodesic connecting
paired points in $\interior\alpha$ and $\interior{\alpha'}$, together
with its endpoints, separates~$x$ from~$y$ in~$\bDD$.

\medskip

Therefore two distinct points~$x,y\in\partial P$ satisfy $x\sim_\cP y$
if and only if either they are the endpoints of a geodesic in~$\cL$,
or they lie in the closure of the same component of $\D\setminus\cL$,
so that a decomposition $\cG$ of~$\csph$ which realizes $\sim_\cP$ can
be constructed with the following elements:
\begin{enumerate}[I)]
\item Closures $\bU$ of components of~$\D\setminus\cL$;
\item Geodesics in~$\cL$, together with endpoints, which are not
  contained in elements of type~I); and
\item Points which are not contained in elements of types~I) and~II).
\end{enumerate}

The elements of~$\cG$ are closed disks, arcs, and points, and so are
connected and do not separate~$\csph$. To complete the proof using
Moore's theorem, it is therefore only required to show that~$\cG$ is
upper semi-continuous. So let~$x_n\to x$ and~$y_n\to y$ be sequences
in~$\csph$ such that $x_n$ and $y_n$ belong to the same decomposition
element for all~$n$: it is required to show that~$x$ and~$y$ belong to
the same decomposition element.

If~$x_n$ and~$y_n$ belong to elements of type~III) for infinitely
many~$n$, then~$x_n=y_n$ for these values of~$n$, and
hence~$x=y$. Passing to a subsequence, it can therefore be assumed
that none of the~$x_n$ or~$y_n$ belong to an element of type~III).

Suppose that infinitely many of the~$x_n$ and~$y_n$ belong to
geodesics~$\gamma_n$ of type~II). If there is no positive lower bound
on the distance in~$\partial P$ between the endpoints of
the~$\gamma_n$ then $x=y$. If there is such a lower bound then the
geodesics~$\gamma_n$ converge to a geodesic (together with
endpoints) $\gamma$ in $\cL$ which contains~$x$ and~$y$.

It can therefore be assumed that all of the~$x_n$ and~$y_n$ belong to
elements~$\bU_n$ of type~I). For each~$n$, let $\alpha_n\subset \bU_n$
be the geodesic arc with endpoints~$x_n$ and~$y_n$, and let~$\gamma_n$
be the geodesic containing~$\alpha_n$. If there is no lower bound on
the distance in~$\partial P$ between the endpoints of the~$\gamma_n$,
then~$x=y$. If there is such a lower bound then the
geodesics~$\gamma_n$ converge to a geodesic~$\gamma$ containing~$x$
and~$y$.

If either~$x=y$ or~$\gamma$ is a geodesic in~$\cL$, then~$x$ and~$y$
lie in the same decomposition element. Otherwise, the geodesic arc
$\alpha\subset\gamma$ connecting~$x$ and~$y$ is disjoint from $\cL$
(if it intersected a geodesic in~$\cL$, then so would~$\alpha_n$ for
large~$n$), and hence~$x$ and~$y$ lie in the same component
of~$\D\setminus\cL$. 
\end{proof}

\medskip

\begin{example}
\label{ex:HorseshoeCantor}
Here the construction in the above proof is considered for the
paper-folding schemes of Example~\ref{ex:generalexample}. Recall
that~$P$ is a square, the top side of~$P$ is folded in half, the
vertical sides are paired, and folds of lengths~$a_i$ are placed along
the bottom side, either continguously from right to left
(Figure~\ref{fig:hsscheme}) or disjointly
(Figure~\ref{fig:cantorscheme}).

Figure~\ref{fig:hsident} depicts the decomposition of~$\csph$ which
realizes $\sim_\cP$ in the scheme of Figure~\ref{fig:hsscheme}. There
is a single decomposition element of type~I) which is
denoted~$g_\infty$: it includes countably many points of~$\partial P$
(the fold endpoints and the bottom left corner of~$P$), and countably
many geodesics joining these points in pairs. The other geodesics
(including their endpoints on~$\partial P$) are decomposition elements
of type~II), and the points of $\Int(P)$ are decomposition elements of
type~III).

\begin{figure}[htbp]
\psfrag{P}{$P$}
\lab{ginfty}{g_\infty}{}
\begin{center}
\pichere{0.6}{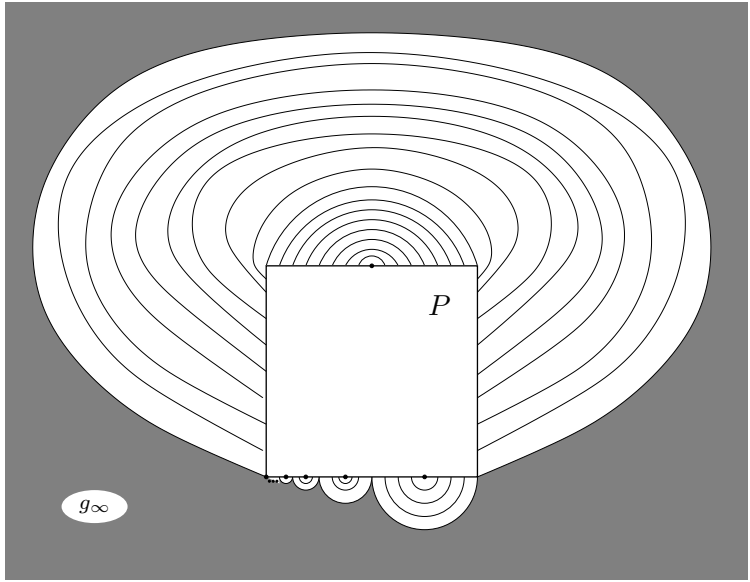}
\end{center}
\caption{The musc decomposition realising $\sim_\cP$ in the paper-folding scheme of
  Figure~\ref{fig:hsscheme}.}
\label{fig:hsident}
\end{figure}

The decomposition for the scheme of Figure~\ref{fig:cantorscheme} is
similar: in this case the single decomposition element~$g_\infty$ of
type~I) intersects~$\partial P$ in a Cantor set, only countably many
of whose points are connected by geodesics.
\end{example}

\subsection{Surface paper folding schemes}
\label{subsec:general-top-struct}
In this section the topological structure of the paper space~$S$ of a
general paper-folding scheme~$(P,\cP)$ is considered. It is clear that
linked identifications around $\partial P$ create handles in the paper
space, and that if there are infinitely many such handles then~$S$
cannot be a compact surface. The main result of this section,
Theorem~\ref{thm:top-structure}, states that this is the only
obstruction to~$S$ being a paper surface, and describes the scar~$G$
in the surface case: each component of~$G$ (corresponding to a
component of~$S$) is the union of countably many dendrites (one
corresponding to each maximal plain arc in~$\partial P$) and a finite
connected graph (corresponding to the linked identifications), each
dendrite being attached to the connected graph at a single point.

The first step is to reduce to the case in which~$P$ is a single
polygon: or, more accurately, a single closed disk, since the
polygonal structure of~$P$ will be lost in the constructions which
follow.

Let~$\Upsilon$ be the finite graph whose vertices are the
components~$P_i$ of~$P$, and which has an edge joining $P_i$ and $P_j$
if and only if there is a pairing
$\llangle\alpha,\alpha'\rrangle\in\cP$ with $\alpha \sbs \partial
P_i$ and $\alpha'\sbs\partial P_j$. Then the connected components
of~$S$ correspond bijectively to the connected components
of~$\Upsilon$. In the remainder of the section it will be assumed
that~$\Upsilon$ is connected (if not, apply the results to each
connected component in turn).

Let~$T$ be a spanning subtree of~$\Upsilon$. Pick one pairing
corresponding to each edge of~$T$, and realize the associated
identifications. The resulting space is a topological disk~$P'$, and
the remaining pairings induce a full interior disjoint
collection~$\cP'$ of arc pairings along~$\partial P'$. It
therefore suffices to consider $(P',\cP')$. As mentioned above, $P'$
is not a polygon, but the polygonal structure is irrelevant to the
topological structure of the quotient: the word ``segment'' will be
understood to mean ``arc'', and the length of an arc in~$\partial P'$
is given by the minimal length of a subset of~$\partial P$ which
projects to it.

In the remainder of this section, then, it will be assumed that~$P$ is
a topological disk with rectifiable boundary, and that~$\cP$ is a full
interior disjoint collection of arc pairings on~$\partial P$. It can
also be assumed that~$(P,\cP)$ is not plain, since otherwise the
results of Section~\ref{subsec:plain-top-structure} apply.

Let~$\Pi=\{\gamma_i\,:\,i\in I\}$ be the set of maximal plain arcs
in~$\partial P$, where~$I$ is a finite or countable index set. These
arcs are mutually disjoint (Remarks~\ref{rmks:unlinkedplain}d)), and
each component of $\partial P \setminus\bigcup_{i\in I}\gamma_i$ is an open
arc (between any two elements of~$\Pi$ lies an arc~$\alpha$ of a
pairing of~$\cP$ whose interior is not contained in any plain arc).

Let~$\cC$ be the simple closed curve obtained from~$\partial P$ by
collapsing each element of~$\Pi$ to a point (in fact the maximal plain
arcs will not be collapsed: rather the identifications on them will be
realised, and in particular their endpoints will be identified). The
pairings in~$\cP$ which are not contained in any element of~$\Pi$
induce a full interior disjoint collection~$\cQ$ of arc pairings
on~$\cC$, with respect to which~$\cC$ contains no plain arcs, since
the preimage in~$\partial P$ of such a plain arc would be a plain arc
intersecting the complement of the~$\gamma_i$.

The aim now is to decompose~$\cC$ into maximal {\em unlinked} arcs.

\begin{defns}[Unlinked arc, maximal unlinked arc]
\label{defn:unlinked-arc}
An arc~$\gamma$ in~$\cC$ is {\em unlinked} if
\begin{enumerate}[a)]
\item Every arc~$\beta_i$ in a pairing in~$\cQ$ which intersects
  $\interior{\gamma}$ is contained in~$\gamma$;
\item $\overline{\bigcup_i \beta_i'}$ is an arc~$\gamma'$, where
  the~$\beta_i'$ are the arcs paired with the arcs~$\beta_i$ of~a);
  and 
\item The restriction of $\cQ$ to $\gamma\cup\gamma'$ is unlinked on~$\cC$.
\end{enumerate}
The arc~$\gamma'$ is said to be {\em paired} with~$\gamma$. An unlinked
arc~$\gamma$ is {\em maximal} if it is not strictly contained in any
other unlinked arc.
\end{defns}

\begin{rmks}\mbox{}
\label{rmk:unlinked-disjoint}
\begin{enumerate}[a)]
\item An unlinked arc~$\gamma$ is disjoint from its paired arc~$\gamma'$, as
otherwise $\gamma\cup\gamma'$ would be a plain arc.
\item 

Let~$\gamma$ be an unlinked arc with pair~$\gamma'$. Then the
restriction of $\sim_\cQ$ to $\gamma\cup \gamma'$ is the arc pairing
$\ssegpair{\gamma}$. 
\end{enumerate}
\end{rmks}

\begin{lem}
\label{lem:maximal-unlinked}
$\cC$ can be written uniquely as a union of maximal unlinked arcs and
points which are not contained in any unlinked arc. The maximal
unlinked arcs intersect only at their endpoints.
\end{lem}

\begin{proof}
If~$\gamma_1$ and~$\gamma_2$ are unlinked arcs whose interiors
intersect, then~$\gamma_1\cup\gamma_2$ is also an unlinked arc. In
particular, two distinct maximal unlinked arcs can only intersect at
their endpoints.

Let~$\beta$ be an arc in a pairing in~$\cQ$. Then~$\beta$ is itself an
unlinked arc. Let~$\gamma_\beta$ be the union of all unlinked arcs
containing~$\beta$. Then~$\gamma_\beta$ is either an arc, or an arc
without one or both of its endpoints, which satisfies the conditions
of Definition~\ref{defn:unlinked-arc} except for the possible absence
of endpoints of $\gamma_\beta$ and $\gamma_\beta'$. However, this
means that the conditions also hold for $\gamma_\beta$ together with
its endpoints, so that~$\gamma_\beta$ is in fact an arc.

Hence every arc in a pairing in~$\cQ$ is contained in a unique maximal
unlinked arc, and these arcs intersect only at their endpoints. If~$x$
is a point of~$\cC$ which is not contained in any of these arcs, then
any arc containing~$x$ must intersect a paired arc~$\beta$ and hence
the maximal arc~$\gamma_\beta$: therefore~$x$ is not contained in any
unlinked arc.
\end{proof}

\begin{defn}[Finitely linked]
\label{defn:finitely-linked}
$(P,\cP)$ is said to be {\em finitely linked} if there are only
finitely many maximal unlinked arcs in the decomposition of
Lemma~\ref{lem:maximal-unlinked}. 
\end{defn}

\begin{rmk}
\label{rmk:finitely-linked}
If~$(P,\cP)$ is finitely linked, then there cannot be any points
of~$\cC$ which are disjoint from all of the maximal unlinked arcs: that
is, $\cC$~can be written uniquely as a union of finitely many maximal
unlinked arcs.
\end{rmk}

The following lemma is the main tool used in the proof of the converse
of Theorem~\ref{thm:top-structure} below, that if~$(P,\cP)$ is not
finitely linked then~$S$ is not a compact surface. The setting of the
lemma is slightly different from that of the remainder of the paper
--- $P$ is a compact surface with a single boundary component~$\cC$
rather than a topological disk --- but the concepts of {\em arc
  pairings on~$\cC$}, {\em plain arcs}, and {\em finitely linked} are
defined analogously.

\begin{lem}
\label{lem:infinitely-linked}
Let~$P$ be a compact surface with a single boundary component~$\cC$,
and~$\cP$ be a full interior disjoint collection of arc
pairings on~$\cC$ which is such that $\cC$ contains no plain arcs, and
is not finitely linked. Then there are maximal unlinked arcs~$\alpha$
and~$\beta$, with paired arcs $\alpha'$ and~$\beta'$, such
that~$\ssegpair{\alpha}$ and~$\ssegpair{\beta}$ are
linked; and carrying out the identifications corresponding to these
arc pairings yields a compact surface~$\hP$ having a single boundary
component~$\hcC$, with the properties that:
\begin{itemize}
\item the genus of~$\hP$ is one greater than that of~$P$;
\item the induced arc pairings~$\hcP$ on~$\hcC$ are such that~$\hcC$
  contains no plain arcs and~$\hcP$ is not finitely linked.
\end{itemize}
\end{lem}

\begin{proof}
First some notation will be introduced. Let~$\cA$ be the set of
maximal unlinked arcs on~$\cC$, and $\phi\co\cA\to\cA$ be the
involution which takes each maximal unlinked arc to its pair. Fix an
orientation of~$\cC$, and for each pair of distinct points
$x,y\in\cC$, write $[x,y]$ for the arc in~$\cC$ with initial and final
endpoints~$x$ and~$y$. 

A {\em segment} in~$\cA$ is a set of the form
\[\lseg x,y\rseg  := \{\gamma\in\cA\,:\,\gamma\subset [x,y]\},\]
where $x$ and $y$ are distinct points of~$\cC$ which do not lie in the
interior of any element of~$\cA$. A segment is {\em non-trivial} if it
has more than one element. Notice that each segment is endowed with a
total order~$<$ induced by the orientation on~$\cC$. Similarly, given
distinct elements~$\gamma$ and~$\delta$ of~$\cA$, write $\lseg
\gamma, \delta\rseg $ for the segment~$\lseg x,y\rseg $,
where~$x$ is the initial point of~$\gamma$ and~$y$ is the final point
of~$\delta$, provided that these points do not coincide; and write
$[\gamma,\delta]$ for the arc $[x,y]$ in~$\cC$.

The set~$\cA$ has the following properties:
\begin{enumerate}[i)]
\item (The elements of~$\cA$ are maximal.) There are no two
  non-trivial disjoint segments $\lseg x,y\rseg $ and $\lseg
  x',y'\rseg $ with the properties that: $\gamma\in\lseg
  x,y\rseg$ if and only if $\phi(\gamma)\in \lseg x',y'\rseg $;
  and $\gamma < \delta$ in $\lseg x,y\rseg $ if and only if
  $\phi(\gamma) > \phi(\delta)$ in $\lseg x',y'\rseg $. (If there
  were such segments, then arcs in each segment could be coalesced to
  give larger unlinked arcs.)
\item(There are no plain arcs.) Let~$\gamma\in\cA$, and suppose that
  $\phi(\delta)\in \lseg \gamma, \phi(\gamma)\rseg $ whenever
  $\delta\in \lseg \gamma, \phi(\gamma)\rseg $. Then there are
  $\alpha, \beta\in \lseg \gamma, \phi(\gamma)\rseg $ such that
  $\alpha<\beta<\phi(\alpha)<\phi(\beta)$ with respect to the order on
  $\lseg\gamma,\phi(\gamma)\rseg$. (Otherwise $[\gamma,
    \phi(\gamma)]$ would be a plain arc.)
\end{enumerate}

The existence of the maximal unlinked arcs~$\alpha$ and~$\beta$ in the
statement of the lemma is immediate from property~ii). Let~$\gamma$ be
any element of~$\cA$. If there is some~$\delta\in\cA$ with $\delta \in
\lseg \gamma, \phi(\gamma)\rseg $ but $\phi(\delta)\not\in\lseg
\gamma, \phi(\gamma)\rseg $, then set $\alpha = \gamma$ and $\beta =
\delta$. If there is no such~$\delta$, then suitable arcs~$\alpha$
and~$\beta$ exist by property~ii).

It will be assumed in the remainder of the proof that the arcs
$\alpha$, $\alpha'$, $\beta$, and~$\beta'$ have no common endpoints:
only minor modifications are required in the case where they share
some endpoints.

Carrying out the identifications corresponding to these arcs clearly
yields a compact surface~$\hP$ with genus one greater than that
of~$P$, having a single boundary component~$\hcC$. Let $\pi\co \cC
\setminus \Int(\alpha\cup\alpha'\cup\beta\cup\beta') \to \hcC$ be the
natural projection: notice that~$\pi$ does not respect the cyclic
orders around~$\cC$ and~$\hcC$. The image under~$\pi$ of the endpoints
of the identified arcs consists of four points $p_1, p_2, p_3, p_4\in
\hcC$. Denote by $\hcC_1$, $\hcC_2$, $\hcC_3$, and $\hcC_4$ the four
components of~$\hcC \setminus\{p_1,p_2,p_3,p_4\}$; and let~$\hcA = \cA
\setminus\{\alpha, \beta, \alpha', \beta'\}$. Notice that the $\hcC_i$
can be regarded as segments in both $\cC$ and $\hcC$, and that each is
totally ordered by the orientation of~$\cC$ (and consistently by the
induced orientation of~$\hcC$). If the $\hcC_i$ are labelled so that
they appear in the order $\hcC_1, \hcC_2, \hcC_3, \hcC_4$
around~$\cC$, then they appear in the reverse order $\hcC_4,
\hcC_3, \hcC_2, \hcC_1$ around~$\hcC$.

If $\gamma\in\hcA$ then $\pi(\gamma)$ is an unlinked arc, with pair
$\pi(\phi(\gamma))$, but it need not be maximal. However, it will be
shown that there are only finitely many~$\gamma\in\hcA$ for
which~$\pi(\gamma)$ is not maximal, which will establish that~$\hcP$
is not finitely linked, as required.

To prove this, suppose that~$\Gamma$ is a subset of~$\hcA$ such that
$\overline{\bigcup_{\gamma\in\Gamma}\pi(\gamma)}$ is an unlinked
arc. Suppose that $\Gamma$ contains two arcs $\gamma$ and $\delta$
such that $\pi(\gamma) < \pi(\delta)$ lie in a single~$\hcC_i$, and
$\pi(\phi(\gamma))$ and $\pi(\phi(\delta))$ also lie in a
single~$\hcC_j$. Then the segment $\lseg \gamma, \delta\rseg $ would
be unlinked, contradicting the assumption that the elements of~$\cA$
are maximal unlinked arcs. Hence~$\Gamma$ contains at most three arcs,
which are adjacent, and there is some~$\gamma\in\Gamma$ such that
either $\pi(\gamma)$ or $\pi(\phi(\gamma))$ contains one of the
points~$p_i$. There are therefore at most~24 elements $\gamma$
of~$\hcA$ for which $\pi(\gamma)$ is not maximal, as required.

It remains to show that~$\hcC$ contains no plain arcs. Suppose for a
contradiction that it does contain some plain
arc~$\delta$. Let
\[\Delta = \{\gamma\in\hcA\,:\, \pi(\gamma)\subset\delta\},\]
and observe that $\phi(\Delta)=\Delta$ since~$\delta$ is
saturated. Moreover, if $\gamma\in\Delta$ then $\pi(\gamma)$ and
$\pi(\phi(\gamma))$ lie in different components~$\hcC_i$, as otherwise
either $[\gamma, \phi(\gamma)]$ or $[\phi(\gamma), \gamma]$ would be a
plain arc in~$\cC$. In particular, $\delta$ contains at least one of
the points~$p_i$ in its interior.

If~$\Delta$ is infinite, then pick an infinite segment~$\lseg
x,y\rseg\subset\Delta$ contained in a
single~$\hcC_i$. Since~$\delta$ is plain, the arcs
$\{\phi(\gamma)\,:\,\gamma\in \lseg x,y\rseg\}$ lie in the union
of the other three components in exactly the opposite order to their
pairs in~$\lseg x,y\rseg$. Picking a subsegment so that the
$\phi(\gamma)$ lie in a single component gives a contradiction to
property~i) above.

On the other hand, if~$\Delta$ is finite then it contains an arc
$\gamma$ such that~$\pi(\gamma)$ and $\pi(\phi(\gamma))$ share a
common endpoint, which must be one of the points~$p_i$. Suppose,
without loss of generality, that the common endpoint is~$p_1$, which
is the image under~$\pi$ of the initial point of~$\alpha$ and the
final point of~$\phi(\alpha)$. If the endpoint of~$\gamma$
(respectively~$\phi(\gamma)$) coincides with the initial point
of~$\alpha$ then the segment~$\lseg \gamma,\alpha\rseg$
(respectively $\lseg \phi(\gamma),\alpha\rseg$) gives a
contradiction to property~i) above.
\end{proof}

\begin{thm}
\label{thm:top-structure}
A connected paper-folding scheme~$(P,\cP)$ is a surface paper folding
scheme if and only if it is finitely linked. In this case, the
scar~$G$ is a local dendrite, which can be written as $G=C\cup\Gamma$,
where
\begin{itemize}
\item $C$ is a finite connected graph in~$S$ with the property that
  any simple closed curve in~$C$ is homotopically non-trivial; and
\item $\Gamma$ is a union of finitely or countably many disjoint
  dendrites, with diameters decreasing to~$0$, each of which
  intersects~$C$ exactly once.
\end{itemize}
\end{thm}

\begin{proof}
Suppose first that~$(P,\cP)$ is finitely linked. The
identifications~$\sim_\cP$ on~$\partial P$ will be done in two
steps. First those arising from pairings contained in plain arcs are
carried out: applying Moore's theorem as in the proof of
Theorem~\ref{thm:plain-top-structure}, these yield a topological
sphere on which the simple closed curve~$\cC$ bounds a disk whose
interior contains only points of~$\csph\setminus P$. Removing this
disk and carrying out the remaining (finitely many) identifications
along~$\cC$ yields a topological surface. The details are similar to
those of the proof of Theorem~\ref{thm:plain-top-structure}, and will
only be sketched. 

As in that proof, regard~$\partial P$ as the unit circle in~$\csph$,
so that~$P$ is the exterior of the unit disk~$\bDD$, which is regarded
as the hyperbolic plane. Connect the points of each interior pair
in~$\partial P$ which is contained in a plain arc with the geodesic
joining them, and include also the geodesic $g_\gamma$ joining the
endpoints of each maximal plain arc~$\gamma$. Let~$\cL$ be the
geodesic lamination obtained by taking the closure of the union of
these geodesics. For each maximal plain arc~$\gamma$, let~$D_\gamma$
denote the disk bounded by $\gamma\cup g_\gamma$.

The musc decomposition~$\cG$ of~$\csph$ with the following elements
then realizes the equivalence relation~$\sim_{\text{Plain}}$
on~$\partial P$ corresponding to the set of pairings contained in
plain arcs:
\begin{enumerate}[I)]
\item Closures of components of $D_\gamma \setminus\cL$, for each
  maximal plain arc~$\gamma$;
\item Geodesics in~$\cL$ (together with endpoints) which are not
  contained in elements of type~I); and
\item Points which are not contained in elements of types~I) and~II).
\end{enumerate}

The quotient of~$\csph$ by $\sim_{\text{Plain}}$ is therefore a
topological sphere containing the simple closed curve~$\cC$ obtained
from~$\partial P$ by collapsing each maximal plain arc to a point. One
complementary component of~$\cC$ is (the projection of)
$\csph\setminus\left(P\cup\bigcup_{\gamma\text{
    plain}}D_\gamma\right)$, while the other contains (the projection
of) the interior of~$P$, together with a dendrite corresponding to
each maximal plain arc~$\gamma$, attached to~$\cC$ at a single point.
Removing the former complementary component yields a topological disk
bounded by~$\cC$, and carrying out the (finitely many) identifications
on~$\cC$ gives the compact surface~$S$: the bounding curve~$\cC$
projects to a finite graph~$C$ in~$S$ in which every simple closed
curve is homotopically non-trivial, and the scar~$G$ consists of this
graph together with the dendrites arising from maximal plain arcs.

For the converse, suppose that~$(P, \cP)$ is not finitely
linked. Carry out the identifications realizing~$\sim_{\text{Plain}}$
as above, and remove the complementary component of~$\cC$
corresponding to the exterior of~$P$. Using
Lemma~\ref{lem:infinitely-linked}, it is possible to carry out
repeated additional pairs of identifications, each of which adds~1 to
the genus of the resulting surface. It follows that~$S$ contains a
surface (with boundary) of genus~$n$ for all~$n$, so that~$S$ cannot
be a compact surface as required.
\end{proof}

The injectivity radius of~$G$, which gives the length of the shortest
simple closed curve in the scar, will be important in
Section~\ref{sec:cxstructure}. 

\begin{defn}[Injectivity radius of $G$]\label{defn:injradius}
The {\em injectivity radius} $\br$ of~$G$ is defined by
\[\br = \frac{1}{2}\, \inf\{|\gamma|_G\,:\, \gamma \text{ is a simple
  closed curve contained in }G\}.\]
\end{defn}

\begin{lem}\label{lem:injradius}
Let~$G$ be the scar of a finitely linked paper-folding scheme. Then
the injectivity radius of $G$ satisfies $\br>0$. For every $x\in G$
and every $0\leq r<\br$, the closed ball $\bB_G(x;r)$ is a dendrite
and is thus contractible.
\end{lem}
\begin{proof}
$\br>0$ since the graph~$C$ of Theorem~\ref{thm:top-structure} is
  finite, and hence there are only finitely many simple closed curves
  in~$G$.

Let~$r<\br$. Then $\bB_G(x;r)$ is a continuum (it is connected since
the metric~$d_G$ is intrinsic) which contains no simple closed
curves. Each dendrite in the decomposition of~$G$ intersects
$\bB_G(x;r)$ in a subcontinuum, and hence in a dendrite
(Theorem~\ref{thm:pprtsdendrites}e)); and similarly~$C$
intersects~$\bB_G(x;r)$ in a tree. Hence any two distinct points
of~$\bB_G(x;r)$ are separated by a third point of~$\bB_G(x;r)$,
establishing by Theorem~\ref{thm:pprtsdendrites}a) that it is a
dendrite as required.
\end{proof}

\subsection{The metric structure of paper spaces}
\label{subsec:metric-structure}

Let~$S$ be a paper space. Near planar points,~$S$ is Euclidean: that
is, such points have an open neighborhood which is isometric to an
open ball in~$\R^2$. Similarly, near a regular vertex~$x$, $S$ is
isometric to the apex of a cone, with cone angle equal to the sum of
the internal angles of the multipolygon~$P$ at the points
of~$\pi^{-1}(x)$. The following definitions formalize this
statement. Throughout this section $(P,\cP)$ is an arbitrary
paper-folding scheme with scar~$G$ and associated paper space~$S$.

\begin{defn}[Cone angle at non-singular points]
Let~$x\in G$ be non-singular with $\pi^{-1}(x) =
\{\tx_1,\ldots,\tx_k\}$, and let~$\teta_i$ be the internal angle
of~$P$ at~$\tx_i$. The {\em cone angle} at~$x$ is the number $\eta(x)
= \teta_1 + \cdots + \teta_k$.
\end{defn}

In particular, the cone angle at a regular point~$x\in G$ is $\eta(x)=2\pi$.

\begin{defn}[Metric cone: see~\cite{BBI}]
 Let $X$ be a topological space. The {\em cone}
  $\opna{Cone}(X)$ over $X$ is the (topological) quotient of
  $[0,\infty)\times X$ by the equivalence relation which collapses
    $\{0\}\times X$ to a point, that is, whose only non-trivial class
    is $\{0\}\times X$. This point in the quotient is the {\em origin}
    or {\em apex} of the cone and is also denoted~$0$. If $(X,d)$ is a
    metric space, then it is possible to make $\opna{Cone}(X)$ into a
    metric space by defining a distance as follows: if $p=[t,x],
    q=[s,y]\in\Cone(X)$, set
\[d_c(p,q)=
\begin{cases}
\sqrt{t^2+s^2-2st\cos\left(d(x,y)\right)}, &\text{ if } d(x,y)\leq\pi \\
t+s, &\text{ if } d(x,y)\geq\pi.
\end{cases} \]
\end{defn}

Let $S^1_r$ denote the circle of radius $r>0$ in~$\R^2$ with the
intrinsic metric. Then the cone $\opna{Cone}(S^1_r)$ is locally
isometric to the plane itself, except at the apex for $r\neq 1$, and
$\opna{Cone}(S^1_1)$ is globally isometric to $\plane$. Moreover, two
cones on circles $\Cone(\sone_{r_1}),\Cone(\sone_{r_2})$ are globally
isometric only if $r_1=r_2$.

\begin{defn}[Conic-flat surface]
A {\em conic-flat surface} is a metric space which is locally
isometric to cones on circles: for every point~$x$, there
exist~$r,\veps>0$ and $c\in \Cone(S^1_r)$ such that $B(x;\veps)$ is
isometric to $B_{\Cone(S^1_r)}(c;\veps)$. 
\end{defn}

\begin{thm}[Metric structure of paper spaces]\mbox{}
\label{thm:metric-structure}
\begin{enumerate}[a)]
\item Let~$x$ be a non-singular point of~$G$. Then there exists~$r>0$ such
that $B_S(x;r)$ is isometric to
$B_{\Cone(S^1_{\eta(x)/2\pi})}(0;r)$. In particular, $S\setminus\cV^s$
is a conic-flat surface.
\item The metric~$d_S$ on a paper space~$S$ is strictly intrinsic. 
\end{enumerate}
\end{thm}

\begin{proof}
The proof of Part~a) is straightforward and technical and is omitted:
this statement is not used in the remainder of the paper. Part~b) is
immediate from Lemma~\ref{lem:intrinsic-quotient-intrinsic} and
Theorem~\ref{thm:compact-intrinsic-strict} since the metric~$d_P$
on~$P$ is intrinsic, and~$P$ (and hence~$S$) is compact.
\end{proof}

Notice that even if~$S$ is a paper surface, the metric on the
conic-flat surface $S\setminus\cV^s$ need not be strictly intrinsic.

\section{The conformal structure on paper surfaces}
\label{sec:cxstructure}

If $(P,\cP)$ is a surface paper-folding scheme and~$S$ is the quotient
paper surface, then there is a natural conformal structure on the set
$\Planar(S)$ of planar points coming from the local Euclidean
structure. The question addressed in this section is whether or not
this complex structure extends uniquely across non-planar points
of~$S$.

The complex structure extends readily across regular vertices of~$G$
using the conic-flat structure on~$S\setminus\cV^s$ described in
Section~\ref{subsec:metric-structure}: at a cone point of
angle~$\eta$, the map $z\mapsto z^{2\pi/\eta}$ can be used to
introduce conformal coordinates. Thus the case of interest is that of
isolated singular points.  Theorem~\ref{thm:main} below provides a
criterion for the complex structure to extend uniquely across such a
point. In particular, if all singular points are isolated, and this
criterion holds at each of them, then~$S$ is a Riemann
surface. Similar conditions which guarantee that the complex structure
extends uniquely across a more general singular set can also be
obtained, and this will be the subject of a forthcoming paper.

The question of whether the complex structure extends uniquely across
an isolated singular point is clearly a local one, at least as far
as~$S$ is concerned. However, both the results and the techniques of
this section will be central in Section~\ref{sec:modcont}, where a
global modulus of continuity for a uniformizing map is obtained. A
more global approach is therefore taken than is necessary for the
results of this section alone.

Again, the local nature of the problem means that there is no
essential distinction between plain and non-plain surface
paper-foldings. For simplicity of exposition, however, the details of
the construction and proof (from
Section~\ref{subsec:cxstruct-foliatedcollar} onwards) will be carried
out in the plain case: the minor modifications needed for non-plain
examples will be described at the end of
Section~\ref{subsec:cxstruct-proof}.

The main theorem is stated in Section~\ref{subsec:cxstruct-statement},
and examples of its application are given in
Section~\ref{subsec:cxstruct-example}. The idea of the proof (which
was inspired by similar constructions in~\cite{EaGa,ugpa,eqgpa};
see also~\cite{FNR}) is to construct a nested sequence of annuli
with divergent module sum zooming down to the singular point~$q$, and
to apply Lemmas~\ref{lem:conformalpnct} and~\ref{lem:summod}. In order
to be able to construct these annuli in a way which makes it possible
to estimate their modules, a {\em foliated collaring} of the
polygon~$P$ is used: this is described in
Section~\ref{subsec:cxstruct-foliatedcollar}. The annuli themselves
are constructed in Section~\ref{subsec:cxstruct-annuli}, and some
technical lemmas needed to estimate their modules are given in
Section~\ref{subsec:cxstruct-geometry-polygon-constants}, before the
proof of Theorem~\ref{thm:main} is given in
Section~\ref{subsec:cxstruct-proof}.

\subsection{Statement of results}
\label{subsec:cxstruct-statement}

\begin{defns}[$m(q;r)$, $n(q;r)$, planar radius]
\label{defn:planarradius} 
Let~$G$ be the scar of a paper-folding scheme~$(P,\cP)$, $q\in G$, and
$r>0$. Recall (Notation~\ref{notation:ballsandcircles}) that
$C_G(q;r)$ denotes the set of points of~$G$ at $d_G$-distance
exactly~$r$ from~$q$. Define
\[n(q;r) := \#C_G(q;r) \in\N\cup\{\infty\}.\]
A radius~$r>0$ is said to be {\em planar for~$q$}, or {\em
  $q$-planar}, if all points of~$C_G(q;r)$ are planar.

Define also
\[m(q;r) := \rmm_G(B_G(q;r)) \in (0,\rmm_G(G)].\]
\end{defns}

In the statement and proof of the following result, recall
(Definition~\ref{defn:injradius} and Lemma~\ref{lem:injradius})
that~$\br$ denotes the injectivity radius of~$G$, and that
$\bB_G(q;r)$ is a dendrite for all $q\in G$ and all $r<\br$.

\begin{lem}\label{lem:planarradius}
If $r\in(0,\br)$ is a planar radius for $q\in G$, then $n(q;r)$ is
finite and is locally constant on both variables $q,r$. Moreover,
given $q\in G$, the set of radii which are not planar for $q$ is a
closed subset of~$(0,\br)$ of measure zero.
\end{lem}
\begin{proof}
Recall (Lemma~\ref{lem:metricstructureofG}a)) that the set~$\ol\cV$ of
non-planar points is a compact subset of~$G$ with zero
$\rmm_G$-measure. If $r>0$ is planar for~$q$ there is therefore
some~$\delta>0$ such that $|d(s,q)-r|>\delta$ for all $s\in\ol\cV$.

It follows that for each $x\in C_G(q;r)$, the ball $I_x=B_G(x;\delta)$
is isometric to an interval of length~$2\delta$. Since~$\bB_G(q;r)$ is a
dendrite and the metric~$d_G$ is strictly intrinsic, $I_x$ cannot
contain any other point of~$C_G(q;r)$ (otherwise the unique shortest
path from~$q$ to one of the points would be strictly shorter than that
to the other). This establishes that~$n(q;r)$ is finite,
since~$\rmm_G(G)$ is finite.

For the local constantness of~$n(q;r)$, observe that if
$d_G(q,q')<\delta/4$ and $|r-r'|<\delta/4$, then each~$I_x$ contains a
point of~$C_G(q';r')$, and hence $n(q';r')\ge n(q;r)$. On the other
hand, since $|d(s,q')-r'|>\delta/2$ for all $s\in\ol\cV$, the same
argument works the other way round to show that $n(q;r)\ge n(q';r')$
as required.

$\ol\cV$ has zero $\rmm_G$-measure, and hence zero $\mu^1_G$-measure
by Lemma~\ref{lem:metricstructureofG}c). Since the function $G\to\R$
defined by $s\mapsto d_G(q,s)$ is distance non-increasing, it follows
from Lemma~\ref{lem:hausdorff-nonincreasing} that the set $\{d_G(q,s)\,:\,
s\in\ol\cV\}$ of non-planar radii for~$q$ has zero measure. 

Since $\ol\cV$ is compact and $x\mapsto d_G(q,x)$ is a continuous map
$G\to\R$, the set of non-planar radii together with zero is closed
in~$\R$, and hence the set of non-planar radii is closed in $(0,\br)$.
\end{proof}

It is now possible to state one of the main theorems of this paper:
\begin{thm}
\label{thm:main}
Let~$(P,\cP)$ be a surface paper-folding scheme with associated paper
surface~$S$ and scar~$G\sbs S$.  If~$q\in G$ is an isolated singular
point, then the complex structure on $S\setminus\cV^s$ extends uniquely
across~$q$ provided that
\begin{equation}\label{eq:divint1}
 \int_0 \frac{\rmd r}{m(q;r)+r\cdot n(q;r)} =
 \infty.
\end{equation}
In particular, if all singular points are isolated and the integral
condition above holds for every one of them, then $S$ is a compact
Riemann surface.
\end{thm}

\begin{rmk}
\label{rmk:divergence}
The integral in the statement diverges at planar points in~$G$ and at
regular vertices. At a regular $k$-vertex with $k<\infty$, or at a
planar point ($k=2$), $n(q;r)=k$ and $m(q;r)=2kr$ for all
sufficiently small~$r$. Therefore \[\int_{\veps} \frac{\rmd
  r}{m(q;r)+r\cdot n(q;r)}\] goes like~$-\ln
\veps$.  The proof of Theorem~\ref{thm:main} presented in this
section also shows that if $q\in G$ is a non-singular point, then
the complex structure on $S\setminus\left(\cV^s\cup\{q\}\right)$
extends uniquely across~$q$.
\end{rmk}

\subsection{Example: $\infty$-od singularities}
\label{subsec:cxstruct-example}
A case of special interest is that of $\infty$-od singularities: they
have already appeared in Example~\ref{ex:generalexample} (see also
Example~\ref{ex:HorseshoeCantor}), and are common in dynamical applications.

\begin{defn}[$\infty$-od singularity]
An isolated singularity~$q\in G$ is an {\em $\infty$-od singularity}
if there is some~$r_0>0$ such that $\bB_G(q;r_0)$ is an $\infty$-od
(Definition~\ref{defn:infty-od}). 
\end{defn}

Suppose that $q\in G$ is an $\infty$-od singularity. Then
if $0<r\leq r_0$, every point $x\in C_G(q;r)$ is joined to $q$ by
a unique arc $\ol{xq}$ all of whose interior points are planar and
whose length is $r$, and these arcs intersect only at~$q$. By
Lemma~\ref{lem:metricstructureofG}d), $\rmm_G(\ol{xq})=2r$. There are
as many such arcs in $\bB_G(q;r)$ as there are points in $C_G(q;r)$,
from which it follows that $m(q;r)\geq 2r\cdot n(q;r)$. This
proves the following
\begin{cor}\label{cor:inftyod}
For the complex structure on $\cR$ to extend uniquely across an
$\infty$-od singularity $q\in G$ it is sufficient that
\begin{equation}\label{eq:divint2}
\int_0\frac{\rmd r}{m(q;r)}=\infty.
\end{equation}
\qed
\end{cor}

\begin{example}\label{ex:both} 
Consider Example~\ref{ex:generalexample}. Notice that in the criterion
of Corollary~\ref{cor:inftyod} only the measure
$m(q;r)=\rmm_G(B_G(q;r))$ is taken into account and the way in which
the identifications about the $\infty$-od singularity~$q$ are done is
not relevant. Thus there is no distinction between the two ways of
arranging the folds in the example, since the resulting scars are
isometric. What is important is the asymptotics of the decreasing
sequence~$(a_n)$ of fold lengths.

Write $N(r):=n(q;r) = \max\{n:\; a_n\ge r\}$ for $r>0$. Then
\begin{equation}\label{eq:measbg}
m(q;r)= 2r \cdot N(r) + 2\sum_{n> N(r)}a_n.
\end{equation}

Suppose first that $a_n\asymp 1/{n^k}$, for some $k>1$: that is,
$C_1/n^k \le a_n\le C_2/n^k$ for some positive constants~$C_1$
and~$C_2$.  In this case, $n\le (C_1/r)^{1/k}$ implies $a_n\ge r$, and
$n>(C_2/r)^{1/k}$ implies $a_n<r$, so that $(C_1/r)^{1/k} \le N(r) \le
(C_2/r)^{1/k}$. Using~(\ref{eq:measbg}) with the lower bound
for~$N(r)$ in the first term, the upper bound for~$N(r)$ in the
second, and lower bounds for~$a_n$ gives $m(q;r) \ge C_3 r^{1-1/k}$
for sufficiently small~$r$, so that
\[\int_0\frac{\rmd r}{m(q;r)}\le 
\int_0\frac{\rmd r}{C_3r^{1-1/k}}<\infty.\] 
This means the criterion
cannot be used to guarantee that the complex structure extends
uniquely across the point~$q$, and the authors know no way of
determining whether or not it does so extend.

\smallskip
If, on the other hand, $a_n\asymp 1/\lambda^n$ for some $\lambda>1$,
then 
\[\frac{\ln(C_1/r)}{\ln\lambda} \le N(r) \le \frac{\ln(C_2/r)}{\ln\lambda},\]
and
it follows from~(\ref{eq:measbg}) that
\begin{eqnarray*}
\frac{1}{2}m(q;r)
  &\le&  \frac{r}{\ln\lambda}\ln\frac{C_2}{r} + \frac{C_2}{\lambda^{N(r)}(\lambda-1)}\\
&\le&\frac{r}{\ln\lambda}\left(\ln C_2 + \ln\frac{1}{r}\right) +
\frac{C_2r}{C_1(\lambda-1)} \\
&\le& C_3r\ln\frac{1}{r} \qquad\text{ for $r$ sufficiently small,}
\end{eqnarray*}
so that
\[2\int_0\frac{\rmd r}{m(q;r)}\ge
\int_0\frac{\rmd r}{C_3r\ln (1/r)}=\infty.\] This time
Corollary~\ref{cor:inftyod} applies and thus the quotient space is a
complex sphere.

\end{example}
\begin{rmk}\label{rmk:Cantor}
If the Cantor construction of Example~\ref{ex:generalexample} is that of the
standard middle thirds Cantor set, then listing the edges of the
$\infty$-od in decreasing order gives
\[a_n= \frac{1}{6\cdot 3^k}\;\quad\text{for}\quad 2^k\le n <
2^{k+1},\] 
which implies that
\[\frac{1}{6}\left(\frac{1}{n}\right)^{\frac{\ln 3}{\ln 2}}\le a_n\le 
\frac{1}{2}\left(\frac{1}{n}\right)^{\frac{\ln 3}{\ln 2}}.\]
Thus this is an example in which the hypothesis of
Corollary~\ref{cor:inftyod} is not satisfied.
\end{rmk}

\subsection{Foliated collaring of~$P$}
\label{subsec:cxstruct-foliatedcollar}
Let~$(P,\cP)$ be a surface paper-folding scheme with associated paper
surface~$S$ and scar~$G$, and let~$q$ be an isolated singular point
of~$G$. In order to show that the complex structure on
$S\setminus\cV^s$ extends uniquely across~$q$ under appropriate
conditions, a nested sequence of annuli zooming down to~$q$ will be
defined. These annuli will be constructed in a foliated neighborhood
of~$q$ arising from a foliated collar of~$P$, which is described in
this section.

For simplicity of exposition, the case in which~$(P,\cP)$ is plain
will be considered first. Thus~$P$ is a single polygon, $G$ is a
dendrite, and~$S$ is a topological sphere
(Definitions~\ref{defns:linked-plain} and
Theorem~\ref{thm:plain-top-structure}). 

The collar~$\tQ$ of~$P$ is constructed as a union of trapezoids whose
bases are the sides of~$P$; whose vertical sides bisect the angles at
the vertices of~$P$; and which have fixed height~$\bh$, chosen small
enough that the trapezoids are far from degenenerate and intersect
only along their vertical sides. It has a {\em horizontal foliation}
by leaves parallel to~$\partial P$, and a {\em vertical foliation} by
leaves joining the base and the top of each trapezoid: see
Figure~\ref{fig:collar}. The following paragraphs define the collar
and foliations more carefully, and set up the notation which will be
used.

\begin{figure}[htbp]
\lab{e0}{\te_0}{}\lab{X}{\te'_0}{}\lab{e1}{\te_1}{}\lab{ei}{\te_{i}}{}
\lab{en1}{\te_{n-1}}{}
\lab{v0}{\tv_0}{}\lab{v1}{\tv_1}{}\lab{v2}{\tv_2}{}\lab{vn1}{\tv_{n-1}}{}
\lab{vi1}{\tv_{i}}{}\lab{vi}{\tv_{i+1}}{}
\lab{P}{P}{}
\lab{Q}{\tQ_0}{}
\lab{ti1}{\ttheta_{i}}{l}
\lab{Qi}{\tQ_i}{r}
\lab{x}{\tx}{}\lab{verx}{\tver(\tx) = \tver(t)}{l}\lab{gxh}{\tgamma(t,h)}{l}
\lab{h}{h}{}\lab{bh}{\bh}{}\lab{Y}{\te'_i}{}\lab{eh}{\te_i(h)}{l}
\lab{t1}{t_i}{}
\lab{tx}{t}{}
\lab{t2}{t_{i+1}}{}
\lab{...}{\ldots}{}
\begin{center}
\pichere{1.0}{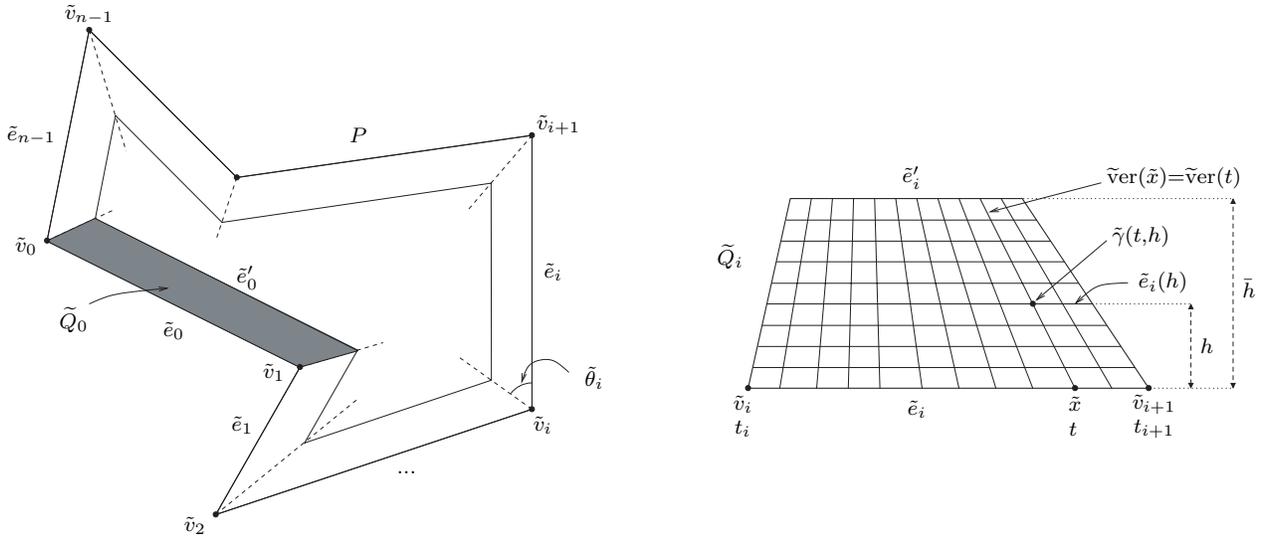}
\end{center}
\caption{The collar~$\tQ$ and its foliations}
\label{fig:collar}
\end{figure}

\subsubsection{The collar~$\tQ$}
\label{sec:the-collar-tq}

Choose a labeling~$\tv_i$ ($i=0,\ldots,n-1$) of the vertices of~$P$
listed counterclockwise around~$\partial P$, and let~$\te_i$ be the
side of~$P$ with endpoints~$\tv_i$ and~$\tv_{i+1}$ (here and
throughout, subscripts on cyclically ordered objects are taken~$\mod
n$). Write~$L=|\partial P|$, and
let~$\tgamma_0\co[0,L)\to\partial P$ be the order-preserving
parameterization of~$\partial P$ by arc-length with
$\tgamma_0(0)=\tv_0$. Denote by~$t_i\in[0,L)$ the parameter with
$\tgamma_0(t_i)=\tv_i$. 

A {\em trapezoid} is a quadrilateral in~$\R^2$ with two parallel
sides, which are called its {\em base} and its {\em top}: the other
sides are called {\em vertical sides}. The {\em height} of the
trapezoid is the distance between the parallel lines containing its
base and its top.

Pick a height~$\bh$ small enough that the trapezoids~$\tQ_i$ which
have bases~$\te_i$, heights~$\bh$, and vertical sides along the rays
bisecting the internal angles of~$P$ satisfy:
\begin{enumerate}[a)]
\item The lengths of the tops of the trapezoids are between half and
  twice the lengths of their bases; and
\item The trapezoids intersect only along their vertical sides.
\end{enumerate}
 This height~$\bh$ is an important quantity in the construction, and
 will remain fixed throughout the remainder of the section.  Denote
 the top of~$\tQ_i$ by $\te_i'$, and let~$\ttheta_i$ be half of the
 internal angle of~$\partial P$ at~$\tv_i$: thus the internal angles
 of~$\tQ_i$ at the endpoints of its base are $\ttheta_i$ and
 $\ttheta_{i+1}$. Condition~a) above is that
\begin{equation}
\label{eq:K}
\frac{1}{2} \le \frac{|\te_i'|}{|\te_i|} \le 2 \qquad \text{for }i=0,\ldots,n-1.
\end{equation}

Let \[\tQ = \bigcup_{i=0}^{n-1} \tQ_i,\]
a closed collar neighborhood of~$\partial P$ in~$P$.

\subsubsection{The foliations $\tHor$ and $\tVer$}
\label{sec:foliations}
For each $h\in[0,\bh]$, let $\te_i(h)\subset\tQ_i$ be the segment
parallel to the base of~$\tQ_i$ at height~$h$, so that
$\te_i(0)=\te_i$ and $\te_i(\bh)=\te_i'$. Then the union $\thor(h)$ of
these segments is a polygonal simple closed curve: these simple closed
curves are the leaves of the {\em horizontal foliation}
\[\tHor = \left\{\thor(h)\,:\, h\in[0,\bh]\right\}\]
of~$\tQ$. The parameter~$h$ is called the {\em height} of the leaf
$\thor(h)$. 

Write
\[\tQ(h) = \bigcup_{h'\in[0,h]}\thor(h'),\] 
the subset of~$\tQ$ consisting of leaves with heights not
exceeding~$h$: $\tQ(h)\subset \tQ = \tQ(\bh)$ is therefore also a
closed collar neighborhood of~$\partial P$ for each~$h\in(0,\bh]$.

To construct the vertical foliation, let~$\varphi_i\co\te_i\to\te_i'$
be the orientation-preserving scaling from~$\te_i$ to
$\te_i'$. For each $\tx = \tgamma_0(t) \in \te_i$, denote by
$\tver(\tx)$ or $\tver(t)$ the straight line segment which joins $\tx$
to $\varphi_i(\tx)$. These segments are the leaves of the {\em
  vertical foliation}
\[\tVer = \left\{\tver(t)\,:\, t\in[0,L)\right\}\]
of $\tQ$. 

Define $\ttheta\co[0,L)\setminus\{t_0,\ldots,t_{n-1}\} \to(0,\pi)$ by
  setting $\ttheta(t)$ to be the angle between $\partial P$ and
  $\tver(t)$ at $\tgamma_0(t)$: that is, the angle between the
  oriented side of $\partial P$ containing $\tgamma_0(t)$ and the leaf
  $\tver(t)$ pointing into~$P$. This function has well-defined limits
  as~$t$ approaches each~$t_i$ from the left or the right:
  $\ttheta(t_i^-)=\pi-\ttheta_i$, and $\ttheta(t_i^+) =
  \ttheta_i$. The notation $\ttheta(\tx)=\ttheta(t)$ will also be used
  when $\tx = \tgamma_0(t)$.

The foliations $\tHor$ and $\tVer$ yield a parameterization
\[\tgamma\co [0,L) \times [0,\bh] \to \tQ\]
of $\tQ$, where $\tgamma(t,h)$ is the unique point of $\tver(t) \cap
\thor(h)$.

For each~$h\in(0,\bh]$, denote by $\tpsi_h\co \tQ(h)\to\partial P$ the
retraction of $\tQ(h)$ onto $\partial P$ which slides each point along
its vertical leaf:
\[
\tpsi_h(\gamma(t,h')) = \gamma(t,0) \qquad \text{(all $t\in[0,L)$ and
    $h'\in[0,h]$)}.
\]
In particular, $\tpsi_{\bh}$ is a retraction which squashes all
of~$\tQ$ onto~$\partial P$.

\subsubsection{The foliations on~$S$}
The projections to the paper surface~$S$ of the structures defined
above are denoted by removing tildes (see
Figure~\ref{fig:foliations}). Thus $Q = \pi(\tQ)$ is a closed disk
neighborhood of the scar~$G$, and similarly $Q(h) = \pi(\tQ(h))$ is a
closed subdisk neighborhood for each $h\in(0,\bh]$.  $Q$~has
horizontal and vertical foliations $\Hor = \pi(\tHor)$ and $\Ver =
\pi(\tVer)$. The leaves of~$\Hor$ are projections of leaves
of~$\tHor$: $\hor(h) = \pi(\thor(h))$. The leaves of $\Ver$, however,
are unions of projections of leaves of~$\tVer$: for each~$x\in G$, the
leaf of~$\Ver$ containing~$x$ is defined to be
\[\ver(x) := \bigcup \left\{ \pi(\tver(\tx))\,:\,\tx\in\pi^{-1}(x)
\right\}. \]
Thus $\ver(x)$ is an arc if and only if $\#\pi^{-1}(x)\le 2$. If~$x$ is
a $k$-vertex for $k>2$ then $\ver(x)$ is a star with~$k$
branches. Note, however, that if~$x$ is an $\infty$-vertex then
$\ver(x)$ is not an $\infty$-od in the sense of
Definition~\ref{defn:infty-od}, since the lengths of its branches do
not converge to zero.

\begin{figure}[htbp]
\lab{G}{G}{b}\lab{x}{x}{b}\lab{vx}{\ver(x)}{t}\lab{hh}{\hor(h)}{t}
\lab{Q}{Q}{b}
\lab{hbh}{\hor(\bh)}{b}
\begin{center}
\pichere{0.65}{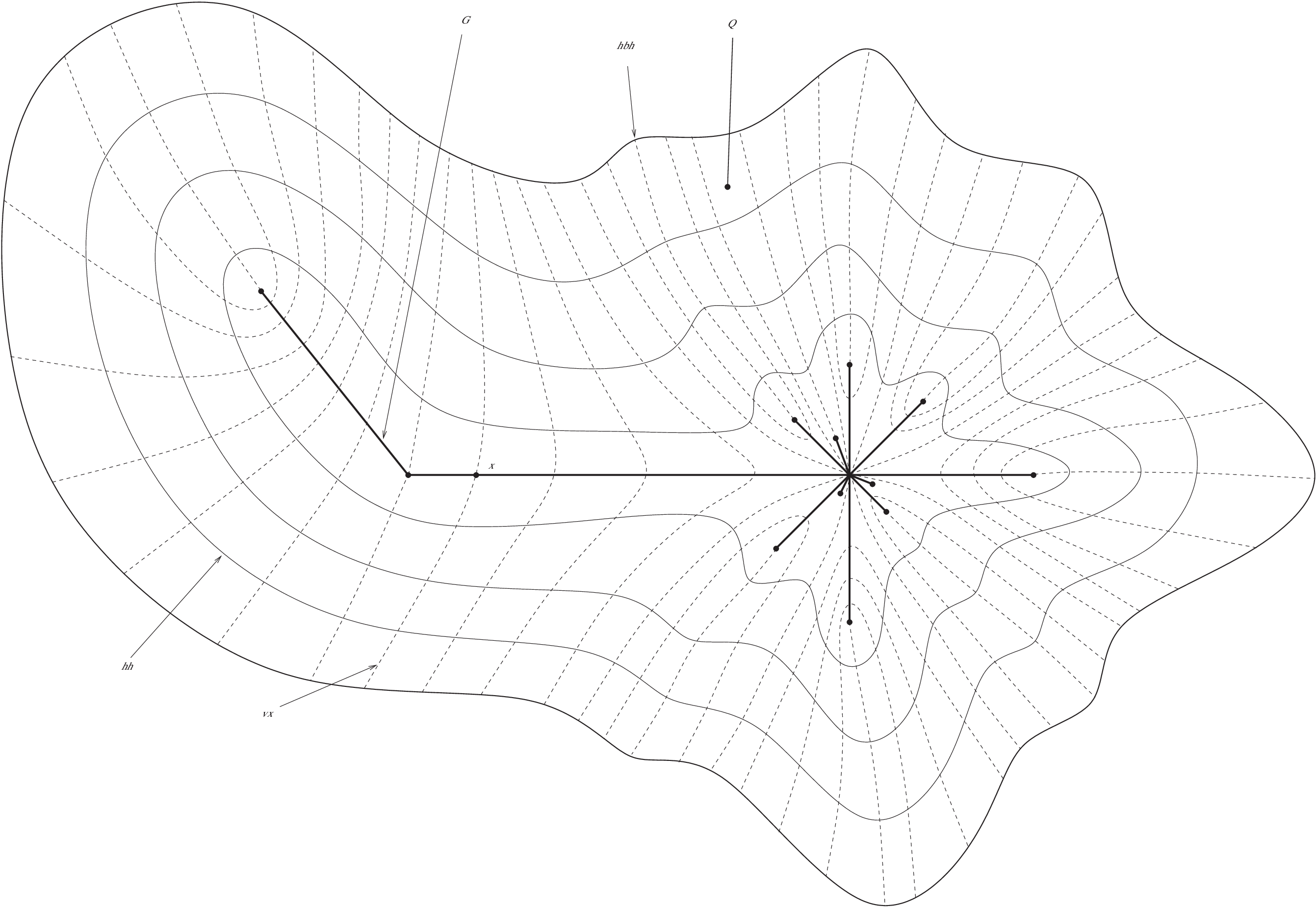}
\end{center}
\caption{The foliations projected to $S$}
\label{fig:foliations}
\end{figure}

The disks~$Q(h)$ for $0<h\le \bh$ are similarly foliated by horizontal
leaves $\hor(h')$ with $0< h'\le h$, and vertical leaves~$\ver_h(x)$,
which are the leaves $\ver(x)$ trimmed at their intersection with
$\hor(h)$.

The composition $\gamma := \pi\circ\tgamma\co [0,L)\times[0,\bh]\to Q$
  parameterizes~$Q$, although it is not injective on~$G$. Notice that
  because the retractions $\tpsi_h\co\tQ(h)\to\partial P$ fix
  $\partial P$ pointwise, the compositions
\[\psi_h := \pi \circ \tpsi_h \circ \pi^{-1} \co Q(h)\to G\]
are well-defined retractions of $Q(h)$ onto $G$.

\subsection{The system of annuli $\Ann(q;r,s)$}
\label{subsec:cxstruct-annuli}
Let~$q$ be a point of the scar~$G$. In this section annuli
$\Ann(q;r,s)$ about~$q$ will be constructed for each pair of
$q$-planar radii $r<s$. The annuli will be defined as differences of
two topological closed disks: $\Ann(q;r,s) = \Int(D(q;s)) \setminus
D(q;r)$.

 There are naturally two parameters~$r$ and~$h$ involved in
 constructing such disks about~$q$ using the foliations of~$Q$, which
 describe respectively the vertical and the horizontal leaves which
 will form its boundary. Here, though, the ratio between these
 parameters will be fixed. The ratio is determined by a
 number~$\br>0$, which in this section can be chosen arbitrarily: in
 the non-plain case, however, it will become the injectivity radius
 of~$G$ (Definition~\ref{defn:injradius}).

Choose $\br>0$, then, and define a function $h:[0,\br]\to[0,\bh/2]$ by
\[h(r) := \left(\frac{\bh}{2\br}\right)\,r,\]
which will fix the parameter~$h$ given the radius~$r$.

\begin{defn}[$D(q;r)$]
\label{defn:Dqr}
Let~$q\in G$ and let $r\in (0,\br]$. The subset $D(q;r)$ of~$Q$ is
defined by
\[D(q;r) := \psi_{h(r)}^{-1}\left(\bB_G(q;r)\right).\]
\end{defn}

Alternatively, $D(q;r)$ is the intersection of~$Q(h(r))$ with the
union of the vertical leaves $\ver(x)$ with $d_G(q,x)\le r$.

\begin{lem}
\label{lem:dqr-is-disk}
Let~$q\in G$ and $r\in(0,\br]$ be a $q$-planar radius. Write
  $n=n(q;r)$ and $h=h(r)$. Then $D(q;r)$ is a topological closed
  disk, whose boundary is composed of~$n$ disjoint subarcs of the
  horizontal leaf $\hor(h)$, and the~$n$ trimmed vertical
  leaves~$\ver_h(x)$ with $x\in C_G(q;r)$.
\end{lem}

\begin{proof} 
(See Figure~\ref{fig:disks}.)  Write
  $C_G(q;r)=\{x_0,\ldots,x_{n-1}\}$. By definition, the boundary
  $\partial_{Q(h)}D(q;r)$ of
  $D(q;r)=\psi_h^{-1}\left(\bB_G(q;r)\right) \subset Q(h)$ in~$Q(h)$
  is contained in $\psi_h^{-1}\left(C_G(q;r)\right)$. Moreover, since~$r$ is
  $q$-planar, every neighborhood of each point of~$C_G(q;r)$ contains
  both points which are closer to~$q$ and points which are further
  away (cf. the proof of Lemma~\ref{lem:planarradius}), and hence
\[\partial_{Q(h)} D(q;r) = \psi_h^{-1}\left(C_G(q;r)\right) = 
\bigcup_{i=0}^{n-1} \psi_h^{-1}(x_i).\]

Since the points~$x_i$ are planar, each $\psi_h^{-1}(x_i) =
\ver_h(x_i)$ is an arc which intersects~$\hor(h) = \partial_S Q(h)$
exactly at its endpoints: that is, a cross cut in~$Q(h)$. For
each~$i$, $Q(h)\setminus\ver_h(x_i)$ has exactly two components, one
of which intersects~$G$ in the complement of $\bB_G(q;r)$. Therefore
every $\ver_h(x_j)$ with $j\not=i$ is contained in the same component
as~$q$. It follows that $\partial_S D(q;r)$ is the simple closed curve
composed of the arcs $\ver_h(x_i)$ and the~$n$ subarcs of~$\hor(h)$
joining the endpoints of consecutive cross cuts in the cyclic order
around~$\hor(h)$.
\end{proof}

\begin{figure}[htbp]
\lab{x0}{x_0}{}\lab{x1}{x_1}{}\lab{x2}{x_2}{}\lab{xn1}{x_{n-1}}{l}
\lab{verx1}{\ver(x_1)}{}\lab{D}{D(q;r)}{}
\lab{bg}{B_G(q;r)}{r}
\lab{pieces}{\text{Pieces of }\hor(h(r))}{l}
\begin{center}
\pichere{0.6}{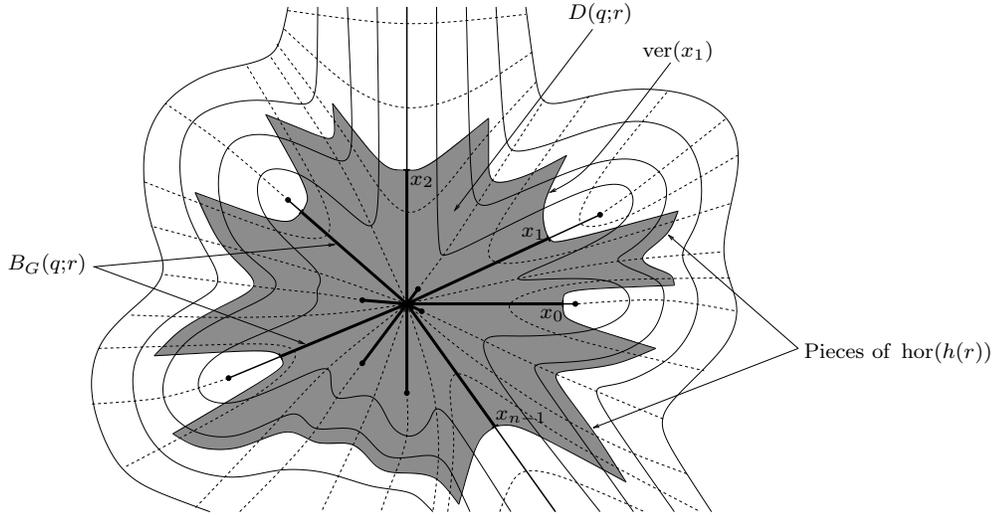}
\end{center}
\caption{The disk $D(q;r)$}
\label{fig:disks}
\end{figure}

The annuli which will be used in the proof of Theorem~\ref{thm:main}
can now be defined.

\begin{defn}[$\Ann(q;r,s)$]
Let $q\in G$, and $r,s\in(0,\br]$ be $q$-planar radii with $r<s$. The
  subset $\Ann(q;r,s)$ of~$Q$ is defined by
\[\Ann(q;r,s) = \Int(D(q;s)) \setminus D(q;r).\] 
\end{defn}

By Lemma~\ref{lem:dqr-is-disk}, and since $D(q;r)\subset\Int(D(q;s))$,
$\Ann(q;r,s)$ is an open annular region with~$q$ in its bounded
complementary component (the complementary component not containing
$\partial Q$).

\subsection{Polygon constants and geometry of the trapezoid construction}
\label{subsec:cxstruct-geometry-polygon-constants}
The goal now is to find lower bounds on the conformal modules of the
annuli~$\Ann(q;r,s)$. In order to use~(\ref{eq:modest}), it is
necessary to estimate the distance in~$S$ between the two boundary
components of the annulus, and the area of the annulus. This will be
done by lifting to the polygon~$P$, where the preimage of the
disk~$D(q;r)$ is a union of {\em polytrapezoids}.

The bounds will be expressed in terms of the {\em polygon constants}
of~$P$: this will make it possible, in Section~\ref{sec:modcont}, to
obtain uniform moduli of continuity for uniformizing maps in families
of polygons with bounded constants.

\begin{defn}[Polygon constants]
\label{defn:constsofP}
The {\em polygon constants} of~$P$ are the numbers $\bh$, $\br$,
and~$|\partial P|$.
\end{defn}

\begin{rmks} \mbox{}
\label{rmks:polygon-constants}
\begin{enumerate}[a)]
\item Of course there is some freedom in the choice of~$\bh$, and total
 freedom in the choice of~$\br$ in the plain case. The important point
 is that, in order to prevent the constructions from becoming
 degenerate in families of examples, $\br$ and~$\bh$ should be
 bounded away from zero, and $|\partial P|$ should be bounded above.
\item The internal semi-angles~$\ttheta_i$ of~$P$ are bounded away
  from~$0$ and~$\pi$ in terms of the polygon constants, since the
  vertical edge of~$\tQ_{i-1}$ and~$\tQ_i$ emanating from~$\tv_i$ has
  length $\bh/\sin\ttheta_i$ and is contained in~$P$: hence
\begin{equation}
\label{eq:bound-theta}
\sin \ttheta_i \ge \frac{2\bh}{|\partial P|} \qquad \text{for
}i=0,\ldots n-1.
\end{equation}
\end{enumerate}
\end{rmks}

\medskip\medskip

The following elementary plane geometry lemma bounds the
derivative of the cotangent of the function $\ttheta(\tx)$ of
Section~\ref{sec:foliations}.

\begin{lem}
\label{lem:bound-dcotx}
Let~$\tx$ lie in the interior of one of the sides~$\te_i$ of~$P$. Then
\[\left|\frac{\rmd}{\rmd\tx} \cot\ttheta(\tx)\right| \le \frac{1}{\bh}.\]
\end{lem}
\begin{proof}
Suppose first that $|\te_i|>|\te_i'|$. Extend the vertical sides of
the trapezoid $\tQ_i$ to form a triangle of height~$H$:
by~(\ref{eq:K}), $H\ge 2\bh$. The vertical leaf through $\tx\in\te_i$
is a segment of the straight line joining $\tx$ to the apex of this
triangle.

Now if $\tx_1$, $\tx_2$ lie in~$\te_i$, then the distance between
$\tx_1$ and $\tx_2$ is given by
$H|\cot(\ttheta(\tx_1))-\cot(\ttheta(\tx_2))|$. Hence
\[\left|\frac{\rmd}{\rmd\tx} \cot\ttheta(\tx)\right| = \frac{1}{H} \le
\frac{1}{2\bh}.\]

If $|\te_i|<|\te_i'|$, then again extend the vertical sides of~$\tQ_i$
to form a triangle of height~$H\ge 2\bh$. In this case the distance
between~$\tx_1$ and~$\tx_2$ is given by
$(H-\bh)|\cot(\ttheta(\tx_1))-\cot(\ttheta(\tx_2))|$, so that
\[\left|\frac{\rmd}{\rmd\tx} \cot\ttheta(\tx)\right| = \frac{1}{H-\bh} \le
\frac{1}{\bh}.\]

If $|\te_i|=|\te_i'|$, then $\ttheta(\tx)$ is constant on~$\te_i$.
\end{proof}

Now let~$q$ be a point of~$G$, which will remain fixed throughout the
remainder of this section (so the dependence of many objects on~$q$
will not be explicitly noted). Let~$(r_1,r_2)\subset(0,\br]$ be an
  interval of~$q$-planar radii, and choose $r\in(r_1,r_2)$. As before
  (see Lemma~\ref{lem:dqr-is-disk} and its proof), write~$n=n(q;r)$,
  and let~$C_G(q;r) = \{x_0(r),\ldots,x_{n-1}(r)\}$, labelling its
  points in the counterclockwise direction around~$\hor(h(r))$.

Since~$G$ is a dendrite, each of the points~$x_i(r)$ is connected
to~$x_{i+1}(r)$ by a unique arc~$\alpha_i=[x_i(r),x_{i+1}(r)]_G$
in~$G$ (Theorem~\ref{thm:pprtsdendrites}b)). These arcs, which are
cross cuts in~$D(q;r)$, form the boundaries in~$D(q;r)$ of~$n$ closed
subdisks $T_i(q;r)$ ($0\le i\le n-1$) (see Figure~\ref{fig:tiqr}). As
a subset of~$S$, $T_i(q;r)$ is bounded by:
\begin{itemize}
\item The arc~$\alpha_i$;
\item $\ver^+_{h(r)}(x_i(r))$ and $\ver^-_{h(r)}(x_{i+1}(r))$,
  segments of vertical leaf with one endpoint at $x_{i}(r)$ and
  $x_{i+1}(r)$ respectively -- denote their other endpoints $z_i^+(r)$
  and $z_{i+1}^-(r)$; and
\item $\Lambda_i(r)$, an arc of~$\hor(h(r))$ from $z_i^+(r)$
  to $z_{i+1}^-(r)$ in the counterclockwise direction.
\end{itemize}

Since the union of the arcs~$\alpha_i$ is a tree whose set of
endpoints is \mbox{$C_G(q;r)\subset\partial D(q;r)$,}
\begin{equation}
\label{eq:tiqr}
D(q;r)=\bigcup_{i=0}^{n-1} T_i(q;r),
\end{equation}
and the subsets~$T_i(q;r)$ intersect only along the
arcs~$\alpha_i$. (If a topological disk~$D$ contains a tree~$T$ which
intersects~$\partial D$ precisely at its endpoints, then components
of~$D\setminus T$ correspond bijectively to pairs of consecutive
endpoints on~$\partial D$.)

\begin{figure}[htbp]
\lab{t0}{T_0(q;r)}{}
\lab{t2}{T_2(q;r)}{}
\lab{x0}{x_0(r)}{l}\lab{x1}{x_1(r)}{l}
\lab{x2}{x_2(r)}{l}\lab{x3}{x_3(r)}{}
\lab{xn-1}{x_{n-1}(r)}{l}
\lab{z0+}{z_0^+(r)}{l}\lab{z0-}{z_0^-(r)}{l}
\lab{hor}{\hor(h(r))}{}
\lab{ver+}{\ver^+_{h(r)}(x_0(r))}{l}
\lab{ver-}{\ver^-_{h(r)}(x_1(r))}{l}
\lab{lambda}{\Lambda_0(r)}{l}
\lab{d}{D(q;r)}{l}
\begin{center}
\pichere{0.6}{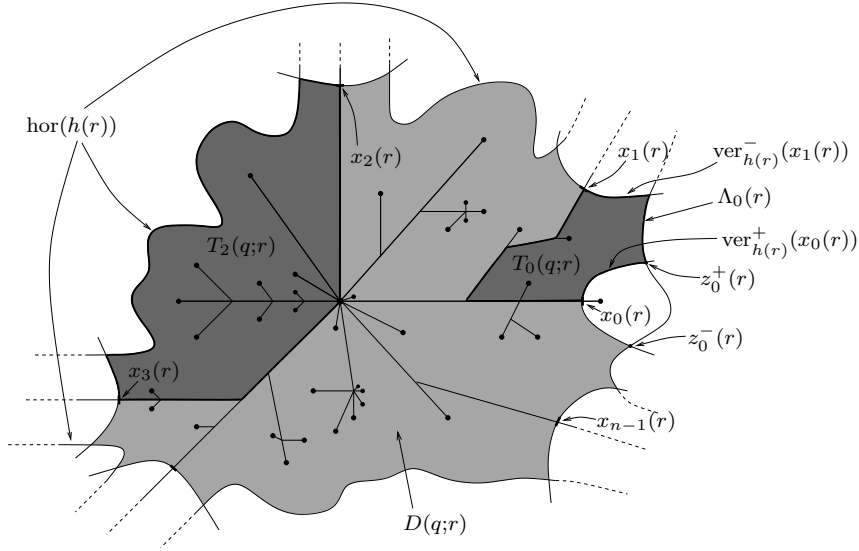}
\end{center}
\caption{The regions~$T_i(q;r)$}
\label{fig:tiqr}
\end{figure}

Now $\Lambda_i(r)$ lies in $Q\setminus G$, and hence has a
well-defined preimage $\tLambda_i(r) = \pi^{-1}(\Lambda_i(r))$, which
is an arc of $\thor(h(r))$ whose initial and final endpoints in the
{\em clockwise} direction along $\thor(h(r))$ are $\tz_i^+(r)$ and
$\tz_{i+1}^-(r)$, the preimages of $z_i^+(r)$ and $z_{i+1}^-(r)$ (see
Figure~\ref{fig:polytrapezoid}). Similarly, $\ver^+_{h(r)}(x_i(r))$
and $\ver^-_{h(r)}(x_{i+1}(r))$ have well-defined preimage arcs
$\tver_{h(r)}(\tx_i^+(r))$ and $\tver_{h(r)}(\tx_{i+1}^-(r))$, with
endpoints $\tx_i^+(r)$ and $\tx_{i+1}^-(r)$ in $\partial P$.

Write $\tDelta_i(r)$ for the union $\tLambda_i(r) \cup
\tver_{h(r)}(\tx_i^+(r)) \cup \tver_{h(r)}(\tx_{i+1}^-(r))$, so
that
\[\pi\left(\bigcup_{i=0}^{n-1}\tDelta_i(r)\right) = \partial D(q;r).\]

Suppose that $\tx_{i+1}^-(r) \in\te_k$ and
$\tx_i^+(r)\in\te_\ell$. The arc $\tL_i(r)$ of $\partial P$ from
$\tx_{i+1}^-(r)$ to $\tx_{i}^+(r)$ in the counterclockwise direction
therefore contains the sides $\te_{k+1},\ldots,\te_{\ell-1}$, and
segments of the sides $\te_k$ and $\te_\ell$. The {\em polytrapezoid}
(union of trapezoids of the same height glued along their vertical
sides) $\tT_i(q;r)$ bounded by $\tL_i(r)$ and $\tDelta_i(r)$ satisfies
$\pi\left(\tT_i(q;r)\right)=T_i(q;r)$, and $\pi\co\tT_i(q;r)\to
T_i(q;r)$ is a homeomorphism away from $\tL_i(r)$. Notice that
\[\bigcup_{i=0}^{n-1} \tL_i(r) = \pi^{-1}\left(\bB_G(q;r)\right),\]
and hence
\begin{equation}
\label{eq:sumli}
\sum_{i=0}^{n-1}\left|\tL_i(r)\right| =
\rmm_G\left(\bB_G(q;r)\right) = m(q;r).
\end{equation}

Denote by $\theta_0(r)$ and $\theta_1(r)$ the internal angles of the
polytrapezoid at the vertices $\tx_{i+1}^-(r)$ and $\tx_i^+(r)$
respectively (the other internal angles along~$\tL_i(r)$ are
$2\ttheta_{k+1}, \ldots, 2\ttheta_\ell$, independent
of~$r\in(r_1,r_2)$). 

\begin{figure}[htbp]
\lab{L}{\tL_i(r)}{}\lab{L'}{\tLambda_i(r)}{}\lab{h}{h(r)}{bl}
\lab{verx-}{\tver_{h(r)}(\tx_{i+1}^-(r))}{br}
\lab{verx+}{\tver_{h(r)}(\tx_{i}^+(r))}{bl}
\lab{ek}{\te_k}{}\lab{ek+1}{\te_{k+1}}{}\lab{el}{\te_\ell}{}
\lab{vk}{\tv_k}{}\lab{vl+1}{\tv_{\ell+1}}{}
\lab{x-}{\tx_{i+1}^-(r)}{r}\lab{z-}{\tz_{i+1}^-(r)}{bl}
\lab{vk+1}{\tv_{k+1}}{}\lab{vl}{\tv_{\ell}}{}
\lab{x+}{\tx_{i}^+(r)}{t}\lab{z+}{\tz_{i}^+(r)}{b}
\lab{t0}{\theta_0(r)}{r}\lab{t1}{\theta_1(r)}{tr}
\lab{tk+1}{\ttheta_{k+1}}{r}\lab{tl}{\ttheta_{\ell}}{}
\lab{t}{\tT_i(q;r)}{bl}
\begin{center}
\pichere{0.6}{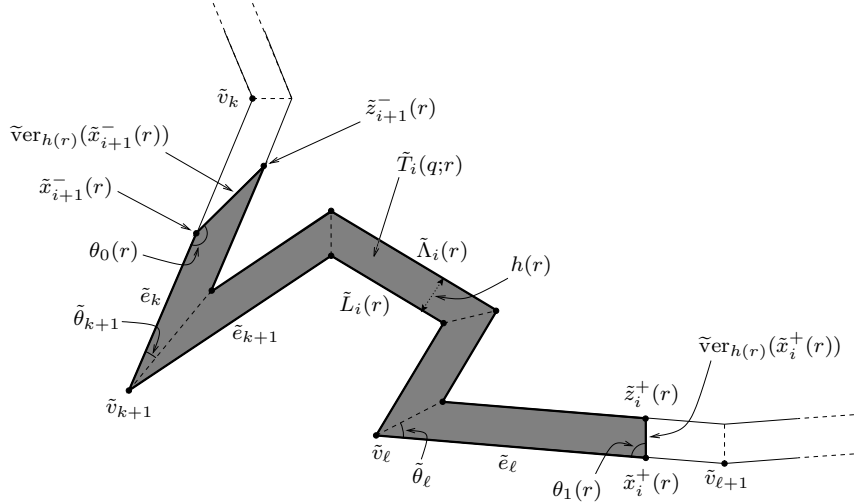}
\end{center}
\caption{The polytrapezoid $\tT_i(q;r)$}
\label{fig:polytrapezoid}
\end{figure}

\begin{lem}
\label{lem:area-tti}
\[\Area(\tT_i(q;r))=h(r) \left|\tL_i(r)\right| -
\frac{h(r)^2}{2}\left(\cot\theta_0(r)+\cot\theta_1(r) +
2\sum_{j=k+1}^{\ell}\cot\ttheta_j\right).\]
\end{lem}
\begin{proof}
Observe that if a trapezoid has base~$e$, top~$e'$, height~$h$, and
internal angles $\theta_1$ and $\theta_2$ at its base vertices, then
\begin{equation}
\label{eq:trapbaseangle}
|e'| = |e| - h(\cot\theta_1 + \cot\theta_2).
\end{equation}
The result follows immediately on summing the areas $h(|e|+|e'|)/2$ of
the constituent trapezoids of~$\tT_i(q;r)$.
\end{proof}

\begin{lem}
\label{lem:darea-tti}
\[0<\frac{\rmd}{\rmd r}\Area\left(\tT_i(q;r)\right) \le \frac{5\bh}{4\br}\left( \left|\tL_i(r)\right| +
 r\right).\] 
\end{lem}
\begin{proof}
The first inequality in the statement is obvious.

Differentiating the expression for the area given by
Lemma~\ref{lem:area-tti}, recalling that $h(r) = (\bh/2\br)r$, and
observing that $\frac{\rmd}{\rmd r}\left|\tL_i(r)\right| = 2$ and that 
\[\left|\tLambda_i(r)\right| = \left|\tL_i(r)\right| -
h(r)\left(\cot\theta_0(r)+\cot\theta_1(r)+2\sum_{j=k+1}^{\ell}\cot\ttheta_j
\right)\]
by~(\ref{eq:trapbaseangle}),
gives
\begin{eqnarray*}
\frac{\rmd}{\rmd r}\Area\left(\tT_i(q;r)\right) &=&
\frac{\rmd h(r)}{\rmd r} \left[
\left|\tL_i(r)\right| - h(r)\left(\cot\theta_0(r)+\cot\theta_1(r)+2\sum_{j=k+1}^{\ell}\cot\ttheta_j\right)
\right] \\
&&\quad + h(r)\left[\frac{\rmd}{\rmd r} \left|\tL_i(r)\right| - \frac{h(r)}{2}
  \frac{\rmd}{\rmd r}\left(\cot\theta_0(r)+\cot\theta_1(r)\right)
  \right] \\
&\le&\frac{\bh}{2\br}\left|\tLambda_i(r)\right| +
\frac{\bh r}{2\br}\left(2+\frac{h(r)}{\bh}\right)\\
&\le& \frac{\bh}{\br} \left|\tL_i(r)\right| + \frac{5\bh}{4\br}\, r
\end{eqnarray*}
as required, using $\left|\tLambda_i(r)\right| \le
2\left|\tL_i(r)\right|$ (which follows from~(\ref{eq:K})) and
Lemma~\ref{lem:bound-dcotx}.
\end{proof}

In order to estimate the $d_S$-distance between the boundary
components of~$\Ann(q;r,s)$, where $r_1<r<s<r_2$, it is necessary to
find a lower bound for the distance between $\tDelta_i(r)$ and
$\tDelta_i(s)$, which project to part of the inner and the outer
boundary components of $\Ann(q;r,s)$ respectively.

\begin{lem}
\label{lem:cap-distance}
Let $(r_1,r_2)$ be an interval of $q$-planar radii. If
$r_1<r<s<r_2$ then
\[d\left(\tDelta_i(r), \tDelta_i(s)\right) \ge (s-r)
\min\left\{\frac{\bh}{2\br}, \frac{1}{2}\right\}.\]
\end{lem}
\begin{proof}
Let~$\tp\in\tDelta_i(s)$. If $\tp\in\tLambda_i(s)$, then every point
of~$\tDelta_i(r)$ is distance at least $h(s)-h(r) =
\frac{\bh}{2\br}(s-r)$ from $\tp$.

If $\tp\in\tver_{h(s)}(\tx_i^+(s))$, then no point of $\tDelta_i(r)$
can lie closer to~$\tp$ than the closest point of
$\tver_{h(s)}(\tx_i^+(r))$ (the arc obtained by extending
$\tver_{h(r)}(\tx_i^+(r)) \subset\tDelta_i(r)$ to
$\thor(h(s))$). Hence $d(\tp, \tDelta_i(r)) \ge (s-r)/2$
by~(\ref{eq:K}). The proof if $\tp\in\tver_{h(s)}(\tx_{i+1}^-(s))$ is
similar. 
\end{proof}

\subsection{Proof of Theorem~\ref{thm:main}}
\label{subsec:cxstruct-proof}
Initially in this section, $(P,\cP)$ is again a plain
paper-folding scheme with associated paper sphere~$S$ and
scar~$G\subset S$, and $q\in G$. The first step is to interpret the
geometric results of the previous section in the sphere~$S$.

\begin{lem}
\label{lem:area'bound}
For all planar radii $r\in(0,\br)$,
\[0 < \frac{\rmd}{\rmd r} \Area_S(D(q;r)) \le \frac{5\bh}{4\br}
\left(m(q;r) + r\cdot n(q;r)\right).\]
\end{lem}
\begin{proof}
Write $n=n(q;r)$. From~(\ref{eq:tiqr}) it follows that
\[\Area_S(D(q;r)) = \sum_{i=0}^{n-1}
\Area_{\R^2}\left(\tT_i(q;r)\right).\]
Hence, by Lemma~\ref{lem:darea-tti}, 
\begin{eqnarray*}
0 < \frac{\rmd}{\rmd r} \Area_S(D(q;r)) &\le& \frac{5\bh}{4\br}
\sum_{i=0}^{n-1}\left(\left|\tL_i(r)\right| + r\right)\\
&=& \frac{5\bh}{4\br}\left(m(q;r) + r\cdot n(q;r)\right),
\end{eqnarray*}
by~(\ref{eq:sumli}) as required.
\end{proof}

\begin{lem}
\label{lem:Annulus-boundary-separation}
Let~$(r_1,r_2)$ be an interval of~$q$-planar radii, and
$r_1<r<s<r_2$. Let~$\xi(r)$ and $\xi(s)$ denote the boundary
components of~$\Ann(q;r,s)$. Then
\[d_S(\xi(r),\xi(s)) \ge  (s-r)
\min\left\{\frac{\bh}{2\br}, \frac{1}{2}\right\}.\]
\end{lem}
\begin{proof}
Let~$D:=d_S(\xi(r),\xi(s))$, and let~$p'\in\xi(r)$ and~$q'\in\xi(s)$
be points with $d_S(p',q')=D$. Since~$d_S$ is strictly intrinsic by
Theorem~\ref{thm:metric-structure}b), there is an arc~$\gamma'$
from~$p'$ to~$q'$ of length~$D$, and apart from its endpoints this
arc, being the shortest from~$\xi(r)$ to~$\xi(s)$, must lie
in~$\Ann(q;r,s)$.

It will be shown that, for any $r\le u<v\le s$, the length of any
arc~$\gamma$ from~$\xi(u)$ to~$\xi(v)$ with
$\interior{\gamma}\subset\Ann(q;u,v)$ is at least $C(v-u)$, where
$C=\min\left\{\frac{\bh}{2\br}, \frac{1}{2}\right\}$, which will
establish the result.

It can be assumed without loss of generality that there are no subarcs
of~$\gamma$ whose endpoints lie on the same component of~$G\cap
\overline{\Ann(q;u,v)}$ and whose interiors lie entirely in the
  interior of a single~$T_i(q;v)$, since such a subarc could be
  replaced with a shorter subarc lying in~$G$. Therefore the
  intersection of~$\interior{\gamma}$ with~$G$ has a finite number~$N$
  of components. The result will be shown by induction on~$N$.

If~$N=0$, then the arc lies in a single~$T_i(q;v)$, and the result is
immediate by Lemma~\ref{lem:cap-distance}. If $N>0$, then let
$t_1=\inf\{t>0\,:\,\gamma(t)\in G\}$ and
$t_2=\inf\{t>t_1\,:\,\gamma(t)\not\in G\}$, so that $\gamma([t_1,t_2])
\subset G$. Split $\gamma$ into three subarcs $\gamma_1\co[0,t_1] \to
S$, $\gamma_2\co[t_1,t_2]\to S$, and $\gamma_3\co[t_2,1]\to S$ (one
of these arcs may degenerate to a point). Let
$d(q,\gamma(t_1))=w_1\in[u,v]$ and $d(q,\gamma(t_2))=w_2\in[u,v]$. By
the inductive hypothesis, $\gamma_1$ has length at least $C(w_1-u)$
and $\gamma_3$ has length at least $C(v-w_2)$, while $\gamma_2$ lies
entirely in a single $T_i(q;w_2)$ and hence has length at least
$C(w_2-w_1)$ by Lemma~\ref{lem:cap-distance}.
\end{proof}

\begin{lem}
\label{lem:modest}
Let $r$ and $s$ be $q$-planar radii with $0<r<s<\br$. Then
\[\mod\Ann(q;r,s) \ge \int_r^s \frac{C^2\,\rmd u}{\Area_S'(D(q;u))},\]
where
\[C := \min\left\{\frac{\bh}{2\br}, \frac{1}{2}\right\} \qquad
\text{and}\qquad \Area_S'(D(q;r)):= \frac{\rmd\Area_S(D(q;r))}{\rmd
  r}. \]
\end{lem}

\begin{proof}
Let~$[r_1,r_2]$ be an interval of $q$-planar radii. Then
\[\mod\Ann(q;r_1,r_2) \ge \frac{C^2
  (r_2-r_1)^2}{\Area_S(\Ann(q;r_1,r_2))}\] by~(\ref{eq:modest}) and
Lemma~\ref{lem:Annulus-boundary-separation}. 

For any partition $r_1=s_0<s_1<\cdots<s_k=r_2$ of $[r_1,r_2]$, it
follows from~(\ref{eq:summod}) that
\begin{eqnarray*}
\mod\Ann(q;r_1,r_2) &\ge& \sum_{j=1}^{k} \mod\Ann(q;s_{j-1},s_j) \\
&\ge& C^2 \sum_{j=1}^{k} \frac{(s_j-s_{j-1})^2}{\Area_S(D(q;s_j)) -
  \Area_S(D(q;s_{j-1}))},
\end{eqnarray*}
and taking the supremum over all partitions of~$[r_1,r_2]$ gives
\[\mod\Ann(q;r_1,r_2) \ge \int_{r_1}^{r_2} \frac{C^2\,\rmd
  u}{\Area_S'(D(q;u))},\]
since $\Area_S'(D(q;u))$ is continuous and positive at planar radii.
Hence for any collection of disjoint intervals $[r_1^k,r_2^k]\subset[r,s]$ of
$q$-planar radii, (\ref{eq:summod}) gives
\[\mod\Ann(q;r,s) \ge \sum_k \int_{r_1^k}^{r_2^k} \frac{C^2\,\rmd
  u}{\Area_S'(D(q;u))}.\]
Taking the supremum over such collections of disjoint intervals, using
the fact that the set of non-planar radii is closed and has zero
measure, gives the result.

\end{proof}

Combining the results of Lemmas~\ref{lem:modest}
and~\ref{lem:area'bound} gives that, for all $q$-planar radii
$0<r<s<\br$,
\begin{equation}
\label{eq:combined}
\mod\Ann(q;r,s) \ge \int_r^s \frac{C^2\,\rmd
  u}{\frac{5\bh}{4\br}\left(
m(q;r) + r\cdot n(q;r)\right)},
\end{equation}
where $C=\min\left\{\frac{\bh}{2\br}, \frac{1}{2}\right\}$. This
motivates the following definition:
\begin{defn}[Paper-folding goodness function]
\label{defn:goodnessfn}
Let~$(P,\cP)$ be a surface paper folding scheme, and let~$P$ have
polygon constants $\bh$ and $\br$ (and~$|\partial P|$). Let
\[M = \frac{1}{5}\min\left\{\frac{\br}{\bh},
\frac{\bh}{\br}\right\},\]
and define a function $\iota\co G\times(0,\br)\to[0,\infty)$ by
\[
\iota(q;r):= 
\begin{cases}
\dfrac{M}{m(q;r)+r\cdot n(q;r)} & 
              \text{ if } n(q;r)<\infty, \\ \\
0 & \text{ otherwise}.
\end{cases}
\]
$\iota$ is called a {\em paper-folding goodness function} for~$(P,\cP)$.
\end{defn}

\begin{rmks}
\label{rmks:pfgoodnessfn}
If $r$ is a $q$-planar radius then $\iota$ is non-zero and continuous
at $(q;r)$ by Lemma~\ref{lem:planarradius} (it is clear that
$m(q;r)$ is continuous at such a~$(q;r)$). Hence, using
Lemma~\ref{lem:planarradius} again, for each~$q\in G$ the set of
radii~$r$ at which $\iota(q;r)=0$ or $\iota(q;r)$ is discontinuous has
measure zero.

As mentioned before, goodness functions are not uniquely determined
since there is some freedom in the choice of polygon constants. What
is important about them in this section is that their integrals down
to zero be divergent. The rate of divergence will be important in
Section~\ref{sec:modcont}, where moduli of continuity are discussed.
\end{rmks}

Inequality~(\ref{eq:combined}) now becomes:
\begin{equation}
\label{eq:modbound}
\mod\Ann(q;r,s)\geq \int_{r}^{s}\iota(q;t)\rmd t
\end{equation}
for every pair of planar radii $0<r<s<\br$.

\medskip
\begin{proof}[Proof of Theorem~\ref{thm:main}]
Let~$q\in G$ be an isolated singular point. Pick a $q$-planar radius
\mbox{$r_0\in(0,\br)$} small enough that~$q$ is the only singular
point in~$D(q;r_0)$, and let \mbox{$W= D(q;r_0)\setminus\{q\}$.}
Then~$W$ is a Riemann surface, since it is contained
in~$S\setminus\cV^s$ where there is a well-defined conformal
structure. Since~$W$ is topologically a disk minus a point, it is
conformally homeomorphic to a plane domain (see~\cite{AhSa}, for
example). In order to prove that~$q$ is a puncture of the complex
structure on~$W$, observe that~(\ref{eq:modbound}) yields, for
$q$-planar radii $0<r<r_0$,
\begin{eqnarray*}
\mod\Ann(q;r,r_0)&\geq&\int_r^{r_0}\iota(q;s)\rmd s.
\end{eqnarray*}
Since this integral diverges as $r\searrow 0$ by hypothesis, the result
follows from Lemma~\ref{lem:conformalpnct} and Riemann's Removable
Singularity Theorem.

Only minor modifications are needed to treat the case of non-plain
surface paper-folding schemes~$(P,\cP)$. The constant~$\bh$ should be
chosen in such a way that foliated collars of height~$\bh$ can be
constructed for all of the polygons~$P$, and the constant~$\br$ should
be chosen to be the injectivity radius of~$G$. It then follows from
Lemma~\ref{lem:injradius} that $\bB_G(q;r)$ is a dendrite for
all~$q\in G$ and all $r<\br$, and the constructions and proof go
through exactly as in the plain case. 
\end{proof}

\section{Modulus of continuity}
\label{sec:modcont}

Let $(P,\cP)$ be a surface paper-folding scheme with associated paper
surface~$S$ and scar~$G$. Suppose that the hypothesis of
Theorem~\ref{thm:main} is satisfied at some point~$q\in G$ other than
a non-isolated singularity (as explained in
Remark~\ref{rmk:divergence}, the hypothesis is necessarily satisfied
when~$q$ is a planar point or regular vertex). Then for any sufficiently
small $q$-planar radius~$r$, \mbox{$D=\Int(D(q;r))$} is a conformal disk, and
by the Uniformization Theorem there is a conformal isomorphism $u\co
D\to\D$ to the unit disk $\D\subset\C$. In this section a modulus of
continuity for~$u$ is obtained.

\begin{defn}[Modulus of continuity]
\label{defn:modcont}
Let $\rho\co [0,\delta)\to[0,\infty)$, for some~$\delta>0$, be a
    continuous and strictly increasing function with $\rho(0)= 0$. A
    function $f\co (X,d_X)\to(Y,d_Y)$ between metric spaces has {\em
      modulus of continuity} $\rho$ at $x_0\in X$ if, for every $x\in
    X$ with $d_X(x_0,x)<\delta$,
\[ d_Y\left(f(x_0),f(x)\right)\leq \rho\left(d_X(x_0,x)\right).\]
$f$ is said to have {\em modulus of continuity} $\rho$ if the
inequality above holds for every $x_0\in X$.
\end{defn}
\begin{rmk}
If $\cF$ is a family of functions all of whose members have the same
modulus of continuity, then $\cF$ is {\em uniformly equicontinuous:}
for all $\veps>0$ there exists $\eta>0$ such that if two points are at
distance less than $\eta$ then their images under any function in
$\cF$ (whose domain contains the two points) are at distance less than
$\veps$. Notice that there is no requirement for the domains of the
functions in $\cF$ to coincide: indeed, in the extreme case they may
have empty intersection.
\end{rmk}

\subsection{Modulus of continuity at a point} 
\label{subsec:localmodcont}
Here a local version of the main result of this section is presented:
it provides an illustration of the ideas of the proof without the
technical details which obscure the argument in the global case.

Let~$(P,\cP)$ be a surface paper-folding scheme, $q$ be an isolated
singularity in the scar~$G$, and suppose that the hypothesis of
Theorem~\ref{thm:main} is satisfied at~$q$, so that the complex
structure on~$S\setminus\cV^s$ extends uniquely across~$q$. It will be
shown how it is possible to obtain a modulus of continuity for a
uniformizing map~$u$ from a disk neighborhood of~$q$ to the unit disk
in the complex plane.

Pick a $q$-planar radius~$r_0\in(0,\br)$ small enough that
$D:=\Int(D(q;r_0))$ contains no singularities other than~$q$. Let
$u\co D\to\D$ be a conformal uniformizing chart, where~$\D$ is the
open unit disk in the complex plane: normalize~$u$ so that $u(q)=0$.

It is shown in Lemma~\ref{lem:ballindisk} below that there are numbers
$\delta>0$ and~$A>0$, depending only on the polygon constants of~$P$,
such that
\[B_S(q;r) \subset D(q;Ar) \qquad\text{for all }r\le\delta.\]
Pick any $x\in D$ with $d_S(q,x) < \min\{\delta, r_0/A\}$, and write
$R:=A\,d_S(q,x)<r_0$, so that $x\in D(q;R)$. Assume that~$R$ is a
$q$-planar radius (if not, increase it by an arbitrarily small
amount). Then $\Ann(q;R,r_0)$ separates~$q$ and~$x$ from~$\partial
D(q;r_0)$, and hence $u(\Ann(q;R,r_0))$ separates $0=u(q)$ and~$u(x)$
from the unit circle in~$\C$. It follows from~(\ref{eq:modbound}) and
the Gr\"otzsch annulus theorem (Theorem~\ref{thm:GAT}) that
\begin{eqnarray*}
\int_R^{r_0} \iota(q;t)\,\rmd t &\le& \mod\Ann(q;R,r_0) \\
&=& \mod u(\Ann(q;R,r_0))\\
&\le& \frac{1}{2\pi} \ln\frac{4}{|u(x)|}.
\end{eqnarray*}
Hence, recalling that $R=A\,d_S(q,x)$ (or~$R$ is greater than
$A\,d_S(q,x)$ by an arbitrarily small amount in the case where~$R$ is
not $q$-planar),
\begin{equation}\label{eq:babymodcont}
 |u(x)|\leq \frac{4}{\exp\left(\displaystyle2\pi
\int_{A\, d_S(q,x)}^{r_0}\iota(q;t)\,\rmd t 
\right)},
\end{equation}
 which is the desired modulus of continuity.

\begin{rmk}
\label{rmk:mod-one-mod-other}
Since the projection $\pi\co P\to S$ is distance non-increasing, a
modulus of continuity~$\rho_q$ for~$u$ at~$q$ is also a modulus of
continuity for the composition $\phi:=u\circ\pi$ at the points
of~$\pi^{-1}(q)$: if $q=\pi(\tq)$ and $x=\pi(\tx)$, then
\begin{eqnarray*}
|\phi(\tq)-\phi(\tx)| &=& |u(q)-u(x)| \\
                   &\leq& \rho_q(d_S(q,x)) \\
                   &\leq& \rho_q(d_P(\tq,\tx)),
\end{eqnarray*}
since $d_S(q,x)\leq d_P(\tq,\tx)$ and $\rho_q$ is increasing. 
\end{rmk}

\begin{example}\label{ex:bothcont}
Consider Examples~\ref{ex:generalexample} and~\ref{ex:both} and
suppose as in the second case of Example~\ref{ex:both} that the fold
lengths~$a_n$ satisfy $a_n \asymp 1/\lambda^n$ for some
$\lambda>1$. As shown in that example, the hypothesis of
Theorem~\ref{thm:main} holds at the unique singularity~$q_0\in G$, and
hence the paper surface~$S$ is conformally isomorphic to the Riemann
sphere. Choose a uniformizing map $u\co S\to\csph$ with $u(q_0)=0$,
and let~$\phi\co P\to\csph$ be the composition $\phi:=u\circ\pi$. Thus
$\phi(0)=0$: an explicit modulus of continuity for~$\phi$ at~$0$ will now be
found. 

From the calculations of Example~\ref{ex:both}, a paper-folding
goodness function at~$q_0$ can be taken to be
\[\iota(q_0;r) = \frac{1}{Cr\ln\frac{1}{r}},\]
for some positive constant~$C$, so that
\begin{equation*}
\int_{R}^{r_0}\iota(q_0;t)\rmd t =
\frac{1}{C}\left(\ln\ln\frac{1}{R}
-\ln\ln\frac{1}{r_0}\right).
\end{equation*}
Choose~$r_0$ small enough that $u(D(q;r_0))\subset \D$. Substitution
in~(\ref{eq:babymodcont}) yields, for $\tx$ close to~$0$ (how close
depends on the choice of uniformizing map~$u$)
\[|\phi(\tx)|\le 
4\left[\frac{\ln(1/r_0)}{\ln(1/\left(A|\tx|\right))}\right]^{2\pi/C}.\] 
\end{example}

\subsection{Global modulus of continuity}
\label{subsec:globalmodcont}
In the remainder of Section~\ref{sec:modcont}, the following
assumptions will be made:
\begin{enumerate}[a)]
\item $(P,\cP)$ is a plain paper-folding scheme;
\item there are only finitely many singular points in the scar~$G$;
  and 
\item at each of these singular points, the hypothesis of
  Theorem~\ref{thm:main} is satisfied.
\end{enumerate}
It will be shown how to obtain a global modulus of continuity for a
suitably normalized uniformizing map~$u\co S\to\csph$ from the paper
sphere~$S$ to the Riemann sphere (and hence also for the composition
$\phi = u\circ\pi\co P\to\csph$).

Let~$\iota(q;r)$ be a paper-folding goodness function
for~$(P,\cP)$. Then
\begin{equation}\label{eq:assumption}
\int_0\iota(q;t)\rmd t =\infty
\end{equation}
for every $q\in G$: at singular points this is the assumption~c)
above, while at other points the integral diverges as explained in
Remark~\ref{rmk:divergence}. 

Recall (Definition~\ref{defn:constsofP}) that the polygon constants
of~$P$ are the numbers~$\bh$, $\br$, and $|\partial P|$: here~$\bh$ is
the height of the collaring of
Section~\ref{subsec:cxstruct-foliatedcollar}, and, since the
paper-folding is plain, $\br>0$ can be chosen arbitrarily. A sensible
choice is $\br=\bh$, which maximises the constant~$M$ in the goodness
function~$\iota$, with $M=1/5$.

For each $h\in(0,\bh]$, write
\[P_h:=P\setminus \tQ(h),\]
an open disk in $P\subset\C$ which is the complement of the collaring
of height~$h$.

\medskip
 To fix a uniformizing map $u\co S\to\csph$, pick points
$\tp_0\in \partial P$ and $\tp_\infty\in P_\bh$. Set
$p_0:=\pi(\tp_0)\in G$ and $p_\infty:=\pi(\tp_\infty)\in S\setminus
Q(\bh)$, and define $u\co S\to\csph$ to be the isomorphism with the
following normalization:
\begin{itemize}
\item $u(p_0)=0$;
\item $u(p_\infty)=\infty$; and
\item the reciprocal $\Phi = 1/\phi$ of the composition
  $\phi:=u\circ\pi\co P\to\csph$ satisfies $\Phi'(\tp_\infty)=1$.
\end{itemize}
Observe that $\phi$ is injective and meromorphic in $\Int(P)$ with a
simple pole at $\tp_\infty$.

\begin{lem}
\label{lem:R}
There is a constant~$R=R(\bh)>0$, depending only on~$\bh$, such that
\[\phi\left(\tQ(\bh/2)\right) \subset \D_R :=\left\{z\in\C\,:\, |z|
< R\right\}\]
(or, equivalently, $u(Q(\bh/2)) \subset \D_R$).
\end{lem}
\begin{proof}
Since it lies in~$P_\bh$, the point $\tp_\infty$ is distance at
least~$\bh/2$ from any point of $\tQ(\bh/2)$: that is,
$\tQ(\bh/2) \subset P \setminus B_\C(\tp_\infty, \bh/2)$. Therefore
\[\Phi\left(\tQ(\bh/2)\right) \subset \Phi\left(P\setminus
B_\C(\tp_\infty; \bh/2)\right) = \csph \setminus
\Phi\left(B_\C(\tp_\infty; \bh/2)\right),\] where $\Phi = 1/\phi$.
Since~$\Phi$ is univalent in $B_\C(\tp_\infty; \bh/2)$ with
$\Phi(\tp_\infty)=0$ and $\Phi'(\tp_\infty)=1$, Koebe's one-quarter
theorem gives a radius $s(\bh)>0$ such that
$\Phi\left(B_\C(\tp_\infty; \bh/2)\right) \supset \bDD_{s(\bh)}$, and
the result follows with $R(\bh)=1/s(\bh)$.
\end{proof}

\bigskip

The main theorem which will be proved in this section is:
\begin{thm}\label{thm:modcont}
Let~$(P,\cP)$ be a plain paper-folding scheme with only finitely many
singular points in its scar~$G$, at each of which the conditions of
Theorem~\ref{thm:main} hold. Then the uniformizing map~$\phi\co
P\to\csph$ has a modulus of continuity $\brho$, with respect to the
Euclidean metric on~$P$ and the spherical metric on~$\csph$, which
depends only on the polygon constants of~$P$ and on the paper-folding
goodness function $\iota\co G\times(0,\br)\to[0,\infty)$.
\end{thm}

Recall that the spherical metric~$d_{\csph}$ on~$\csph$ is defined by
\[d_{\csph}(w,z) = \inf_{\gamma\in\Gamma} \left(2\int_\gamma
\frac{|\rmd z|}{1+|z|^2}\right) ,\] where $\Gamma$ is the set of paths
in~$\csph$ from~$w$ to~$z$.  The only properties of the spherical
metric which will be used here are: $d_{\csph}(1/w,1/z) =
d_{\csph}(w,z)$ for all $w,z\in\csph$; and $d_{\csph}(w,z) \le 2|w-z|$
for all $w,z\in\C$.

\bigskip

\subsection{Two technical lemmas}

Recall from Section~\ref{sec:foliations}
that~$\tpsi_\bh:\tQ(\bh)\to\partial P$ is the retraction obtained by
sliding points along leaves of~$\tVer$: for economy of notation,
this retraction will henceforth be denoted~$\tpsi$, and similarly
$\psi$ rather than $\psi_\bh$ will be used to denote the retraction
$\pi\circ\tpsi\circ\pi^{-1}\co Q(\bh)\to G$. The first lemma in this
section essentially says that~$\psi$ is Lipschitz with constant
determined by the polygon constants.

\begin{lem}
\label{lem:lipschitz-psi}
Let~$\alpha$ be a rectifiable path in~$Q(\bh)$. Then
\[|\psi\circ\alpha|_G \le \frac{|\partial P|}{\bh}\,|\alpha|_S.\]
\end{lem}

\begin{proof}[Sketch proof]
Lift~$\alpha$ to a path in~$\tQ(\bh)$, and consider a point
$\tgamma(t,h)$ of~$\alpha$ lying in a trapezoid~$\tQ_i$. As in the
proof of Lemma~\ref{lem:bound-dcotx}, extend the vertical sides
of~$\tQ_i$ to form a triangle~$T$ of height~$H\ge 2\bh$.  The action
of~$\tpsi$ in~$\tQ_i$ is projection from the apex of~$T$ onto the base
of~$\tQ_i$.

An infinitesimal arc through~$\tgamma(t,h)$ is maximally stretched
by~$\tpsi$ when it is perpendicular to the line segment from
$\tgamma(t,h)$ to the apex of~$T$, in which case it is stretched by a
factor $\frac{H}{(H-h)\sin\ttheta(t)}$. The result follows
using~(\ref{eq:bound-theta}) and $H\ge 2h$.
\end{proof}

The following lemma relates metric balls $B_S(q,\delta)$ in $S$ to the
disks $D(\psi(q);\mu)$ for $q$ in a thinner collar
$Q(\delta)=\pi\circ\tQ(\delta)$ (see Figure~\ref{fig:ballindisk}).

\begin{figure}[htbp]
\lab{q}{q}{}\lab{pq}{\psi(q)}{}
\lab{B}{B_S(q;r)}{b}
\lab{D}{D\left(\psi(q);A(r+h_q)\right)}{t}
\begin{center}
\pichere{0.5}{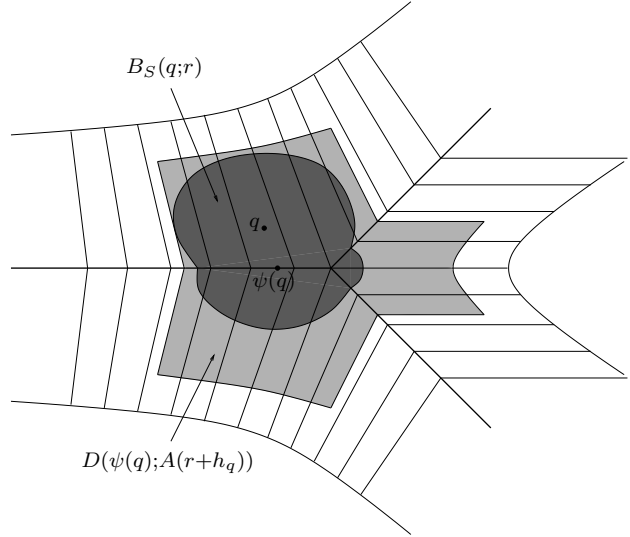}
\end{center}
\caption{The disk $D\left(\psi(q);A(r+h_q)\right)$ and the ball
$B_S(q;r)$.}
\label{fig:ballindisk}
\end{figure}

\begin{lem}
\label{lem:ballindisk}
Let 
\begin{equation}
\label{eq:delta-A}
\delta=\frac{1}{4}\min\left\{\bh,\br, \frac{2\bh\br}{|\partial
  P|}\right\}\qquad \text{and}\qquad A = \frac{\br}{2\delta} =
2\max\left\{\frac{\br}{\bh}, 1, \frac{|\partial P|}{2\bh} \right\}.
\end{equation} 
Then for
every $q\in Q(\delta)$ and every $r\in[0,\delta]$,
\[\bB_S(q;r) \subset D(\psi(q); A(r+h_q)),\]
where $h_q\in[0,\delta]$ is the height of~$q$ (i.e. $q\in\hor(h_q)$).
\end{lem}
\begin{rmk}
Notice that $A(r+h_q) \le A(\delta+\delta)=\br$, so that
$D(\psi(q);A(r+h_q))$ is defined (even when~$A(r+h_q)$ is not a
$\psi(q)$-planar radius, in which case it may not be a disk).
\end{rmk}
\begin{proof}[Proof of Lemma~\ref{lem:ballindisk}]
Observe first that the $d_S$-distance from~$q$ to the
boundary~$\hor(\bh)$ of $Q=Q(\bh)$ is at least
$\bh-h_q\ge\bh-\delta\ge 3\bh/4>\delta\ge r$, so that $\bB_S(q;r)
\subset Q(\bh)$.

Write $K=A(r+h_q)$. Since $D(\psi(q); K) =
\psi_{h(K)}^{-1}\left(\bB(q;K)\right)$ where $h(K)= \bh K/2\br$
(Definition~\ref{defn:Dqr}), it is required to show that every
point~$x\in S$ with $d_S(q,x)\le r$ has height $h_x\le \bh K/2\br$ and
satisfies $d_G(\psi(q),\psi(x)) \le K$.

The former property is immediate since $h_x\le r+h_q = K/A \le \bh
K/2\br$.  For the latter, since~$d_S$ is strictly intrinsic
(Theorem~\ref{thm:metric-structure}b)), there is a path in~$S$
from~$q$ to~$x$ of length~$d_S(q,x)$, and this path must lie in
$Q(\bh)$ since the distance from~$q$ to~$\hor(\bh)$ is greater
than~$r$. The image under~$\psi$ of this path is a path in~$G$ from
$\psi(q)$ to $\psi(x)$, whose length is no more than
$\frac{|\partial P|}{\bh} d_S(q,x)$
by Lemma~\ref{lem:lipschitz-psi}. Hence 
\[d_G(\psi(q),\psi(x)) \le \frac{|\partial P|}{\bh}\,r \le A(r+h_q)=K\] 
as required.
\end{proof}

\subsection{Modulus of continuity in a collar}
\label{sec:modcont-collar}
Let~$\delta$ and~$A$ be given by~(\ref{eq:delta-A}), and
\mbox{$R=R(\bh)>0$} be the constant given by Lemma~\ref{lem:R}. Define
functions $\mu\co Q(\delta)\times[0,\delta)\to[0,\br)$ and $\rho \co
Q(\delta)\times[0,\delta) \to[0,\infty)$ by
\[\mu(q,t) := A(t+h_q)\]
(so $\mu(q,t) < 2A\delta = \br$) and
\begin{equation}\label{eq:modcont}
\rho_q(t)=\rho(q,t):=
\begin{cases}
0, & \text{if } t=0;
\\ \\
\dfrac{8Rt}{h_q\cdot
        \exp\left(2\pi\displaystyle\int^\br_{\mu(q,h_q)}\iota(\psi(q);s)\rmd
	s\right)}, 
       & \text{if } 0< t\le h_q;
\\ \\
 \dfrac{8R}{\exp\left(2\pi\displaystyle\int^\br_{\mu(q,t)}
\iota(\psi(q);s)\rmd s\right)}, & \text{if } h_q\le t \le \delta. 
\end{cases}
\end{equation}
It will be shown in this section that, for every $\tq\in\tQ(\delta)$,
$\rho_{\pi(\tq)}$ is a modulus of continuity for~$\phi\co P\to\csph$
  at~$\tq$. 

\begin{rmk}
\label{rmk:rho}
For each~$q\in Q(\delta)$, $\rho_q\co[0,\delta)\to [0,\infty)$ is a
  modulus of continuity in the sense of Definition~\ref{defn:modcont}:
\begin{enumerate}[a)]
\item It is continuous in~$(0,\delta)$ and strictly increasing since: 
\begin{itemize}
\item when $t\in(0,h_q]$, $\rho_q(t)$ is proportional to~$t$; 
\item when $t\in(h_q,\delta)$, $\rho_q(t)$ is inversely proportional
  to $\displaystyle{\exp\left(2\pi\int_{A(t+h_q)}^\br
    \iota(\psi(q);s)\,\rmd s\right)}$, which depends continuously
  on~$t$, and is strictly decreasing since $A(t+h_q)$ is strictly
  increasing and $\iota$ is positive except on a set of measure zero;
  and
\item when $t=h_q>0$, the two expressions for $\rho_q$ clearly agree.
\end{itemize}
\item It is continuous at~$0$ since
\begin{itemize}
\item If $q\in G$ (i.e. $h_q=0$), the second expression
  in~(\ref{eq:modcont}) is used, and the integral in its denominator
  diverges as $t\to 0$ by assumption; and
\item If $q\not\in G$ (i.e. $h_q>0$), the first expression
  in~(\ref{eq:modcont}) is used for $t$ sufficiently small, which
  clearly converges to~$0$ as $t\to 0$. 
\end{itemize}
\end{enumerate}
The (joint) continuity of~$\rho\co
Q(\delta)\times[0,\delta)\to[0,\infty)$ will be proved in
  Lemma~\ref{lem:continuity-rho} below.
\end{rmk}

\begin{example}
For the simplest possible example, suppose that $q\in G$ is a planar
point, and that~$\br$ is small enough that $m(q;r)=4r$ and
$n(q;r)=2$ for all $r\le \br$: hence $\iota(q;r) =
M/6r$. Then~(\ref{eq:modcont}) gives a modulus of continuity
\begin{eqnarray*}
\rho_q(t) &=& \frac{8R}{\exp\left(\displaystyle\int_{At}^\br \frac{\pi M}{3s} \,\rmd s
\right)} \\
&=& 8R \left(\frac{A}{\br}\right)^{\frac{\pi M}{3}} t^{\frac{\pi M}{3}}.
\end{eqnarray*}
Since~$M\le 1/5$ (Definition~\ref{defn:goodnessfn}), the exponent $\pi
M/3$ lies in $(0,1)$. This modulus of continuity is of course not
optimal, since the uniformizing map is Lipschitz near planar points.
\end{example}

\begin{lem}\label{lem:modcontaroundG}
If $\tq\in\tQ(\delta)$, then $\phi\co P\to\csph$ has modulus of
continuity~$\rho_{\pi(\tq)}$ at~$\tq$. That is, if $\tq\in\tQ(\delta)$
and $\tx\in P$ with $d_P(\tq,\tx)<\delta$, then
\[\left|\phi(\tq)-\phi(\tx)\right| \le
\rho_{\pi(\tq)}\left(d_P(\tq,\tx)\right). \]
\end{lem}

\begin{proof}
Notice that $\phi(\tq)$ and $\phi(\tx)$ lie in~$\C$ (since
$\phi^{-1}(\infty)=\tp_\infty\not\in\tQ(\bh)$), and so
$|\phi(\tq)-\phi(\tx)|$ makes sense.

As explained in Remark~\ref{rmk:mod-one-mod-other}, a modulus of
continuity for $u\co S\to\csph$ is also a modulus of continuity for
$\phi=u\circ\pi$, and it therefore suffices to obtain the former. So
let~$q\in Q(\delta)$, and let~$x\in S$ be such that 
\[t:=d_S(q,x) < \delta.\] It is required to show that $|u(q)-u(x)| \le
\rho_q(t)$. 

\noindent\textbf{Case 1: $h_q \le t$}

$x\in D(\psi(q);\mu(q,t))$ by Lemma~\ref{lem:ballindisk}. Assuming at
first that $\mu(q,t)$ is a $\psi(q)$-planar radius, this means that
the annular region $\Ann(\psi(q); \mu(q,t), \br)$ separates~$q$
and~$x$ from $S\setminus Q(\bh/2)$. Hence, by Lemma~\ref{lem:R}, the
image $u\left(\Ann(\psi(q); \mu(q,t), \br)\right)$ separates~$u(q)$
and~$u(x)$ from the circle $\{|z|=R\}$. The modulus of continuity can
thus be obtained from the Gr\"otzsch annulus theorem
(Theorem~\ref{thm:GAT}) just as in Section~\ref{subsec:localmodcont}:
\begin{eqnarray*}
\int_{\mu(q,t)}^\br \iota(\psi(q);s)\,\rmd s &\le &
\mod\Ann(\psi(q);\mu(q,t), \br) \\
&=& \mod u\left(\Ann(\psi(q);\mu(q,t), \br)\right) \\
&\leq& \mod\Gr\left(\left|
\frac{u(q)-u(x)}{R-\frac{u(x)\overline{u(q)}}{R}}
\right|\right) \\
 &\leq& \mod\Gr\left(\frac{|u(q)-u(x)|}{2R}\right)\\
    &\leq& \frac{1}{2\pi}\ln\frac{8R}{|u(q)-u(x)|}
\end{eqnarray*}
(where the M\"obius transformation
$z\mapsto\frac{R(z-u(q))}{R^2-z\overline{u(q)}}$ has been used to move
$u(q)$ to~$0$ and the circle $|z|=R$ to the circle~$|z|=1$), giving
\[|u(q)-u(x)|\leq 
\dfrac{8R}{\displaystyle\exp\left(2\pi\int_{\mu(q,t)}^\br
\iota(\psi(q);s)\rmd s\right)}\]
as required.

If $\mu(q;t)$ is not $\psi(q)$-planar, then increasing it by an
arbitrarily small amount to a planar radius gives the result.

\noindent\textbf{Case 2: $0< t\le h_q$}
Since $x\in D(\psi(q);\mu(q,t)) \subset D(\psi(q);\mu(q,h_q))$, the
annular region $\Ann(\psi(q); \mu(q,h_q), \br)$ separates~$q$ and~$x$
from $S\setminus Q(\bh/2)$ (if $\mu(q,h_q)$ is not a planar radius, then
perturb as in Case~1). Moreover, since $\bB_S(q;h_q) \subset
D(\psi(q); \mu(q,h_q))$ by Lemma~\ref{lem:ballindisk}, the annular
region $B_S(q;h_q) \setminus \bB_S(q;t)$ is nested inside the first
annulus and also separates~$q$ and~$x$ from $S\setminus
Q(\bh/2)$. Arguing as in Case~1, it follows that
\[
\mod\left(B_S(q;h_q)\setminus\bB_S(q;t)\right)+
\mod\Ann\left(\psi(q);\mu(q,h_q),\br\right)
    \leq \frac{1}{2\pi}\ln\frac{8R}{|u(q)-u(x)|}.
\]
Since $B_S(q;h_q)$ is isometric to a plane metric ball,
\[\mod\left(B_S(q;h_q)\setminus\bB_S(q;t)\right)=
\frac{1}{2\pi}\ln\frac{h_q}{t},\]
and hence
\[\frac{1}{2\pi}\ln\frac{h_q}{t}+\int^\br_{\mu(q,h_q)}
\iota(\psi(q);s)\rmd s 
\leq \frac{1}{2\pi}\ln\frac{8R}{|u(q)-u(x)|}
\]
so that 
\[|u(q)-u(x)|\leq 
\dfrac{8Rt}{\displaystyle{h_q}\cdot\exp\left(2\pi\int_{\mu(q,h_q)}^\br
\iota(\psi(q),s)\rmd s\right)}\]
as required.
\end{proof}

In order to use Lemma~\ref{lem:modcontaroundG} to construct a global
modulus of continuity in~$Q(\delta)$, the maximum over the
functions~$\rho_q$ will be taken. To show that this maximum is itself
continuous it is necessary to prove that~$\rho\co
Q(\delta)\times[0,\delta] \to [0,\infty)$ is continuous: this is the
  aim of the remainder of this section.

\begin{lem}
\label{lem:properties-I}
Let~$(P,\cP)$ be a plain paper-folding scheme with only finitely many
singular points in its scar~$G$, at each of which the conditions of
Theorem~\ref{thm:main} hold. Define $I\co G\times(0,\br) \to
(0,\infty)$ by
\[I(q,r) = \int_r^{\br} \iota(q;s) \rmd s.\]
Then
\begin{enumerate}[a)]
\item $I$ is continuous, and
\item For all~$K>0$, there is some~$\eta>0$ such that $I(q,r)>K$ for
  all $q\in G$ and all $r\in(0,\eta)$.
\end{enumerate}
\end{lem}
\begin{proof}
 Recall that
\[I(q,r) = \int_r^{\br} \iota(q;s)\rmd s = M \int_r^{\br} \frac{\rmd
  s}{m(q;s) + s\cdot n(q;s)},\]
where $m(q;s)=\rmm_G(B_G(q;s))$. 

\begin{enumerate}[a)]
\item Let~$(q_0;r_0)\in G\times(0,\br)$:
it will be shown that $I$ is continuous at~$(q_0;r_0)$.
Now for $(q;r)\in G\times(0,\br)$,
\[I(q,r) - I(q_0,r_0) = \int_{r_0}^{\br} \iota(q;s)\rmd s - \int_{r_0}^\br
\iota(q_0;s) \rmd s + \int_r^{r_0} \iota(q;s)\rmd s.\] The final
integral converges to zero as $(q,r)\to (q_0,r_0)$, since $\iota(q;s)$
is bounded above by $M/r_0$ for $s\ge r_0/2$. Hence it suffices
to prove that, for all~$\veps>0$,
\[\left|\int_{r_0}^{\br} \left(\iota(q;s) - \iota(q_0;s)\right)\rmd
s\right|  
= M\left|\int_{r_0}^{\br} \left( 
\frac{1}{m(q;s)+s\cdot n(q;s)} - \frac{1}{m(q_0;s)+s\cdot n(q_0;s)}
\right)\right|
< \veps\]
provided $d_G(q,q_0)$ is sufficiently small.

Because the absolute value of the integrand is bounded above
by~$M/r_0$, a small open interval of radius~$\delta_1$ can be
excised from the range of integration about each of the finitely many
values of~$s$ for which there is a singularity at distance~$s$
from~$q_0$, without changing the integral by more
than~$\veps/3$. Again, since there are now only finitely many values
of~$s$ in the range of integration for which there is a vertex at
distance~$s$ from~$q_0$, further open intervals of
radius~$\delta_2\le\delta_1$ can be excised about each of these values
without changing the integral by more than an additional~$\veps/3$. 

Now if $d_G(q,q_0)<\delta_2/2$ then $n(q_0;s)=n(q;s)$ in each of the
remaining intervals of integration. Moreover, $m(q;s)$ is continuous
in both variables for $q\in B_G(q_0;\delta_2/2)$ and $s$ in a single
interval of integration, and the result follows.

\item Let~$q_0\in G$. Since~(\ref{eq:divint1}) holds at~$q_0$, there
  is some~$\veps>0$ such that
\[M\int_{\veps}^\br \frac{\rmd s}{m(q_0;s) + s\cdot n(q_0;s)} > 3K.\]
The integrand is bounded above by~$M/\veps$ so, as in part~a),
there is a subset~$L$ of $[\veps, \br]$ and a number~$\delta>0$ such
that:
\begin{itemize}
\item $L$ consists of finitely many intervals;
\item $(d_G(q_0,q^*)-\delta, d_G(q_0,q^*)+\delta)$ is disjoint
  from~$L$ for any non-planar point~$q^*$ of~$G$; and
\item \[M\int_L \frac{\rmd s}{m(q_0;s) + s\cdot n(q_0;s)} > 2K.\]
\end{itemize} 
Now if $s\in L$ then $n(q_0;s) = n(q;s)$ for $d_G(q,q_0)<\delta/2$;
and $m$ is continuous in $B_G(q_0;\delta/2) \times L$. There is
therefore some~$\delta'\in(0,\delta/2)$ such that
\[M\int_L \frac{\rmd s}{m(q;s) + s\cdot n(q;s)} > K\]
provided that $d_G(q;q_0)<\delta'$. Hence
\[I(q,r) > \int_{\veps}^{\br}\iota(q;s)\rmd s \ge \int_L\iota(q;s)\rmd
s > K\] for all~$q$ with $d_G(q,q_0)<\delta'$ and all
$r\in(0,\veps)$. The result follows by compactness of~$G$.
\end{enumerate}

\end{proof}

\begin{lem}
\label{lem:continuity-rho}
Let~$(P,\cP)$ be a plain paper-folding scheme with only finitely many
singular points in its scar~$G$, at each of which the conditions of
Theorem~\ref{thm:main} hold. Then the function $\rho\co
Q(\delta)\times[0,\delta) \to [0,\infty)$ of~(\ref{eq:modcont}) is
  continuous. 
\end{lem}
\begin{proof}
Observe that if $\xi\co Q(\delta) \times(0,\delta) \to [1,\infty)$ is
  defined by
\[\xi(q,t) = \max\left(1, \frac{h_q}{t}\right),\]
then~(\ref{eq:modcont}) can be written
\[
\rho(q,t)=
\begin{cases}
0, & \text{if } t=0;
\\ \\
\dfrac{8R}{\xi(q,t)\cdot
        \exp\left(2\pi\displaystyle\int^\br_{\mu(q,t\xi(q,t))}\iota(\psi(q);s)\rmd
	s\right)}, 
       & \text{if } t>0.
\end{cases}
\]
Since $\xi$, $\mu$, and $\psi$ are continuous, the continuity
of~$\rho$ at points~$(q_0,t_0)$ with $t_0>0$ follows from
Lemma~\ref{lem:properties-I}a). 

It remains to prove that~$\rho$ is continuous at all points~$(q_0,0)$
with~$q_0\in Q(\delta)$: that is, that $\rho(q,t)\to 0$ as $(q,t)\to
(q_0,0)$ with $t>0$. 

If $h_{q_0}>0$ then $\xi(q,t) = h_q/t$ for all $(q,t)$ close to
$(q_0,0)$, so that $\xi(q,t)\to\infty$, and hence $\rho(q,t)\to 0$, as
$(q,t)\to(q_0,0)$.

If $h_{q_0}=0$ (i.e. $q_0\in G$) then, since $\xi(q,t)\ge 1$ for
all~$(q,t)$, the result is immediate from
Lemma~\ref{lem:properties-I}b) and the continuity of~$\psi$.
\end{proof}

\subsection{Modulus of continuity in the interior}

In Section~\ref{sec:modcont-collar}, a modulus of continuity for~$\phi$ was
obtained at points~$\tq\in\tQ(\delta)$, where~$\delta$ is given
by~(\ref{eq:delta-A}). In this section, a modulus of continuity is
derived at points $\tq$ in the complement $P_{\delta/2}$ of
$\tQ(\delta/2)$. These overlapping moduli of continuity will then be
glued together in Section~\ref{sec:proof-modcont} to give the required
global modulus.

\begin{lem}
\label{lem:modcontininterior}
Let~$\tq,\tx\in P_{\delta/2}$. Then
\[d_\csph(\phi(\tq),\phi(\tx)) \le \kappa\,d_P(\tq,\tx),\]
where
\begin{equation}
\label{eq:kappa}
\kappa =
2\exp\left(\frac{16\diam_{P_{\delta/2}}P_{\delta/2}}{\delta}\right) \le
2\exp\left(\frac{32\left |\partial P\right|}{\delta}\right)
\end{equation}
(recall that $d_\csph$ denotes the spherical metric on~$\csph$). That is,
$\phi$ is Lipschitz in $P_{\delta/2}$ with a constant which depends
only on the polygon constants of~$P$.
\end{lem}
\begin{proof}
Let~$\Phi\co \Int(P)\to\C$ be defined by $\Phi:=1/\phi$. Then~$\Phi$ is
univalent with $\Phi(\tp_\infty)=0$ and
$\Phi'(\tp_\infty)=1$ by the choice of normalization of~$\phi$. It
follows from Theorem~\ref{thm:Koebedistn} that~$\Phi$ is Lipschitz
in~$P_{\delta/2}$ with constant~$\kappa/2$, where $\kappa$ is as in
the statement of the lemma. Then
\begin{eqnarray*}
d_\csph\left(\phi(\tq),\phi(\tx)\right)&=&
    d_\csph\left(\Phi(\tq),\Phi(\tx)\right) \\
      &\leq&2|\Phi(\tq)-\Phi(\tx)| \\
      &\leq& \kappa\,d_P(\tq,\tx)
\end{eqnarray*}
as required. Finally, observe that
\[\diam_{P_{\delta/2}}P_{\delta/2}\le 2 \left| \partial P\right|\] 
since the path between any two points of~$P_{\delta/2}$ obtained by
connecting each with a horizontal arc to the boundary and then joining
the endpoints of these arcs with the shorter boundary arc between them
has length bounded above by $\frac{1}{2}|\partial P| +
\frac{1}{2}|\partial P| + \frac{1}{2}|\partial P_{\delta/2}| \le
2|\partial P|$ by~(\ref{eq:K}).
\end{proof}

\subsection{Proof of Theorem~\ref{thm:modcont}}
\label{sec:proof-modcont}

Define $\hrho\co [0,\delta) \to (0,\infty)$ by
\[\hrho(t) := 2\max_{q\in Q(\delta)} \rho(q,t),\]
which is well defined since~$\rho$ is continuous and~$Q(\delta)$ is
compact. Then $\hrho$ is a continuous strictly increasing function
with $\hrho(0)=0$ (since each~$\rho_q$ has these properties and~$\rho$
is continuous), and $\phi$ has modulus of continuity $\hrho$ on
$\tQ(\delta)$ with respect to the Euclidean metric on~$\tQ(\delta)$
and the spherical metric on~$\csph$ (the factor~2 arises from the
translation from the Euclidean to the spherical metric).

On the other hand, $\phi$ is $\kappa$-Lipschitz in $P_{\delta/2}$ by
Lemma~\ref{lem:modcontininterior}, where~$\kappa=\kappa(\delta,
|\partial P|)$ is given by~(\ref{eq:kappa}). Hence
$\brho\co[0,\delta)\to[0,\infty)$ defined by
\[\brho(t) := \max\left\{\hrho(t),\,\kappa t\right\}\]
is the desired modulus of continuity.

\section{A dynamical application: convergence to the tight horseshoe}
\label{sec:dynam-appl}

\subsection{Introduction}

This section contains an extended example. It illustrates that it is
practicable to estimate the modulus of continuity provided
by~(\ref{eq:modcont}) uniformly across a family of paper-folding
schemes, and thereby to construct limits which are of interest in
dynamical systems theory. The example is the simplest non-trivial one
known to the authors, but the context in which it arises, which is
described below, provides a wealth of similar cases. In fact, similar
arguments construct an uncountable collection of limiting maps in the
family of unimodal generalized pseudo-Anosov maps defined
in~\cite{ugpa} as will be shown elsewhere. 

\smallskip
The family of {\em tent maps} $T_\lambda\co[0,1]\to[0,1]$, where
$\lambda\in(1,2]$, is defined by
\[
T_\lambda(x)=
\begin{cases}
\lambda(x-1)+2 & \text{if }x\le 1-\frac{1}{\lambda},\\
\lambda(1-x) & \text{if }x\ge 1-\frac{1}{\lambda}.
\end{cases}
\]
Thus $T_\lambda(0)=2-\lambda$, $T_\lambda(1-\frac{1}{\lambda})=1$,
$T_\lambda(1)=0$, and ~$T_\lambda$ has constant slope~$\lambda$ on
$[0,1-\frac{1}{\lambda}]$ and constant slope~$-\lambda$ on
$[1-\frac{1}{\lambda}, 1]$ (see Figure~\ref{fig:tent}).

\begin{figure}[htbp]
\lab{0}{0}{}
\lab{1}{1}{}
\lab{2}{1-\frac{1}{\lambda}}{t}
\lab{3}{2-\lambda}{r}
\begin{center}
\pichere{0.25}{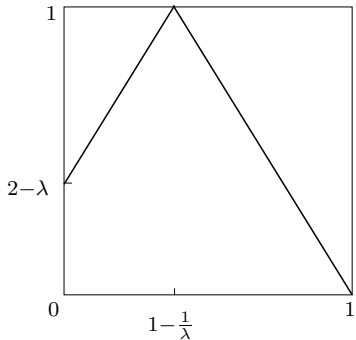}
\end{center}
\caption{Graph of the tent map~$T_\lambda$}
\label{fig:tent}
\end{figure}

Tent maps have been widely studied in the field of unimodal dynamics,
not least because every unimodal map~$f\co[0,1]\to[0,1]$ with positive
topological entropy $h=h(f)$ is semi-conjugate to the tent map of
slope $e^h$~\cite{MT}. In this section, only tent maps with
slopes \mbox{$\lambda\in(\sqrt2,2]$} will be considered.

There is a countable dense subset~$\Lambda\subset(\sqrt2,2]$ of
parameters~$\lambda$ for which $1$ is either a {\em periodic}
($T_\lambda^n(1)=1$ for some~$n\ge 1$) or a {\em preperiodic}
($T_\lambda^k(1)$ is periodic for some~$k\ge 0$) point of~$T_\lambda$.
In~\cite{ugpa} the authors constructed, for each~$\lambda\in\Lambda$,
a complexification $f_\lambda\co S_\lambda\to S_\lambda$
of~$T_\lambda$ as a quasi-conformal automorphism of a complex
sphere~$S_\lambda$. Naturally, one would like to be able to define
such complexifications~$f_\lambda$ for all~$\lambda\in(\sqrt2,2]$ by
taking limits through~$\Lambda$, and this was one of the motivations
for the work described in the current paper. The above spheres~$S_\lambda$
are constructed as paper spheres, and the first step in taking limits
is to identify each~$S_\lambda$ with the Riemann sphere~$\csph$ by a
suitable choice of normalization. In order to construct limits, it is
necessary to show that the uniformizing maps~$u_\lambda\co
S_\lambda\to\csph$ have a uniform modulus of continuity.

The family~$(f_\lambda)_{\lambda\in\Lambda}$ contains a subfamily,
parameterized by rationals~$m/n\in(0,1/2)$, for which the paper
sphere~$S_{m/n}:=S_\lambda$ arises from a paper-folding scheme with
only finitely many segment pairings: these are the so-called {\em NBT}
examples of~\cite{Ha}. In these cases, $f_{m/n}:=f_\lambda$ is a
pseudo-Anosov automorphism: it preserves a transverse pair of measured
foliations with finitely many singularities, stretching one foliation
uniformly by a factor~$\lambda$ and contracting the other uniformly
by~$1/\lambda$ --- the foliations are the projections to~$S_\lambda$
of the horizontal and vertical foliations of the
polygon~$P_\lambda$. For all other $\lambda\in\Lambda$, $f_\lambda$ is
a {\em generalized pseudo-Anosov}: these are defined in the same way
as pseudo-Anosovs, except that their invariant foliations are
permitted to have infinitely many singularities, provided that these
singularities accumulate in only finitely many points.

The simplest example of a generalized pseudo-Anosov is the {\em tight
  horseshoe}, which is the complexification of the {\em full} tent
map~$T_2$. Its sphere~$S=S_2$ of definition is obtained from the plain
paper-folding scheme~$(\Sigma, \cP)$, where~$\Sigma =
[0,1]\times[0,1]\subset\R^2$ and the segment pairings~$\cP$ consist of
two folds of length~$1/2^i$ for each~$i\ge 1$: the top and right sides
of~$\Sigma$ are folded in half, and the bottom and left sides are
covered by folds of lengths~$1/2^i$ for $i\ge 2$, arranged in order of
decreasing length from the bottom right and top left vertices
respectively (Figure~\ref{fig:tight-hs}). As in Example~\ref{ex:both},
the conditions of Theorem~\ref{thm:main} hold at the unique singular
point arising from the identification of the bottom left corner
of~$\Sigma$ with all of the fold endpoints, so that~$S$ has a unique
complex structure induced by the Euclidean structure on~$\Sigma$. The
scar~$G$ is an $\infty$-od, having two edges of length $1/2^i$ for
each $i\ge 1$.

\begin{figure}[htbp]
\labellist
\small
\pinlabel {$\scriptsize{\Sigma}$} at 265 272
\endlabellist
\pichere{0.3}{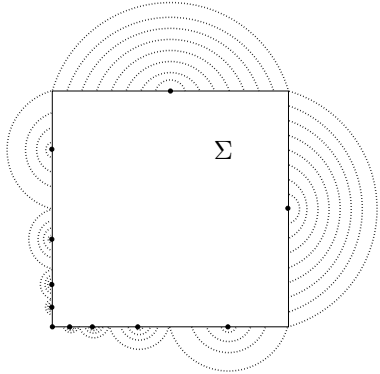}
\caption{The sphere of definition of the tight horseshoe}
\label{fig:tight-hs}
\end{figure}

Let $F:\Sigma\to\Sigma$ be the (discontinuous and non-injective)
function defined by
\[
F(x,y)=
\begin{cases}
(2x, \frac{y}{2}) & \text{ if }x\le \frac12,\\
(2(1-x), 1-\frac{y}{2}) & \text{ if }x>\frac12,
\end{cases}
\]
so that the first coordinate of $F(x,y)$ is $T_2(x)$, and $F$
contracts by a factor~$2$ in the vertical direction. The
identifications on~$\Sigma$ are precisely those which are needed to
make~$F$ continuous and injective, and it therefore descends to a
homeomorphism $f\co S\to S$. This homeomorphism preserves the
foliations on~$S$ which are the projections of the horizontal and
vertical foliations of~$\Sigma$, stretching the former by a factor~$2$
and contracting the latter by a factor~$1/2$: the foliations have
$1$-pronged singularities corresponding to the mid-point of each fold,
which accumulate on the unique singular point. $f$ is called the tight
horseshoe because it can be obtained from Smale's horseshoe
map~\cite{Smale} by collapsing the gaps in the non-wandering Cantor
set. 

The example elaborated in this section treats the convergence of
the NBT pseudo-Anosov homeomorphisms~$f_{1/n}\co S_{1/n}\to S_{1/n}$
(which will henceforth be denoted $f_n\co S_n\to S_n$) to the tight
horseshoe~$f\co S\to S$ as $n\to\infty$: uniformizing maps $u_n\co
S_n\to\csph$ and $u\co S\to\csph$ are chosen so that the
homeomorphisms $u_n\circ f_n\circ u_n^{-1}$ of $\csph$ converge to
$u\circ f\circ u^{-1}$. This will provide a dynamical meaning to the
term {\em tight horseshoe}. It was shown in~\cite{Ha} that the
sequence~$(f_n)$ effectively exhausts the finite invariant sets of the
horseshoe, in the sense that for any finite invariant set~$A$ of the
horseshoe, there is some~$N$ such that, for all~$n\ge N$, $f_n$ is the
pseudo-Anosov representative in the isotopy class of the horseshoe
relative to a finite invariant set~$A'\supset A$: thus the tight
horseshoe is tight in the sense of the introduction to this article.

The paper-folding schemes~$(P_n, \cP_n)$ which provide the
spheres~$S_n$, and the homeomorphisms $f_n\co S_n\to S_n$ will be
described in Section~\ref{sec:origami-surfaces-s_n}. This
application of the results of the article has been chosen so that an
elementary and self-contained description of these spaces and
functions can be given without requiring the machinery of train
tracks. However, the price for eschewing this machinery is that the
constructions are both ad-hoc and involved.  In
Section~\ref{sec:polygon-constants} it is shown that there are uniform
polygon constants for the polygons~$P_n$, which makes it possible, in
Section~\ref{sec:conv-unif-maps}, to obtain a uniform modulus of
continuity for the uniformizing maps~$\phi_n\co P_n\to\csph$ and hence
to complete the proof of convergence.

\subsection{The paper spheres~$S_n$ and the pseudo-Anosov
  homeomorphisms~$f_n\co S_n\to S_n$}
\label{sec:origami-surfaces-s_n}

Let~$\lambda\in(\sqrt2,2]$. The {\em itinerary}
  $k_\lambda(x)\in\{0,C,1\}^\N$ of a point~$x\in[0,1]$ under the tent
  map~$T_\lambda$ is defined by
\[
k_\lambda(x)_i = 
\begin{cases}
0 & \text{if } T_\lambda^i(x)<1-\frac{1}{\lambda},\\
C & \text{if } T_\lambda^i(x)=1-\frac{1}{\lambda},\\
1 & \text{if } T_\lambda^i(x)>1-\frac{1}{\lambda}.
\end{cases}
\]
The {\em kneading invariant} of~$T_\lambda$ is defined as
\[k(T_\lambda) = k_\lambda(1).\]

The kneading theory of Milnor and Thurston~\cite{MT} provides a
way of understanding the dynamics of~$T_\lambda$ by means of its
kneading invariant. In particular, $\lambda$ itself can be recovered
as the reciprocal of the smallest positive root~$\mu$ of a certain
power series. In the case where~$1$ is a period~$N$ point
of~$T_\lambda$, $\mu$ is the smallest positive root of the polynomial
$\sum_{i=0}^{N-1} \theta_i t^i$ whose coefficients
$\theta_i\in\{-1,+1\}$ are given by $\theta_0=1$ and
\[
\theta_i=
\left\{
\begin{array}{rl}
\theta_{i-1} & \quad\text{if }k(T_\lambda)_{i-1} = 0,\\
-\theta_{i-1}  & \quad\text{if }k(T_\lambda)_{i-1} = 1
\end{array}
\right.
\]
for $1\le i < N$.

Let~$n\ge 3$. The value~$\lambda_n$ of~$\lambda$ corresponding to the
$1/n$ NBT case (which are the ones of interest in this section) is the
slope of the tent map with kneading invariant
\[k(T_\lambda)=\left(10^{n-1}1C\right)^\infty\]
($n$ will be fixed throughout most this subsection, and hence the
suffix~$n$ on~$\lambda_n$ will usually be omitted). Then $\mu$ is the
smallest positive root of the polynomial
\[f_n(t) = 1 - t - t^2 - \cdots - t^n + t^{n+1},\]
and $\lambda = 1/\mu$ is the largest root of the same
polynomial. In particular, multiplying through by $(t-1)$, $\lambda$
satisfies
\begin{equation}
\label{eq:lambda}
\lambda^{n+2}-2\lambda^{n+1}+2\lambda-1 = 0.
\end{equation}

 $(\lambda_n)$ is therefore an increasing sequence with $\lambda_n\to
2$ as $n\to\infty$.

Since $k(T_\lambda)=(10^{n-1}1C)^\infty$, $1$ is a period~$n+2$ point
of~$T_\lambda$: write $p_0=1$ and $p_{i}=T_\lambda^i(p_0)$ for $1\le
i\le n+1$ so that $T_\lambda(p_{n+1})=p_0$. The following explicit
description of the points~$p_i$ is easily verified by induction:

\begin{equation}
\label{eq:pi}
p_i = 
\begin{cases}
1 & \text{if } i = 0,\\
0 & \text{if } i = 1,\\
\frac{(2-\lambda)(\lambda^{i-1}-1)}{\lambda-1} & \text{if } 2\le i\le
n, \\
1-\frac{1}{\lambda} & \text{if }i = n+1.
\end{cases}
\end{equation}

Now define a function $F_n\co \Sigma\to\R^2$ by
\[
F_n(x,y) = 
\begin{cases}
\left(
\lambda_n(x-1)+2\,,\, \frac{y}{\lambda_n} - \frac{1}{\lambda_n^{n+1}+1}
\right) 
& \text{ if } x\le 1-\frac{1}{\lambda_n},\\
\left(
\lambda_n(1-x)\,,\, 1-\frac{y}{\lambda_n}
\right)
& \text{ if } x> 1-\frac{1}{\lambda_n}
\end{cases}
\]
% Note that 1/(lambda^{n+1}-1) = -(lambda^{n+2}-2lambda^{n+1}+1)/(lambda^{n+2}-1)
so that the first coordinate of $F_n(x,y)$ is $T_{\lambda_n}(x)$,
and~$F_n$ contracts by a factor $1/\lambda_n$ in the~$y$ direction
(see Figure~\ref{fig:F-action}). Notice that~$F_n$ is not
injective, and that it is discontinuous across the
line~$x=1-1/\lambda_n$.

\begin{figure}[htbp]
\labellist
\small
\pinlabel {$0$} [tr] at 0 9
\pinlabel {$1$} [r] at 0 236
\pinlabel {$1$} [t] at 227 9
\pinlabel {$1-\frac{1}{\lambda_n}$} [t] at 106 9
\pinlabel {$A$}  at 49 117
\pinlabel {$B$}  at 162 117
\pinlabel {$F(A)$}  at 424 58
\pinlabel {\rotatebox{180}{$F(B)$}}  at 403 177
\endlabellist
\pichere{0.6}{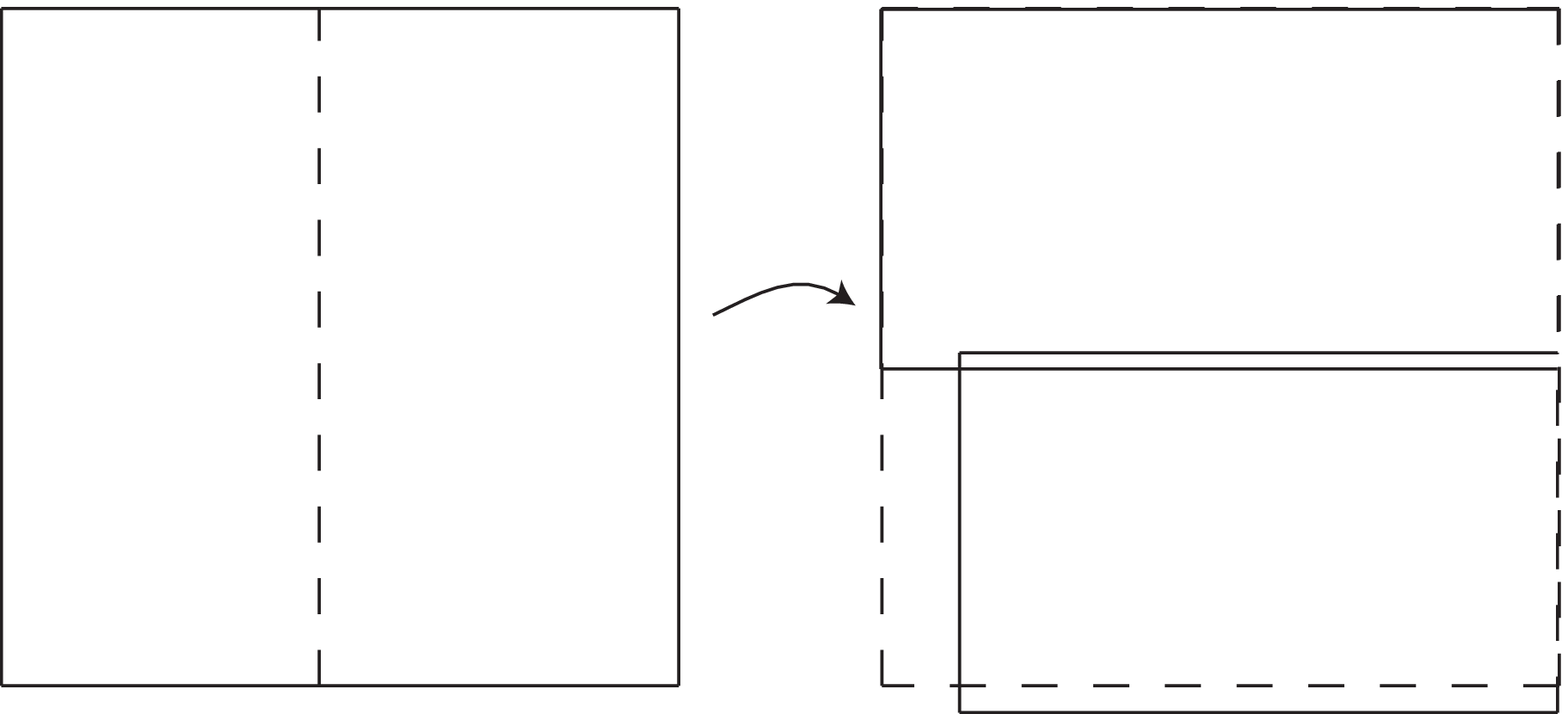}
\caption{The action of~$F_n\co \Sigma\to\R^2$}
\label{fig:F-action}
\end{figure}

The polygon~$P_n$ in the paper-folding scheme~$(P_n,\cP_n)$ is chosen
so that $F_n(P_n)=P_n$ and~$F_n$ is injective on its
interior. The segment pairings~$\cP_n$ are then chosen to make the
induced action of~$F_n$ on the quotient paper sphere~$S_n$
continuous and injective. The resulting homeomorphism $f_n\co S_n\to
S_n$ will be a pseudo-Anosov map, with $n+2$ one-pronged singularities,
and one $n$-pronged singularity.

$P_n$ has $n+2$ horizontal sides and $n+2$ vertical sides. The
vertical sides $V_i$ ($0\le i\le n+1$) are contained in the vertical
fibers over the points~$p_i$ of the orbit of~$1$ under the tent
map~$T_\lambda$: they are defined by
\[
V_i = 
\begin{cases}
\{1\}\times \left[ 0, \frac{\lambda^{n+1}}{\lambda^{n+1}+1}\right] & \text{ if } i=0,\\
F_n^i(V_0) & \text{ if } 1\le i\le n+1.
\end{cases}
\]

By~(\ref{eq:lambda}), the height~$h$ of~$V_0$ satisfies
\[h = \frac{\lambda^{n+1}}{\lambda^{n+1}+1} = 
\frac{2\lambda^{n+1}(\lambda-1)}{\lambda^{n+2}-1},\] so that the sum
$\sum_{i=0}^{n+1}h/\lambda^i$ of the heights of the vertical sides is
equal to~2.

The essential point about the configuration of these vertical sides is
contained in the following lemma (see Figure~\ref{fig:vertical}).

\begin{figure}[htbp]
\labellist
\small
\pinlabel {$0$} [tr] at 0 0
\pinlabel {$1$} [r] at 0 500
\pinlabel {$1$} [t] at 500 0
\pinlabel {$V_0$} [l] at 500 240
\pinlabel {$V_1$} [r] at 0 376
\pinlabel {$V_2$} [r] at 59 175
\pinlabel {$V_3$} [r] at 168 74
\pinlabel {$V_4$} [r] at 376 18
\pinlabel {$V_5$} [l] at 234 490
\endlabellist
\pichere{0.4}{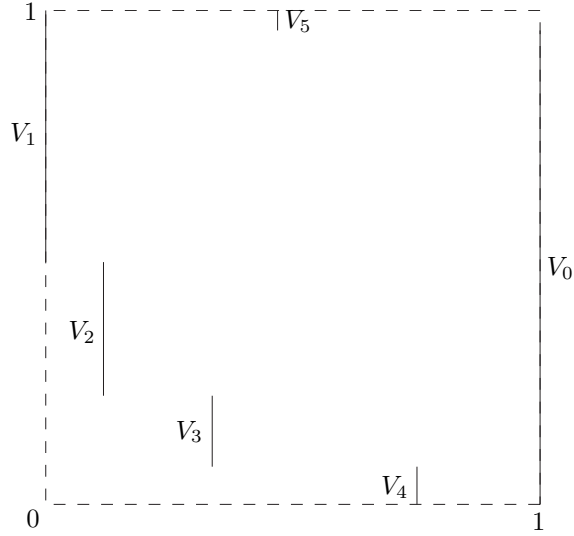}
\caption{The vertical sides of~$P_4$}
\label{fig:vertical}
\end{figure}

\begin{lem}
\label{lem:vertical-configuration}
Let $\pi_y\co \Sigma\to [0,1]$ be projection onto the second coordinate. Then
\begin{enumerate}[a)]
\item $\pi_y(V_0\cup V_{n+1})=[0,1]$, and $\pi_y(V_0) \cap \pi_y(V_{n+1})$ is a single point.
\item $\pi_y(V_1\cup \cdots \cup V_n) = [0,1]$; and for $1\le i<j\le
  n$, $\pi_y(V_i)\cap \pi_y(V_j)$ is a single point if $j=i+1$, and is
  empty otherwise.
\end{enumerate}
\end{lem}

\begin{proof}
Recall that $V_0$ has height 
\[
|V_0|=h=\frac{\lambda^{n+1}}{\lambda^{n+1}+1}.
\]
Since $F_n(1,0)=(0,1)$ and $|V_1|=h/\lambda$, it follows that
\[V_1 = \{0\} \times \left[ 1-\frac{h}{\lambda}, 1\right].\]

Now, using~(\ref{eq:lambda}),
\[F_n(0,1) = 
\left(p_2,\frac{1}{\lambda}-\frac{1}{\lambda^{n+1}+1} \right) = 
\left(p_2, 1-\frac{h}{\lambda}\right).
\]
That is, the top point of~$V_2$ is on the same horizontal level as the
bottom point of~$V_1$. Since~$F_n$ preserves horizontal and
vertical lines in $x\le 1-1/\lambda$, the top point of~$V_{i+1}$ is on
the same horizontal level as the bottom point of~$V_i$ for $1\le i\le
n-1$. Since
\[\sum_{i=1}^n |V_i| = h\sum_{i=1}^n \frac{1}{\lambda^i} =
\frac{\lambda^{n+1}-\lambda}{\lambda^{n+2}-\lambda^{n+1}+\lambda-1} =
1\qquad\text{(by~(\ref{eq:lambda}))},\]
statement~b) of the lemma follows.

In particular, 
\[V_n = \{p_n\}  \times \left[0, \frac{h}{\lambda^n} \right],\]
and since $F_n(p_n,0) = (p_{n+1},1)$ it follows that
\[V_{n+1} = \{p_{n+1}\} \times \left[ 1-\frac{h}{\lambda^{n+1}}, 1 \right].\]
Now $1-h/\lambda^{n+1} = \lambda^{n+1}/(\lambda^{n+1}+1)$, the vertical coordinate of the top
of~$V_0$, so statement~a) follows also.
\end{proof}

Horizontal sides $H_0,\ldots,H_{n+1}$ can therefore be added to join
the endpoints of the vertical sides and bound the
polygon~$P_n$. Specifically, let
\begin{itemize}
\item $H_0$ join the top of~$V_0$ to the bottom of~$V_{n+1}$;
\item $H_i$ join the bottom of~$V_i$ to the top of~$V_{i+1}$ for $1\le i\le n-1$;
\item $H_n$ join the bottom of~$V_n$ to the bottom of $V_0$, and
\item $H_{n+1}$ join the top of~$V_{n+1}$ to the top of~$V_1$
\end{itemize}
(see Figure~\ref{fig:P4}). In particular, $H_{n-1}$ is the only
horizontal side which intersects the line of
discontinuity~$x=p_{n+1}=1-1/\lambda$ in its interior.

\begin{figure}[htbp]
\labellist
\small
\pinlabel {$V_0$} [r] at 250 114
\pinlabel {$V_1$} [l] at 0 186
\pinlabel {$V_2$} [l] at 29 89
\pinlabel {$V_3$} [l] at 84 37
\pinlabel {$V_4$} [l] at 188 9
\pinlabel {$V_5$} [r] at 117 244
\pinlabel {$H_0$} [b] at 181 240
\pinlabel {$H_1$} [t] at 12 124
\pinlabel {$H_2$} [t] at 53 55
\pinlabel {$H_3$} [t] at 137 19
\pinlabel {$H_4$} [t] at 215 0
\pinlabel {$H_5$} [b] at 54 250
\pinlabel {$A$}  at 61 159
\pinlabel {$B$}  at 181 123
\pinlabel {$F_4(A)$}  at 450 70
\pinlabel {\rotatebox{180}{$F_4(B)$}} at 435 181
\pinlabel {$I$} [b] at 449 123
\pinlabel {$F_4(V_5)$} [tr] at 530 102
\endlabellist
\pichere{0.95}{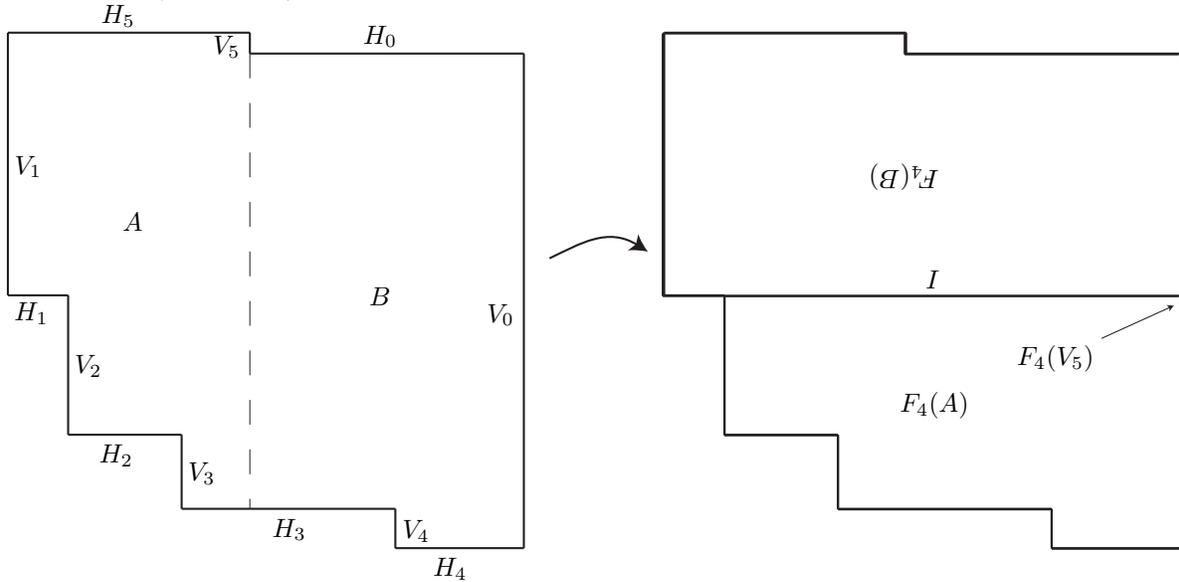}
\caption{The polygon~$P_4$ and its image under~$F_4$}
\label{fig:P4}
\end{figure}

Next, the segment pairings on~$\partial P_n$ which make~$F_n$
continuous and injective will be described. The pairings on the
vertical sides are straightforward and are treated first.

\begin{lem}
\label{lem:vertical-coincidence}
The image under~$F_n$ of the midpoint of~$V_{n+1}$ is the
midpoint of~$V_0$.
\end{lem}
\begin{proof}
The midpoint of~$V_{n+1}$ has vertical coordinate
$1-h/2\lambda^{n+1}$, and so its image has vertical coordinate
\[\frac{1-h/2\lambda^{n+1}}{\lambda} - \frac{1}{\lambda^{n+1}+1} = 
\frac{1}{\lambda} - \frac{1}{2\lambda(\lambda^{n+1}+1)} -
\frac{1}{\lambda^{n+1}+1}.\] 
Using~(\ref{eq:lambda}), this is equal to
$\frac{\lambda^{n+1}}{2(\lambda^{n+1}+1)}$, the vertical coordinate of
the midpoint of~$V_0$.
\end{proof}

Define segment pairings on the vertical sides of~$P_n$ by folding each
side~$V_i$ about its midpoint. Since $V_i$ is mapped
affinely onto~$V_{i+1}$ for $0\le i\le n$, and $V_{n+1}$ is mapped
affinely into~$V_0$ with the midpoint being sent to the midpoint, the
identifications arising from these pairings are respected by the
action of~$F_n$. Moreover, the identifications make~$F_n$
continuous across the line~$x=1-1/\lambda$.

Write $\tP_n$ for the quotient space under these identifications on
the vertical sides (a disk with boundary $\bigcup H_i$), so that
$F_n$ induces a continuous surjection, $\tF_n\co
\tP_n\to\tP_n$.

The purpose of the segment pairings on the horizontal sides of~$P_n$
is to make~$\tF_n$ injective: they are more complicated to
describe. Observe first how~$F_n$ acts on these sides:
\begin{itemize}
\item $F_n(H_0)=I\cup H_1$, where~$I$ is the horizontal segment
 joining the right hand end of~$H_1$ to $V_0$ (see
  Figure~\ref{fig:P4}). Write $H_0=H_0'\cup H_0''$, where
  $F_n(H_0')=I$ and $F_n(H_0'')=H_1$;
\item $F_n(H_i) = H_{i+1}$ for $1\le i \le n-2$;
\item $F_n(H_{n-1}) = H_n \cup H_0$ (recall that~$H_{n-1}$ crosses the
  line of discontinuity~$x=1-1/\lambda$);
\item $F_n(H_n) = H_{n+1}$; and
\item $F_n(H_{n+1}) = I$.
\end{itemize}

In order to make~$F_n$ injective, it is therefore necessary to
identify $H_{n+1}$ with $H_0'$, and then to propagate this
identification under the dynamics.

Writing $I\to J$ to mean $F_n(I)\supset J$, observe that
\[H_0'' \to H_1\to H_2 \to \cdots \to H_{n-1} \to H_0''.\]
There is therefore a period~$n$ point $q_0$ of~$F_n$ in $H_0''$, the
points $q_i=F_n^i(q_0)$ ($1\le i<n$) of whose orbit satisfy $q_i\in
H_i$, and are therefore ordered cyclically according to their indices
around the boundary of $\tP_n$.

\begin{lem}
\label{lem:q0}
The horizontal coordinate of~$q_0$ is 
$\displaystyle{1-\frac{2-\lambda}{\lambda(\lambda+1)}}$.
Moreover, the distance from the left hand endpoint of~$H_0$ to~$q_0$
is equal to the length of~$H_{n+1}$ plus the distance from the left
hand endpoint of~$H_1$ to~$q_1$.
\end{lem}
\begin{proof}
Let~$L$ be the segment of~$H_0$ from its left hand end to $q_0$, and
$M$ be the segment of~$H_1$ from~$q_1$ to its right hand end (see
Figure~\ref{fig:idents}). Then $F_n(L) = I\cup M$ and
$F_n^{n-1}(M)=L$. Thus
\[\lambda |L|= |I| + |M| = \lambda - 1 + |M|\]
(using that the right hand endpoint of~$H_1$ has horizontal coordinate
$2-\lambda$), and
\[\lambda^{n-1}|M| = |L|.\]

\begin{figure}[htbp]
\labellist
\small
\pinlabel {$q_0$} [t] at 557 521
\pinlabel {$q_1$} [b] at 93 297
\pinlabel {$q_2$} [b] at 167 163
\pinlabel {$q_3$} [b] at 327 90
\pinlabel {$L$} [t] at 440 527
\pinlabel {$\scriptsize{M}$} [l] at 190 326
\endlabellist
\pichere{0.5}{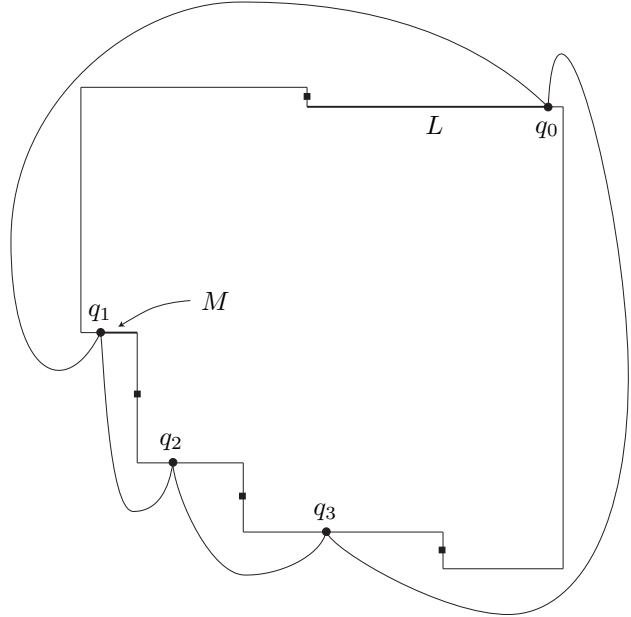}
\caption{Identifications on the horizontal sides of~$P_4$.}
\label{fig:idents}
\end{figure}

These give $|L| = \lambda^{n-1}(\lambda-1)/(\lambda^n-1)$, and hence
the horizontal coordinate~$\xi$ of $q_0$, given by $\xi =
1-\frac{1}{\lambda}+|L|$, is
\[\xi = 1-\frac{1}{\lambda}+\frac{\lambda^{n-1}(\lambda-1)}{\lambda^n-1}\]
which simplifies to the required expression using~(\ref{eq:lambda}).

Now $|H_{n+1}| = 1-1/\lambda$, and the distance from the left hand
endpoint of~$H_1$ to~$q_1$ is
\[
|H_1| - |M| = 2-\lambda-\frac{|L|}{\lambda^{n-1}} =
2-\lambda-\frac{\lambda-1}{\lambda^n-1},
\]
and the sum of these simplifies, using~(\ref{eq:lambda}), to~$|L|$ as
required.
\end{proof}

It follows that identifying~$L$ with the union of~$H_{n+1}$ and the
segment of~$H_1$ between its left hand endpoint and~$q_1$ identifies
$H_{n+1}$ with $H_0'$ as required. Propagating this identification
under the dynamics requires identifying the two halves of the segment
of the boundary of~$\tP_n$ between $q_i$ and $q_{i+1}$ for each $0\le
i<n-1$, and the two halves of the segment between $q_{n-1}$ and $q_0$
(Figure~\ref{fig:idents}). Each of these identifications can be
realised as a segment pairing on~$P_n$, with the exception of the
identification of~$L$ with $H_{n+1}$ and a segment of~$H_1$, and the
identification of the two halves of the segment between $q_{n-1}$ and
$q_0$, each of which can be expressed as two segment pairings. There
is therefore a total of~$n+2$ horizontal segment pairings, in addition
to the~$n+2$ vertical segment pairings.

The paper surface~$S_n = P_n/d_{P_n}^{\cP_n}$ is a sphere, on which
the horizontal and vertical foliations of~$P_n$ descend to a pair of
transverse measured foliations, with $n+2$ $1$-pronged singularities
(at the midpoints of the vertical sides) and an $n$-pronged
singularity (at the points of the period~$n$ orbit~$Q$ on the
horizontal boundary, which are all identified). Since~$F_n$ stretches
in the horizontal direction by a factor~$\lambda$, and contracts in
the vertical direction by a factor~$1/\lambda$, it induces a
pseudo-Anosov homeomorphism~$f_n\co S_n\to S_n$ of the paper surface.

The scar~$G_n$ of~$(P_n, \cP_n)$ is a tree with~$2n+4$ edges
corresponding to the $2n+4$ segment pairings in~$\cP_n$
(Figure~\ref{fig:scar}): however, $n$ of its vertices (the projections
of the endpoints of~$V_i$ for $2\le i \le n+1$) have valence~2. It
will be convenient to describe it in the following way. Let~$q_0\in
G_n$ denote the projection of the periodic
orbit~$\{q_0,\ldots,q_{n-1}\}$ in~$P_n$: $q_0$ will be called the {\em
  central vertex} of~$G_n$. There are~$n$ {\em horizontal edges}
$h_0,\ldots, h_{n-1}$ of~$G_n$ which emanate from~$q_0$ and have
lengths $|h_i|=\alpha\lambda^i$, where
$\alpha=(\lambda-1)/(\lambda^n-1)$ so that
$\sum_{i=0}^{n-1}\alpha\lambda^i=1$. Similarly there are~$n+2$ {\em
  vertical edges} $v_0,\ldots, v_{n+1}$ of~$G_n$, which have lengths
$|v_i|=\beta\lambda^i$, where $\beta =
(\lambda-1)/(\lambda^{n+2}-1)$. For $0\le i<n$, the edge~$v_i$
emanates from the end of~$h_{n-1-i}$. On the other hand $v_n$
(corresponding to~$V_1$) emanates from~$h_{n-1}$ at
distance~$\alpha/\lambda$ from~$q_0$ (this distance comes from
Lemma~\ref{lem:q0}); and~$v_{n+1}$ (corresponding to~$V_0$) emanates
from $h_{n-2}$ at distance~$\alpha/\lambda^2$ from~$q_0$. In
particular, $h_{n-2}$ and $h_{n-1}$ are strictly speaking unions of
two edges of~$G_n$.

\begin{figure}[htbp]
\labellist
\small
\pinlabel {$v_0$} [l] at 98 7
\pinlabel {$v_1$} [t] at 9 149
\pinlabel {$v_2$} [r] at 91 206
\pinlabel {$v_3$} [b] at 135 154
\pinlabel {$v_4$} [tr] at 132 128
\pinlabel {$v_5$} [br] at 49 90
\pinlabel {$h_0$} [bl] at 118 173
\pinlabel {$h_1$} [l] at 95 176
\pinlabel {$h_2$} [b] at 44 154
\pinlabel {$h_3$} [l] at 94 79
\endlabellist
\pichere{0.3}{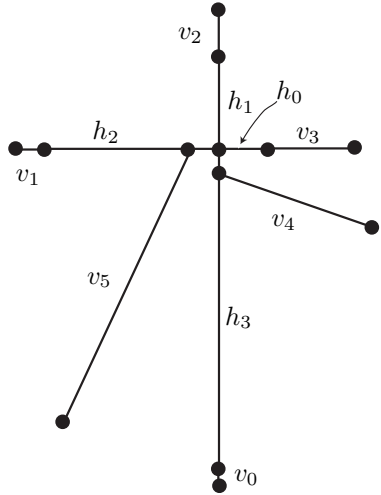}
\caption{The scar of the $P_4$ paper surface}
\label{fig:scar}
\end{figure}

The final lemma of this section shows that the polygons~$P_n$ converge
to the square~$\Sigma$ (see Figure~\ref{fig:hausdorff-converge}), the
equivalence relations~$\sim_{\cP_n}$ converge to the equivalence
relation~$\sim_\cP$ of the tight horseshoe, and the
functions~$F_n\co P_n \to P_n$ converge in an appropriate
sense to $F\co\Sigma\to\Sigma$. From this point on, superscripts~$n$
will be added to indicate dependence on~$n$ where they were previously
omitted: thus, for example, the horizontal and vertical sides of~$P_n$
will be denoted $H^n_i$ and $V^n_i$ for $0\le i\le n+1$.

\begin{figure}[htbp]
\pichere{0.95}{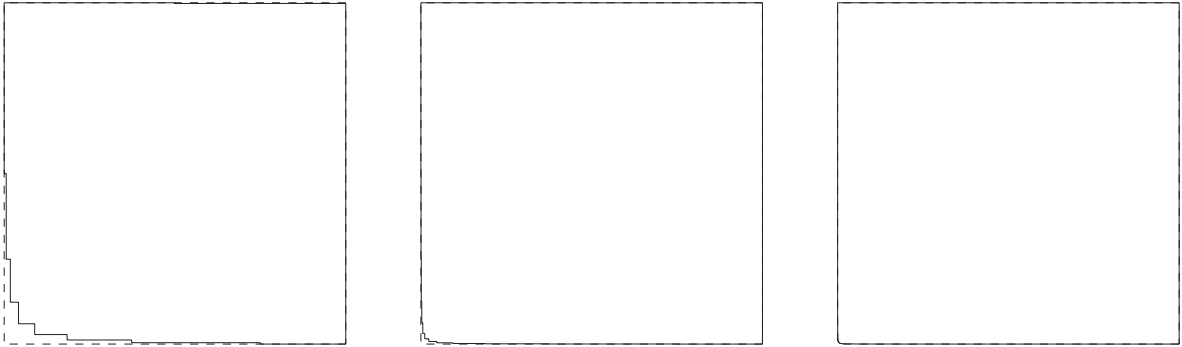}
\caption{The polygons~$P_n$ (here $n=8,\,12,\,16$) converge to~$\Sigma$ (shown
with dashed lines)}
\label{fig:hausdorff-converge}
\end{figure}

\begin{lem}
\label{lem:topological-convergence}
\begin{enumerate}[a)] \mbox{}
\item Let~$\veps>0$. Then the square $\Sigma_\veps = [\veps,
  1-\veps]\times[\veps, 1-\veps]$ is contained in~$P_n$ for all
  sufficiently large~$n$. In particular, the polygons~$P_n$ converge
  in the Hausdorff metric to~$\Sigma$.
\item The equivalence relations $\sim_{\cP_n}$ converge in the
  Hausdorff metric to the equivalence relation~$\sim_\cP$.
\item Let $(x_n,y_n)\to(x,y)$ be a convergent sequence in~$\Sigma$
  with $(x_n,y_n)\in P_n$ for all~$n$. If $x\not=1/2$ then
  $F_n(x_n,y_n)\to F(x,y)$. If $x=1/2$ then there is a subsequence
  $(x_{n_i}, y_{n_i})$ such that $F_{n_i}(x_{n_i}, y_{n_i}) \to z\in\Sigma$
  where $z\sim_{\cP} F(x,y)$.
\end{enumerate}
\end{lem}

\begin{proof}\mbox{}
\begin{enumerate}[a)]
\item It is enough to show that all of the sides of~$P_n$ are
  contained in $C_\veps = \Sigma\setminus\Sigma_\veps$.

$V_0^n$, $V_1^n$, $H_n^n$, and $H_{n+1}^n$, being subsets
of~$\partial\Sigma$, are contained in~$C_\veps$ for all~$n$.

Since~$\lambda_n>3/2$ for all~$n$ (as $\lambda_3\simeq 1.722$), the
heights of the vertical sides satisfy $|V^n_i|<(2/3)^i$. Pick~$K$
large enough that $\sum_{i=K}^\infty (2/3)^i < \veps$: then the
vertical side $V_i^n$ is contained
in~$[0,1]\times([0,\veps)\cup(1-\veps,1])\sbs C_\veps$ whenever~$i>K$:
this in turn means that $H_i^n\sbs C_\veps$ whenever~$i\ge K$, and
$H_0^n \sbs C_\veps$ provided that $n>K$.

Now pick~$N>K$ large enough that $2-\lambda_N < \veps/2^{K+1}$, and
suppose that $n\ge N$. Then, by~(\ref{eq:pi}), whenever~$2\le i\le K$
the side~$V_i^n$ has horizontal coordinate
\begin{eqnarray*}
p_i^n &=&  \frac{(2-\lambda_n)(\lambda_n^{i-1}-1)}{\lambda_n-1} \\
&<& 2(2-\lambda_n)(\lambda_n^{i-1}-1) \qquad\text{(as
  $\lambda_n>3/2$)} \\
&<& \veps(\lambda_n^{i-1}-1)/2^K < \veps,
\end{eqnarray*}
and hence $V_i^n\sbs[0,\veps)\times[0,1]\sbs C_\veps$: this in turn
means that $H_{i-1}^n \sbs C_\veps$.
\item Let~$\veps>0$. It is required to show that there is some~$N$
  such that, for all~$n\ge N$:
\begin{enumerate}[i)]
\item If $a,b\in\Sigma$ with $a\sim_{\cP} b$, then there exist $a',b'\in
  P_n$ with $a'\sim_{\cP_n} b'$ and $|a-a'|<\veps$,
  $|b-b'|<\veps$.
\item If $a,b\in P_n$ with $a\sim_{\cP_n}b$, then there exist
  $a',b'\in \Sigma$ with $a'\sim_{\cP}b'$ and $|a-a'|<\veps$,
  $|b-b'|<\veps$.  
\end{enumerate}
There are four cases to consider:

\noindent{\bf Case 1:} $a=b$. This is dealt with by~a).

\noindent{\bf Case 2:} $a$ and $b$ are an interior pair of a vertical
segment pairing.

Choose~$K$ so that $\sum_{i=1}^K|V_i^n| > 1-\veps/2$ for
all~$n>K$. For each $i$ with $0\le i\le K$, $|V_i^n| =
\lambda_n^{n+1-i}/(\lambda_n^{n+1}+1)$ converges to $1/2^i$ as
$n\to\infty$. Thus, for sufficiently large~$n$: 
\begin{itemize}
\item $V_0^n$ is within Hausdorff distance $\veps/2$ of
  $\{1\}\times[0,1]$;
\item $V_i^n$ is within Hausdorff distance $\veps/2$ of $\{0\}
  \times [1/2^i, 1/2^{i-1}]$ for $1\le i \le K$; and
\item $|V_i^n| < \veps/2$ for $K<i\le n+1$.
\end{itemize}
That is, each folded vertical segment in~$P_n$ (respectively~$\Sigma$)
is either very close to a folded vertical segment in~$\Sigma$
(respectively~$P_n$), or is very small. In the former case, $(a,b)$
can be approximated by some~$(a',b')$; and in the latter case, it
can be approximated by some~$(a', a')$.

\noindent{\bf Case 3:} $a$ and $b$ are an interior pair of a horizontal
  segment pairing.

The argument is similar to that of the second case, except that (see
Figure~\ref{fig:idents}) the relevant quantities are the horizontal
coordinates of the points~$\{q_0^n, q_1^n, \ldots, q_{n-1}^n\}$ of the
periodic orbit of~$F_n$ on $\partial P_n$. When~$n$ is
large, the horizontal coordinate of~$q_0^n$ is very close to~$1$, and
the horizontal coordinates of $q_i^n$ for $1\le i < n$ are either
very small, or very close to $1/2^{n-i}$.

\noindent{\bf Case 4:} $a$ and $b$ are endpoints of segment pairings.

The argument is an extension of the third case. When~$n$ is large, the
set of segment endpoints on~$\partial\Sigma$ is Hausdorff close to
the union of the periodic orbit of~$F_n$ on~$\partial P_n$
with the points~$(1,0)$, $(0,1)$ of~$\partial P_n$: these two points
are very close to being identified with $q_0^n$ and $q_1^n$
respectively.

\item If $x<1/2$ (respectively $x>1/2$) then $x_n<1-1/\lambda_n$
  (respectively $x_n>1-1/\lambda_n$) for all sufficiently large~$n$,
  and the result is immediate from the definitions of $F_n$ and $F$.

If $x=1/2$ then $F(x,y) = (1, y/2)$. Take a subsequence $(x_{n_i},
y_{n_i})$ such that either $x_{n_i}<1-1/\lambda_{n_i}$ for all~$i$, or
$x_{n_i} > 1-1/\lambda_{n_i}$ for all~$i$. Then
$F_{n_i}(x_{n_i},y_{n_i})$ converges to $(1,y/2)$ in the former case,
and to $(1, 1-y/2) \sim_{\cP} (1,y/2)$ in the latter case.

\end{enumerate}
\end{proof}

\subsection{Polygon constants}
\label{sec:polygon-constants}

Recall (Definition~\ref{defn:constsofP}) that the polygon constants of
a polygon~$P$ are: the length~$|\partial P|$ of its boundary; the
collaring height~$\bh>0$ (Section~\ref{sec:the-collar-tq}); and
$\br>0$, which in the plain case can be chosen arbitrarily. In this
section it is shown that these constants can be chosen uniformly for
the polygons~$P_n$. Since $|\partial P_n|=4$ for all~$n$, and $\br$ is
arbitrary, the only issue is to find a uniform collaring
height~$\bh$. It will be shown that $\bh=\frac{1}{24}$ is a collaring
height for all~$n$, and the constant $\br$ will then be chosen as
\[\br = \bh = \frac{1}{24}.\]

Observe (Figure~\ref{fig:polygon-constants}) that the trapezoid of
height~$\bh=1/24$ on every side of~$P_n$ except for $V_0^n$, $V_1^n$,
$H_n^n$, and $H_{n+1}^n$ is a parallelogram. However, the lengths of
these four exceptional sides are uniformly bounded below:
\begin{eqnarray*}
|V_0^n| &=& \frac{\lambda_n^{n+1}}{\lambda_n^{n+1}+1}> \frac{1}{2},\\
|V_1^n| &=& \frac{|V_0^n|}{\lambda_n} > \frac{1}{4},\\
|H_{n+1}^n| &=& 1-\frac{1}{\lambda_n} > \frac{1}{3},
\qquad\text{and}\\
|H_n^n| &=& \frac{|H_{n+1}^n|}{\lambda_n} > \frac{1}{6}.
\end{eqnarray*}
Hence the ratio between the lengths of the bases and the tops of the
trapezoids on these sides lies in~$[1/2,2]$ since $\bh \le
\frac{1}{6}\times\frac{1}{4} = \frac{1}{24}$. (Since the base angles
are $\pi/4$, the difference between the length of the base and the
length of the top is $2\bh$.) It therefore only remains to show that
the trapezoids of height $\bh$ only intersect along their common
vertical sides. This is an immediate consequence of the observation
that the difference between the vertical coordinates of the top of
$V_0^n$ and the bottom of $V_1^n$ is 
\[
\frac{\lambda_n^{n+1}}{\lambda_n^{n+1}+1} -
\left(1-\frac{\lambda_n^n}{\lambda_n^{n+1}+1}\right)
=\frac{1-1/\lambda_n^n}{\lambda_n+1/\lambda_n^n} >
\frac{1-(2/3)^3}{2+1}=\frac{19}{81} > 2\bh.\]

\begin{figure}[htbp]
\labellist
\small
\pinlabel {$V_0$} [l] at 500 258
\pinlabel {$V_1$} [r] at 0 375
\pinlabel {$H_6$} [t] at 440 0
\pinlabel {$H_7$} [b] at 116 500
\endlabellist
\pichere{0.4}{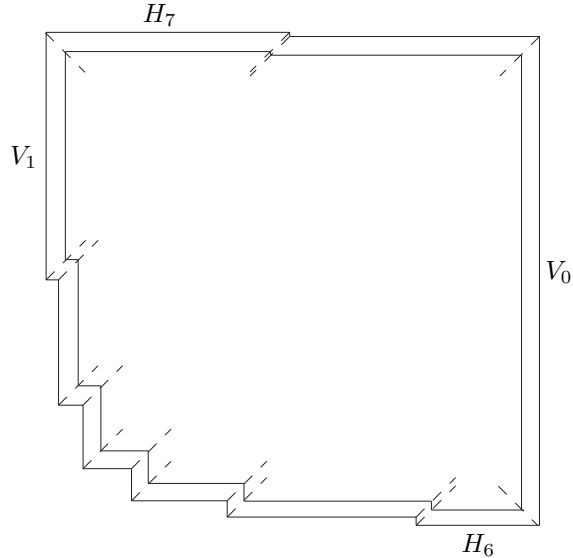}
\caption{Collaring~$P_n$ (here $n=6$)}
\label{fig:polygon-constants}
\end{figure}

\subsection{Convergence of the uniformizing maps}
\label{sec:conv-unif-maps}

Since the scars~$G_n$ have no singular points, the paper spheres~$S_n$
all have a natural complex structure, and there are unique
uniformizing maps $u_n\co S_n\to\csph$ such that the compositions
$\phi_n=u_n\circ\pi_n\co P_n\to\csph$ satisfy $\phi_n(1,0)=0$,
$\phi_n(1/2,1/2)=\infty$, and $(1/\phi_n)'(1/2,1/2)=1$ (the
point~$(1/2,1/2)$ lies in the complement of the height~$\bh$ collaring
of~$P_n$ for sufficiently large~$n$ by
Lemma~\ref{lem:topological-convergence}a): in fact it can be shown
that this is true for all~$n\ge 3$). The pseudo-Anosov homeomorphisms
$f_n\co S_n\to S_n$ induce homeomorpisms $\widehat{f}_n\co \csph
\to\csph$ by $\widehat{f}_n=u_n\circ f_n\circ u_n^{-1}$.

Similarly, recalling that $(\Sigma,\cP)$ is the paper-folding scheme
used in the definition of the tight horseshoe $f\co S\to S$, there is
a unique uniformizing map $u\co S\to\csph$ with the property that the
composition $\phi=u\circ\pi\co \Sigma\to\csph$ satisfies
$\phi(1,0)=0$, $\phi(1/2,1/2)=\infty$, and $(1/\phi)'(1/2,1/2)=1$. The
generalized pseudo-Anosov $f\co S\to S$ induces a homeomorphism
$\widehat{f}\co \csph\to\csph$.

In this section it is shown that the functions~$\phi_n$ converge
to~$\phi$, and hence that the homeomorphisms
$\widehat{f}_n$ converge to~$\widehat{f}$. The key result,
Lemma~\ref{lem:uniform-mod-cont} below, is that
the~$\phi_n$ have a uniform modulus of continuity.

\begin{defn}[$\cI(t)$]
\label{defn:cI}
Let $\cI\co(0,\br) \to (0,\infty)$ be the function defined by
\[\cI(t) := \frac{\ln 2}{12} \int_t^\br \frac{\rmd s}{s \ln \left(
  \frac{8}{s-t} \right) }.\]
\end{defn}

\begin{rmk}
\label{rmk:cI-props}
$\cI$ is a decreasing function, having derivative
\[\cI'(t) = \frac{-\ln2}{12} \int_t^\br \frac{\rmd
  s}{s(s-t)\left(\ln(8/(s-t))\right)^2} < 0.\]
Moreover $\cI(t)\to\infty$ as $t\to 0$.
\end{rmk}

\begin{lem}
\label{lem:uniform-mod-cont}
Define $\brho\co [0,\delta)\to[0,\infty)$ by
\[
\brho(t) :=
\begin{cases}
0, &\text{if } t=0;\\
\max\left(
\dfrac{16R}{\exp\left(2\pi M\cI(2At)\right)},
\kappa t
\right), &\text{if } t>0.
\end{cases}
\]
Then for all~$n\ge 3$, $\brho$ is a modulus of continuity
for~$\phi_n\co P_n\to\csph$, with respect to the Euclidean metric
on~$P_n$ and the spherical metric on~$\csph$.
\end{lem}

Here~$\delta = 1/4608$ and $A = 96$ are given by~(\ref{eq:delta-A}),
$M=1/5$ is given by Definition~\ref{defn:goodnessfn}, $R=R(\bh)$ is
given by Lemma~\ref{lem:R}, and $\kappa=\kappa(|\partial P_n|,
\delta)$ is given by~(\ref{eq:kappa}). Observe that $\brho$ is a
modulus of continuity in the sense of Definition~\ref{defn:modcont}:
it is continuous, positive, and strictly increasing in~$(0,\delta)$ as
the maximum of two functions with these properties; and $\brho(t)\to
0$ as $t\to 0$ by Remark~\ref{rmk:cI-props}.

Two preliminary lemmas are required.

\begin{lem}
\label{lem:Gn-centre-m-and-n}
Let~$n\ge 3$ and $r\in(0,\br)$. Then, in~$G_n$,
\begin{eqnarray*}
m(q_0;r) &\le& 8r\log_2(8/r) \qquad\text{and}\\
n(q_0;r) &\le& 4\log_2(4/r),
\end{eqnarray*}
where~$q_0$ is the central vertex of~$G_n$.
\end{lem}
\begin{proof}
The edges $v_0$, $v_1$, $v_n$, $v_{n+1}$, $h_{n-1}$ and~$h_{n-2}$ of~$G_n$
contribute at most~$8r$ to~$m(q_0;r)$ and at most $4$ to~$n(q_0;r)$. The
remaining edges form an $(n-2)$-od, with edges of lengths
\[\ell_i = |h_i|+|v_{n-1-i}| = \alpha \lambda^i + \beta\lambda^{n-1-i}
\qquad (0\le i < n-2).\] Here $\alpha =
\frac{1}{1+\lambda+\cdots+\lambda^{n-1}} < \frac{1}{\lambda^{n-1}}$
and $\beta =
\frac{1}{1+\lambda+\cdots+\lambda^{n+1}}<\frac{1}{\lambda^{n+1}}$, so
that
\[\ell_i < \frac{1}{\lambda^{n-1-i}} + \frac{1}{\lambda^{i+2}} 
< \frac{1}{(\sqrt2)^{n-1-i}} + \frac{1}{(\sqrt2)^{i+2}} := L_i.\]

Let~$k=2+\lceil 2\log_2(1/r) \rceil \le 3+2\log_2(1/r)$ (where~$\lceil
x\rceil$ denotes the smallest integer which is not less than~$x$), so
that $\frac{1}{(\sqrt2)^k} \le \frac{r}{2}$. Then provided $n-1-i\ge
k$ and $i+2\ge k$ (i.e $k-2\le i\le n-1-k$), $L_i$ satisfies
\[L_i = \frac{1}{(\sqrt2)^{n-1-i}} + \frac{1}{(\sqrt2)^{i+2}} 
\le \frac{2}{(\sqrt2)^k} \le r.\] 
There are at most~$2k-4$ values
of~$i$ which don't satisfy these inequalities, so that
\[n(q_0;r) \le 4 + (2k-4) \le 6 + 4\log_2(1/r) < 4(2 + \log_2(1/r)) =
4\log_2(4/r)\]
as required. Similarly
\[m(q_0;r) \le 8r + (4k-8)r + 2\sum_{i=k-2}^{n-1-k}L_i < 8r+(4k-8)r +
8r < 8r(3+\log_2(1/r)) = 8r\log_2(8/r),\]
since 
\[\sum_{i=k-2}^{n-1-k}L_i < \frac{2}{(\sqrt2)^k}\sum_{i=0}^\infty
\frac{1}{(\sqrt2)^i} < \frac{8}{(\sqrt2)^k} \le 4r.\]
\end{proof}

\begin{lem}
\label{lem:Gn-integral-bound}
For all~$n\ge3$, $q\in G_n$, and $t\in (0,\br)$
\[\int_t^\br \frac{\rmd s}{m(q;s) + s\cdot n(q;s)} \ge \cI(t).\]
\end{lem}

\begin{proof}
Write~$D:=d_{G_n}(q_0,q)$, where $q_0$ is the central vertex
of~$G_n$. Then
\[
m(q;r) \le
\begin{cases}
6r, &\text{if } r\le D;\\
6D+m(q_0;r-D) \le 6D+8(r-D)\log_2\left(\frac{8}{r-D}\right) & \text{if }r>D,
\end{cases}
\]
and
\[
n(q;r) \le
\begin{cases}
3, &\text{if } r\le D;\\
2+n(q_0;r-D) \le 2+4\log_2\left(\frac{4}{r-D}\right) <
4\log_2\left(\frac{8}{r-D}\right) & \text{if }r>D. 
\end{cases}
\]
There are three cases to consider: $0\le D\le t$; $t\le D\le\br$; and
$D\ge \br$.

\noindent\textbf{Case 1: $0\le D\le t$}

Then \[\int_t^\br \frac{\rmd s}{m(q;s)+s\cdot n(q;s)} \ge
\int_t^{\br}\dfrac{\rmd s}{6D+(12s-8D)
  \log_2\left(\frac{8}{s-D}\right)}. \]
Since $\log_2(8/(s-D)) \ge \log_2(8/\br) > 1$, this integral is in
turn bounded below by
\[F(D,t) := \int_t^{\br} \dfrac{\rmd s}{12 s
  \log_2\left(\frac{8}{s-D}\right)}  = \frac{\ln2}{12}\int_t^{\br}
\dfrac{\rmd s}{s\ln\left(\frac{8}{s-D}\right)}.\]
Now
\[\frac{\partial F}{\partial D} = \frac{-\ln2}{12}\int_t^\br
\dfrac{\rmd s}{s(s-D)\left(\ln\left(\frac{8}{s-D}\right)\right)^2}<0\] for $D\in[0,t]$, so
\[\int_t^\br \frac{\rmd s}{m(q;s)+s\cdot n(q;s)} \ge F(t,t) = \cI(t)\]
for all $q\in G_n$ with $d_{G_n}(q,q_0)\le t$ as required.

\noindent\textbf{Case 2: $t\le D\le \br$}

Then  \[\int_t^\br \frac{\rmd s}{m(q;s)+s\cdot n(q;s)} \ge
\int_t^D \frac{\rmd s}{9s} + F(D,D) =
\frac{1}{9}\ln\left(\frac{D}{t}\right) + \frac{\ln 2}{12} \int_D^\br
\dfrac{\rmd s}{s\ln\left(\frac{8}{s-D}\right)}.
\]
The derivative of this lower bound with respect to D is
\begin{eqnarray*}
\frac{1}{9D} - \frac{\ln2}{12} \int_D^\br \dfrac{\rmd
  s}{s(s-D)\left(\ln\left(\frac{8}{s-D}\right)\right)^2} &\ge& 
\frac{1}{D} \left(\frac{1}{9} - \frac{\ln2}{12} \int_D^\br \dfrac{\rmd
  s}{(s-D)\left(\ln\left(\frac{8}{s-D}\right)\right)^2} \right)\\
&=& \frac{1}{D}\left(
\frac{1}{9} - \dfrac{\ln2}{12\ln\left(\frac{8}{\br-D}\right)}
\right)\\
&>&\frac{1}{D}\left(
\frac{1}{9} - \frac{\ln2}{12\ln(192)}
\right) > 0
\end{eqnarray*}
for $D\in[t,\br]$, using $\br = 1/24$. Hence
\[\int_t^\br \frac{\rmd s}{m(q;s)+s\cdot n(q;s)} \ge
\frac{1}{9}\ln(1) + F(t,t) = \cI(t)\]
for all $q\in G_n$ with $t\le d_{G_n}(q,q_0)\le \br$ as required.

\noindent\textbf{Case 3: $ D\ge \br$}

Then 
\[
\int_t^\br \frac{\rmd s}{m(q;s)+s\cdot n(q;s)} \ge \frac{1}{9}\int_t^\br
\frac{\rmd s}{s} 
> \frac{\ln 2}{12}\int_t^\br \dfrac{\rmd
  s}{s\ln\left(\frac{8}{s-t}\right)} = \cI(t)
\]
as required, using $\ln(8/(s-t))> \ln(8/\br)>1$.
\end{proof}

\begin{proof}[Proof of Lemma~\ref{lem:uniform-mod-cont}]
Following the proof of Theorem~\ref{thm:modcont}, it is only required
to show that the function $\rho_q(t)$ of~(\ref{eq:modcont}) satisfies
$\rho_q(t) \le \brho(t)/2$ for all~$n$, all~$q\in Q_n(\delta)$, and all
$t\in (0,\delta)$. (As in that proof, the factor~2 arises from the
translation between the Euclidean and spherical metrics.)

So let $q\in Q_n(\delta)$ and $t\in (0,\delta)$. If $h_q\le t$ then
\begin{eqnarray*}
\rho_q(t) &=& \dfrac{8R}{\exp\left(2\pi\displaystyle{
\int_{A(t+h_q)}^\br \iota(\psi(q);s)\rmd s}
\right)}\\
&\le& \dfrac{8R}{\exp\left(2\pi\displaystyle{
\int_{2At}^\br \iota(\psi(q);s)\rmd s}
\right)}\\
&\le& \dfrac{8R}{\exp(2\pi M\cI(2At))}
\end{eqnarray*}
by Lemma~\ref{lem:Gn-integral-bound}, using $\iota(q;s) =
\frac{M}{m(q;s) + s\cdot n(q;s)}$.

On the other hand, if $h_q>t$ then an analogous argument gives
\[\rho_q(t) \le \dfrac{8Rt}{h_q\exp(2\pi M\cI(2Ah_q))}.\]
However the function $x\mapsto x\exp(2\pi M\cI(2Ax))$ is increasing on
$(t,\delta]$, so
\[\rho_q(t) \le \dfrac{8Rt}{t\exp(2\pi M\cI(2At))}\]
as required.
\end{proof}

\begin{lem}
\label{lem:convergence-phi-n}
For each~$N\ge 3$ the sequence of functions $(\phi_n\vert_{X_N}\co
X_N\to\csph)_{n\ge N}$ converges uniformly to $\phi\vert_{X_N}\co
X_N\to\csph$, where $X_N = \bigcap_{n\ge N}P_n$. In particular, if
$x_n\to x$ is a convergent sequence in~$\Sigma$ with $x_n\in P_n$ for
all~$n$, then $\phi_n(x_n)\to\phi(x)$.
\end{lem}

\begin{rmk}
The sequence of polygons~$(P_n)$ is not increasing, since the lengths
$|V_1^n| = \lambda_n^n/(\lambda_n^{n+1}+1)$ of the sides $V_1^n$,
which are contained in~$\partial\Sigma$, decrease with~$n$. However,
the sequence~$(X_n)$ is increasing, and converges Hausdorff
to~$\Sigma$ by Lemma~\ref{lem:topological-convergence}a).
\end{rmk}

\begin{proof}
Since the sequence~$(\phi_n)$ is equicontinuous and the sequence of
domains $(P_n)$ converges Hausdorff to~$\Sigma$, a variant of the
Arzel\`a-Ascoli theorem (which is proved identically to the standard
version) shows that there is a subsequence~$(\phi_{n_i})$ and a
continuous function $\phi_\infty\co\Sigma\to\csph$ such that, for
each~$N\ge 3$, the sequence $(\phi_{n_i}\vert_{X_N})_{n_i \ge N}$
converges uniformly to $\phi_\infty\vert_{X_N}$. It will be shown that
$\phi_\infty=\phi$ for any such subsequence, which will establish the
result. 

Now $\phi_\infty(1/2,1/2)=\infty$ and $\phi_\infty(1,0)=0$, since
$\phi_n(1/2,1/2)=\infty$ and $\phi_n(1,0)=0$ for all~$n$. Moreover,
because the functions~$\phi_n$ are univalent on~$\Int(P_n)$ and any
open subset of $\Int(\Sigma)$ whose boundary is disjoint from
$\partial\Sigma$ is contained in~$X_n$ for sufficiently large~$n$, the
function~$\phi_\infty$ is univalent in $\Int(\Sigma)$, and satisfies
$(1/\phi_\infty)'(1/2,1/2)=1$. Because $\phi_\infty$ restricted to
$\Int(\Sigma)$ is open and injective, there do not exist points
$w\in\partial\Sigma$ and $z\in\Int(\Sigma)$ with
$\phi_\infty(w)=\phi_\infty(z)$.

By Lemma~\ref{lem:topological-convergence}b), if
$w,z\in\partial\Sigma$ with $w\sim_{\cP}z$, then
$\phi_\infty(w)=\phi_\infty(z)$ ($(w,z)$ is arbitrarily closely
approximated by points $(w_n,z_n)\in\sim_{\cP_n}$, which therefore
satisfy $\phi_n(w_n)=\phi_n(z_n)$). There is therefore a
continuous function $u_\infty\co S\to\csph$ with $\phi_\infty =
u_\infty \circ \pi$, which restricts to a conformal homeomorphism
\[u_\infty|_{S\setminus G}\co S\setminus G \to \csph \setminus
u_\infty(G).\]

Let~$z\in\partial\Sigma$ be any point in the interior of a paired
segment~$\alpha$ of~$\cP$, let~$\veps>0$ be the distance from~$z$ to
the nearest endpoint of~$\alpha$, and let~$z'$ be the point of the
paired segment~$\alpha'$ which is paired
with~$z$. Let~$\Sigma^z\subset\C$ be obtained from~$\Sigma$ by
excising $B_\Sigma(z'; \veps/2)$ and gluing this half-disk
onto~$\Sigma$ near~$z$ according to the
pairing~$\ssegpair{\alpha}$. Then $\phi_\infty$ induces a natural
function $\phi^z_\infty\co \Sigma^z\to\csph$.

Choose sequences $z_n\to z$ and $z_n'\to z'$ where $z_n,z_n'\in P_n$
and $z_n\sim_{\cP_n} z_n'$. Then for sufficiently large~$n$ an
analogous construction can be carried out to obtain functions
$\phi_n^z\co P_n^z\to P_n^z$ which are univalent on $\Int(P_n^z)$. It
follows that $\phi^z_\infty$ is univalent on $\Int(\Sigma^z)$, and in
particular that there is no point $w\in\partial\Sigma$ with
$\phi_\infty(w)=\phi_\infty(z)$ other than $z'$.

Now if $w,z\in G$ with $w\not=z$ and $u_\infty(w)=u_\infty(z)$, then
every point except possibly one of the interval $[w,z]_G$ is
identified with another such point by~$u_\infty$, as otherwise
$u_\infty(G)$ would contain a simple closed curve, contradicting the
fact that its complement is connected; in particular, $u_\infty$
identifies some point in the interior of a paired segment with another
point which it is not~$\sim_{\cP}$-equivalent to. Since this has been
shown to be impossible, $u_\infty$ is injective on~$G$.

It follows that $u_\infty\co S\to\csph$ is a homeomorphism: since it
is conformal on $S\setminus G$, and $G$ has finite $1$-dimensional
Hausdorff measure by Lemma~\ref{lem:metricstructureofG}c), $u_\infty$
is conformal by Theorem~\ref{thm:removability2}. Because~$u_\infty$ is
normalized in the same way as~$u$ the two are equal, and hence
\[\phi_\infty = u_\infty\circ\pi = u\circ\pi = \phi\] as required.

\end{proof}

\begin{thm}
\label{thm:convergence}
The pseudo-Anosov homeomorphisms $\widehat{f}_n\co \csph \to \csph$
converge pointwise to the generalized pseudo-Anosov
$\widehat{f}\co\csph\to\csph$.
\end{thm}
\begin{proof}
Let~$z\in\csph$ and choose, for each~$n$, a point $x_n\in P_n$ with
$\phi_n(x_n)=z$: thus \mbox{$\widehat{f}_n(z) = \phi_n(F_n(x_n))$.} It will
be shown that if $\phi_{n_i}(F_{n_i}(x_{n_i}))\to w$ is any convergent
subsequence then $w=\widehat{f}(z)$, which will establish the result.

Given such a subsequence, assume without loss of generality that
$x_{n_i}\to x\in\Sigma$, and (using
Lemma~\ref{lem:topological-convergence}c)) that $F_{n_i}(x_{n_i})$
converges to some $y\in\Sigma$ with $\phi(y)=\phi(F(x))$. Since
$\phi_{n_i}(x_{n_i})=z$ for all~$i$, it follows from
Lemma~\ref{lem:convergence-phi-n} that $\phi(x)=z$. By the same lemma,
\[\phi_{n_i}(F_{n_i}(x_{n_i})) \to \phi(y) = \phi(F(x)) =
\widehat{f}(\phi(x)) = \widehat{f}(z)\]
as required.
\end{proof}

\bibliography{bib_origami}

\end{document}

%% file: defs_origami.tex
\usepackage{xspace}

\newtheorem{thm}{Theorem}
\newtheorem{lem}[thm]{Lemma}
\newtheorem{cor}[thm]{Corollary}

\theoremstyle{definition}

\newtheorem{defn}[thm]{Definition}
\newtheorem{defns}[thm]{Definitions}

\newtheorem{example}[thm]{Example}

\theoremstyle{remark}

\newtheorem{rmk}[thm]{Remark}

\newtheorem{rmks}[thm]{Remarks}

\newtheorem{notation}[thm]{Notation}

\newcommand{\hP}{{\hat{P}}}

\newcommand{\hrho}{{\hat{\rho}}}

\newcommand{\hcA}{{\hat{\mathcal{A}}}}
\newcommand{\hcC}{{\hat{\mathcal{C}}}}
\newcommand{\hcP}{{\hat{\mathcal{P}}}}

\newcommand{\talpha}{{\tilde{\alpha}}}

\newcommand{\tgamma}{{\tilde{\gamma}}}

\newcommand{\teta}{{\tilde{\eta}}}

\newcommand{\tLambda}{{\tilde{\Lambda}}}
\newcommand{\tDelta}{{\tilde{\Delta}}}

\newcommand{\tpsi}{{\tilde{\psi}}}

\newcommand{\ttheta}{{\tilde{\theta}}}

\newcommand{\tP}{{\tilde{P}}}
\newcommand{\tF}{{\tilde{F}}}

\newcommand{\tL}{{\tilde{L}}}
\newcommand{\tQ}{{\widetilde{Q}}}

\newcommand{\tT}{{\tilde{T}}}

\newcommand{\te}{{\tilde{e}}}

\newcommand{\tp}{{\tilde{p}}}
\newcommand{\tq}{{\tilde{q}}}
\newcommand{\tv}{{\tilde{v}}}

\newcommand{\tx}{{\tilde{x}}}
\newcommand{\ty}{{\tilde{y}}}
\newcommand{\tz}{{\tilde{z}}}

\newcommand{\br}{{\bar r}}
\newcommand{\bh}{{\bar h}}

\newcommand{\bB}{{\overline B}}

\newcommand{\bDD}{{\overline\D}}
\newcommand{\bcV}{{\overline\cV}}
\newcommand{\bU}{{\overline U}}

\newcommand{\brho}{{\bar{\rho}}}

\newcommand{\sone}{S^1}

\newcommand{\csph}{{\widehat{\C}}}

\newcommand{\plane}{{\mathbb{R}}^2}

\newcommand{\lseg}{[\![}
\newcommand{\rseg}{]\!]}

\newcommand{\I}{^{-1}}
\newcommand{\co}{\colon}

\newcommand{\llangle}{\left\langle}
\newcommand{\lang}{\left\langle}
\newcommand{\rrangle}{\right\rangle}
\newcommand{\rang}{\right\rangle}

\newcommand{\sbs}{\subset}

\newcommand{\ol}[1]{\overline{#1}}
\newcommand{\interior}[1]{{\stackrel{\circ}{#1}}}
\newcommand{\opna}[1]{\operatorname{#1}}

\newcommand{\segpair}[2]{\left\langle #1,#2\right\rangle}
\newcommand{\ssegpair}[1]{\segpair{#1}{#1'}}

\newcommand{\rmd}{\,\mathrm{d}}
\newcommand{\rmm}{\mathrm{m}}
\newcommand{\rmL}{{\mathrm{L}}}
\newcommand{\rmR}{{\mathrm{R}}}

\newcommand{\Int}{\operatorname{Int}}

\newcommand{\Area}{\operatorname{Area}}
\newcommand{\Ann}{\opna{Ann}}

\newcommand{\im}{\opna{im}}

\renewcommand{\mod}{\opna{mod}}

\newcommand{\Cone}{\operatorname{Cone}}

\newcommand{\diam}{\operatorname{diam}}

\newcommand{\Planar}{{\opna{Planar}}}

\newcommand{\thor}{{\widetilde{\opna{h}}}{\opna{or}}}
\newcommand{\tHor}{{\widetilde{\opna{H}}}{\opna{or}}}
\newcommand{\tver}{{\widetilde{\opna{v}}}{\opna{er}}}
\newcommand{\tVer}{{\widetilde{\opna{V}}}{\opna{er}}}
\newcommand{\hor}{\opna{hor}}
\newcommand{\Hor}{\opna{Hor}}
\newcommand{\ver}{\opna{ver}}
\newcommand{\Ver}{\opna{Ver}}

\newcommand{\Gr}{\opna{Gr}}

\newcommand{\veps}{\varepsilon}

\newcommand{\vphi}{\varphi}

\newcommand{\cA}{{\mathcal{A}}}

\newcommand{\cF}{{\mathcal{F}}}

\newcommand{\cR}{{\mathcal{R}}}
\newcommand{\cV}{{\mathcal{V}}}
\newcommand{\cG}{{\mathcal{G}}}
\newcommand{\cC}{{\mathcal{C}}}
\newcommand{\cP}{{\mathcal{P}}}

\newcommand{\cI}{{\mathcal{I}}}
\newcommand{\cL}{{\mathcal{L}}}
\newcommand{\cQ}{{\mathcal{Q}}}

\newcommand{\R}{\mathbb{R}}

\newcommand{\Q}{\mathbb{Q}}
\newcommand{\N}{\mathbb{N}}
\newcommand{\C}{\mathbb{C}}

\newcommand{\D}{\mathbb{D}}

\newcommand{\pichere}[2]
{\begin{center}\includegraphics[width=#1\textwidth]{#2}\end{center}}

\newcommand{\lab}[3]{\psfrag{#1}[#3]{$\scriptstyle{#2}$}}